\documentclass[final]{siamart220329}

\usepackage[noend]{algpseudocode}
\usepackage{amsmath,amssymb,amsopn}
\usepackage{bm}
\usepackage{enumitem}
\usepackage{mathtools}
\usepackage{tikz-cd}
\usepackage{listings}

\lstMakeShortInline|
\newcommand{\stext}[1]{{\text{\small{#1}}}}

\newsiamremark{remark}{Remark}
\newsiamthm{assump}{Assumption}

\definecolor{highyellow}{RGB}{255,251,194}

\usetikzlibrary{shapes,decorations}
\tikzstyle{highbox} = [draw=red, fill=highyellow, very thick,
    rectangle, rounded corners, inner sep=12pt, inner ysep=12pt]
\tikzstyle{hightitle} =[fill=red, text=white]

\DeclareMathOperator{\Id}{Id}
\DeclareMathOperator{\Prob}{Prob}
\DeclareMathOperator{\ran}{ran}
\DeclareMathOperator{\spn}{span}
\DeclareMathOperator{\tr}{tr}
\DeclareMathOperator{\binlength}{length}
\DeclareMathOperator*{\argmax}{argmax}
\DeclareMathOperator{\NRMSE}{NRMSE}
\DeclareMathOperator{\AC}{AC}
\DeclareMathOperator{\mnth}{month}
\DeclareMathOperator{\mean}{\mathbb E}
\DeclareMathOperator{\var}{var}
\newcommand{\degr}{$^\circ$}
\newcommand{\ssa}{S_*(\mathfrak A)} 
\newcommand{\ssb}{S_*(\mathfrak B)} 
\newcommand{\scb}{S_C(\mathfrak B)} 
\newcommand{\ssbl}{S_*(\mathfrak B_L)} 
\newcommand{\ssw}{S_*(\mathfrak W)} 
\newcommand{\lf}{\mathsf F} 
\newcommand{\lh}{\mathsf h} 
\newcommand{\lx}{\mathsf x} 
\newcommand{\ly}{\mathsf y} 
\newcommand{\La}{\circled{\textcolor{blue}{\small{A}}}}
\newcommand{\Lq}{\circled{\textcolor{red}{\small{Q}}}}
\newcommand{\Lm}{\circled{\textcolor{violet}{\small{M}}}}
\newcommand{\Inputs}{\textbf{Inputs}}
\newcommand{\Outputs}{\textbf{Outputs}}
\newcommand{\Steps}{\textbf{Steps}}
\newcommand{\Training}{\textbf{Training phase}}
\newcommand{\Prediction}{\textbf{Prediction phase}}
\renewcommand{\Return}{\textbf{Return:} }
\renewcommand{\Require}{\emph{Require:} }
\newcommand{\tbm}[1]{\tilde{\bm{#1}}}
\newcommand{\hbm}[1]{\hat{\bm{#1}}}

\makeatletter
\usetikzlibrary{calc}
\DeclareRobustCommand{\circled}[1]{\tikz[baseline=(char.base)]{
    \node[shape=circle, draw, inner sep=0pt, 
        minimum height={\f@size*1},] (char) {\vphantom{WAH1g}#1};}}
\makeatother

\headers{Data Assimilation in Operator Algebras}{D. Freeman, D. Giannakis, B. Mintz, A. Ourmazd, and J. Slawinska}

\title{Data Assimilation in Operator Algebras\thanks{\funding{D.G.\ acknowledges support from the US National Science Foundation under grants 1842538 and DMS-1854383, the US Office of Naval Research under MURI grant N00014-19-1-242, and the US Department of Defense, Basic Research Office under Vannevar Bush Faculty Fellowship grant N00014-21-1-2946. D.C.F.\ is supported as a PhD student under the last grant. A.O.\ was supported by the US Department of Energy, Office of Science, Basic Energy Sciences under award DE-SC0002164 (underlying dynamical techniques), and by the US National Science Foundation under awards STC-1231306 (underlying data analytical techniques) and DBI-2029533 (underlying analytical models). J.S.\ acknowledges support from NSF EAGER grant 1551489.}}}

\author{David C. Freeman\thanks{Department of Mathematics, Dartmouth College, Hanover, NH 03755, USA (\email{david.c.freeman.gr@dartmouth.edu}, \email{dimitrios.giannakis@dartmouth.edu}, \email{brian.a.mintz.gr@dartmouth.edu}, \email{joanna.m.slawinska@dartmouth.edu}).} \and Dimitrios Giannakis\footnotemark[2] \and Brian Mintz\footnotemark[2] \and Abbas Ourmazd\thanks{Department of Physics, University of Wisconsin-Milwaukee, Milwaukee, 53211, USA (\email{ourmazd@uwm.edu})} \and Joanna Slawinska\footnotemark[2]}

\begin{document}

\maketitle

\begin{abstract}
    We develop an algebraic framework for sequential data assimilation of partially observed dynamical systems. In this framework, Bayesian data assimilation is embedded in a non-abelian operator algebra, which provides a representation of observables by multiplication operators and probability densities by density operators (quantum states). In the algebraic approach, the forecast step of data assimilation is represented by a quantum operation induced by the Koopman operator of the dynamical system. Moreover, the analysis step is described by a quantum effect, which generalizes the Bayesian observational update rule. Projecting this formulation to finite-dimensional matrix algebras leads to new computational data assimilation schemes that are (i) automatically positivity-preserving; and (ii) amenable to consistent data-driven approximation using kernel methods for machine learning. Moreover, these methods are natural candidates for implementation on quantum computers. Applications to data assimilation of the Lorenz 96 multiscale system and the El Ni\~no Southern Oscillation in a climate model show promising results in terms of forecast skill and uncertainty quantification. 
\end{abstract}

\section{Introduction}

Since its inception in weather forecasting~\cite{Cressman59} and object tracking problems~\cite{Kalman60}, sequential data assimilation, also known as filtering, has evolved into an indispensable tool in forecasting and uncertainty quantification of dynamical systems \cite{MajdaHarlim12,LawEtAl15}. In its essence, data assimilation is a Bayesian inference theory: Knowledge about the state of the system at time $t$ is described by a probability distribution $p_t$. The system dynamics acts on probability distributions, carrying along $p_t $ to a time-dependent family of distributions $p_{t,\tau}$, which can be used to forecast observables of the system at time $t + \tau$, $\tau \geq 0$. When an observation is made, at time $t+\Delta t$, the forecast distribution $ p_{t,\Delta t}$ is updated in an analysis step using Bayes' theorem to a posterior distribution $p_{t+\Delta t}$, and the cycle is repeated. 

In real-world applications, the Bayesian theoretical ``gold standard'' is seldom feasible to employ due to a variety of challenges, including high-dimensional nonlinear dynamics, nonlinear observation modalities, and model error. Weather and climate dynamics \cite{Kalnay03} represent a classical application domain where these challenges are prevalent due to the extremely large number of active degrees of freedom (which necessitates making dynamical approximations such as subgrid-scale parameterization) and nonlinear equations of motion and observation functions (which prevent direct application of Bayes' theorem). Addressing these issues has stimulated the creation of a broad range of data assimilation techniques, including variational \cite{Bannister16}, ensemble \cite{KarspeckEtAl18}, and particle \cite{VanLeeuwenEtAl19} methods. 

In this paper, we examine Bayesian data assimilation and its representation through finite-dimensional computational methods from a new algebraic perspective. Our formulation employs different levels of description, depicted schematically in \cref{figSchematic}. We begin by assigning to a measure-preserving dynamical flow $\Phi^t : X \to X$, $ t \in \mathbb R$, an \emph{algebra of observables} (complex-valued functions of the state) $\mathfrak A = L^\infty(X,\mu) $, where $X$ is the state space and $\mu$ the invariant measure. This algebra is a commutative, or abelian, von Neumann algebra \cite{Takesaki01} under pointwise function multiplication. The \emph{state space} of $\mathfrak A$, denoted as $S(\mathfrak A)$, is the set of continuous linear functionals $\omega: \mathfrak A \to \mathbb C$, satisfying the positivity condition $\omega( f^* f ) \geq 0 $ for all $f \in \mathfrak A$ and the normalization condition $\omega \bm 1 = 1$. Here, $^*$ denotes the complex conjugation of functions, and $\bm 1 $ is the unit of $\mathfrak A$,  $\bm 1(x) = 1$ for all $x\in X$. Every probability density $ p \in L^1(X,\mu) \equiv \mathfrak A_* $ induces a state $\omega_p \in S(\mathfrak A)$ that acts on $\mathfrak A$ as an expectation functional, $\omega_p f = \int_X f p \, d\mu$. Such states $\omega_p$ constitute the set of \emph{normal states} of $\mathfrak A$, denoted as $S_*(\mathfrak A)$.

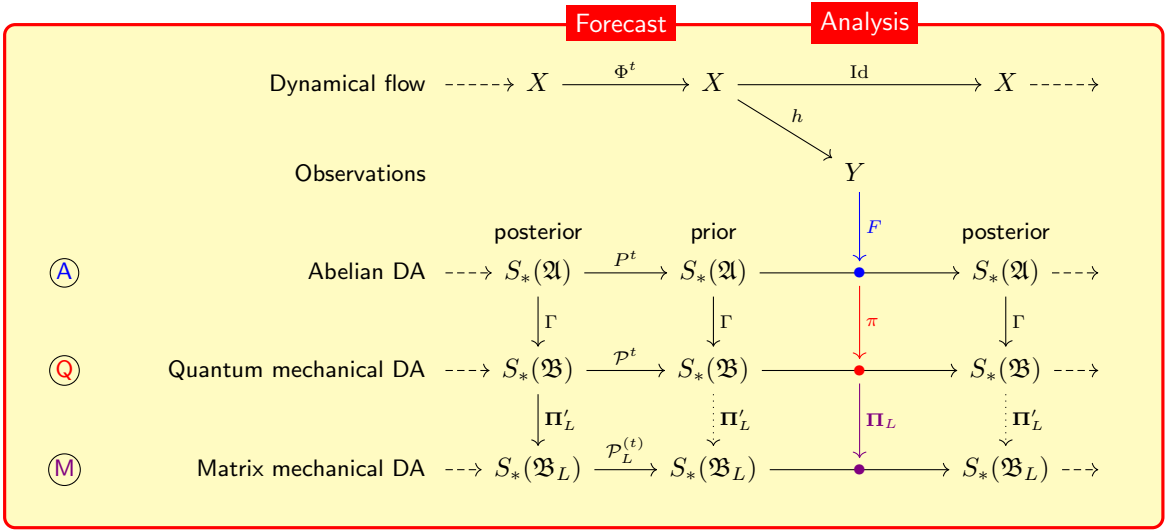
\begin{figure}
    \centering
    \sffamily
    \begin{tikzpicture}
        \node[highbox] (box){%
            \begin{tikzcd}[background color=highyellow]
                &\makebox[\widthof{\stext{Quantum mechanical DA}}][r]{ \stext{Dynamical flow}} \; \ar[r, dashed] & [-1em] X \ar[r,"\Phi^t"] & X \ar[rr,"\Id"] \ar[dr, "h"]  && X \ar[r, dashed] & [-1em] \  \\
                &\makebox[\widthof{\stext{Quantum mechanical DA}}][r]{\stext{Observations}} \; & & & Y \ \\[-1em]
                &&\stext{posterior} & \stext{prior} & & \stext{posterior} \\[-2em] 
                \La &\makebox[\widthof{\stext{Quantum mechanical DA}}][r]{\stext{Abelian DA}} \; \ar[r, dashed] & [-1em] \ssa \ar[r,"P^t"] \ar[d, "\Gamma"] & \ssa \ar[rr, "\bullet"{name=E, marking, blue}] \ar[d, "\Gamma"] && \ssa \ar[r, dashed] \ar[d,"\Gamma"] & [-1em] \ \\ 
                \Lq & \stext{Quantum mechanical DA} \; \ar[r, dashed] & [-1em] \ssb \ar[r,"\mathcal P^t"] \ar[d, "\bm\Pi_L'"] & \ssb \ar[rr,"\bullet"{name=F, marking, red}] \ar[d, "\bm\Pi_L'", dotted] && \ssb \ar[r, dashed] \ar[d, "\bm\Pi_L'", dotted] & [-1em] \ \\
                \Lm &\makebox[\widthof{\stext{Quantum mechanical DA}}][r]{\stext{Matrix mechanical DA}} \; \ar[r, dashed] & [-1em] \ssbl \ar[r,"\mathcal P^{(t)}_L"] & \ssbl \ar[rr,"\bullet"{name=Fl, marking, violet}, crossing over] && \ssbl \ar[r, dashed] & [-1em] \ 
                \ar[from=2-5, to=E, blue, "F"] 
                \ar[from=E, to=F, red, "\pi"] 
                \ar[from=F, to=Fl, violet, "\bm\Pi_L"] 
            \end{tikzcd}
        };
        \node[hightitle,left=-35] at (box.north){Forecast};
        \node[hightitle, left=93] at (box.north east){Analysis};
    \end{tikzpicture}
    \caption{Schematic representation of the abelian and non-abelian formulations of sequential data assimilation (DA), showing a forecast--analysis cycle. The top row of the diagram shows the dynamical flow $\Phi^t : X \to X $. The second row shows the observation map $ h : X \to Y$ used to update the state of the DA system in the analysis step. The rows labeled \La, \Lq, and \Lm\ show the abelian, infinite-dimensional non-abelian (quantum mechanical), and finite-dimensional non-abelian (matrix mechanical) DA systems, respectively.  In \La, the forecast step is carried out by the transfer operator $P^t : \ssa \to \ssa $ acting on states of the abelian algebra $\mathfrak A$. The analysis step (blue dot) is represented by an effect-valued map $ F : Y \to \mathfrak A$ that updates the state given observations in $Y$. In \Lq, the forecast step is carried out by the transfer operator $\mathcal P^t : \ssb \to \ssb$ acting on states of the non-abelian operator algebra $\mathfrak B$. The analysis step (red dot) is carried out by an effect $ \mathcal F : Y \to \mathfrak B$ given by the composition of $F$ with the regular representation $\pi$ of $\mathfrak A$ into $\mathfrak B$ (red arrow). The state space $\ssa$ is embedded into $\ssb$ by means of a map $\Gamma$ which is compatible with both forecast and analysis; see~\eqref{eqPiGamma} and~\eqref{eqBayesCompat}. This compatibility is represented by the commutative loops between \La\ and \Lq\ having $\Gamma$ as a vertical arrow. To arrive at the matrix mechanical DA, \Lm, we project $\mathfrak B$ into an $L^2$-dimensional operator algebra $\mathfrak B_L$ using a positivity-preserving projection $ \bm\Pi_L$. The composition of this projection with $\mathcal F$ leads to an effect $\mathcal F_L : Y \to \mathfrak B_L$ employed in the analysis step (purple arrow and dot). Moreover, $\bm\Pi_L$ induces a state space projection $\bm\Pi'_L : \ssb \to \ssbl $ and a projected transfer operator $\mathcal P_L^{(t)} : \ssbl \to \ssbl$ employed in the forecast step. Vertical dotted arrows indicate asymptotically commutative relationships that hold as $L \to \infty$.}
    \label{figSchematic}
\end{figure}

The next level of our framework, labeled \Lq\ in \cref{figSchematic}, generalizes data assimilation by embedding it in a \emph{non-abelian} operator algebra. Operator algebras form the mathematical backbone of quantum mechanics \cite{Emch09}---one of the most successful theories in physics. Quantum information theory and quantum probability provide a unified mathematical framework to characterize the properties of information transfer in both abelian and non-abelian systems through maps (quantum operations) acting on elements of the algebra and the corresponding states \cite{AlberEtAl01,Holevo01,Wilde13}. Our non-abelian formulation \Lq\ is based on the von Neumann algebra $\mathfrak B \equiv B(H)$ of bounded operators on the Hilbert space $H = L^2(X,\mu)$, equipped with operator composition as the algebraic product. The state space $S(\mathfrak B)$ is defined as the space of continuous, positive, normalized functionals analogously to $S(\mathfrak A)$; that is, every state $\omega \in S(\mathfrak B)$ satisfies $\omega(A^*A) \geq 0 $ and $\omega I = 1$, where $^*$ denotes the operator adjoint on $B(H)$ and $I$ is the identity operator. Analogously to $\mathfrak A$, $S(\mathfrak B)$ has a subset of normal states, $\ssb$, induced in this case by trace-class operators. Specifically, letting $ \mathfrak B_* \equiv B_1(H) $ denote the space of trace-class operators in $B(H)$, every positive operator $\rho \in \mathfrak B_*$ of unit trace induces a state $\omega_\rho\in \ssb$ such that $\omega_\rho A = \tr(\rho A )$. Such operators $\rho$ are called \emph{density operators}, and can be thought of as non-abelian analogs of probability densities $p \in \mathfrak A_*$. As we will see below, analogs of the transfer operator and Bayesian update described above for $\ssa$ are naturally defined for $\ssb$.

To arrive at a practical data assimilation algorithm, we project (discretize) the infinite-dimensional system on $\mathfrak B$ to a system on a finite-dimensional subalgebra $\mathfrak B_L \subset \mathfrak B$, which is concretely represented by an $L \times L$ matrix algebra (see \Lm\ in \cref{figSchematic}). We show that this approach naturally leads to computational techniques which are well-suited for assimilation of high-dimensional observables, while enjoying structure-preservation properties that cannot be obtained from finite-dimensional projections of abelian formulations. Moreover, by virtue of being rooted entirely in linear operator theory, these methods are amenable to a consistent data-driven formulation using kernel methods for machine learning.

\subsection{Previous work} Recently, an operator-theoretic framework for data assimilation, called quantum mechanical data assimilation (QMDA) \cite{Giannakis19b}, was developed using ideas from Koopman operator theory  \cite{Koopman31,EisnerEtAl15} in conjunction with the Dirac-von Neumann axioms of quantum dynamics and measurement \cite{Takhtajan08}. In QMDA, the state of the data assimilation system is a density operator $\rho_t$ acting on $ H=L^2(X,\mu)$ (rather than a probability density $p_t \in L^1(X,\mu)$ as in classical data assimilation), and the assimilated observables are multiplication operators in $B(H)$ (rather than functions in $L^\infty(X,\mu)$). Between observations, $\rho_t$ evolves under the induced action of the transfer operator, and the forecast distribution of observables is obtained as a quantum mechanical expectation with respect to $\rho_t$. Upon making an observation, the density operator $\rho_t$ is updated projectively as a von Neumann measurement, which is the quantum mechanical analog of Bayes' rule. Using kernel methods for Koopman operator approximation \cite{GiannakisEtAl15,Giannakis19,DasEtAl21}, QMDA has a data-driven formulation, which was shown to perform well in low-dimensional applications.

\subsection{\label{secContrib}Contributions} We provide a general algebraic framework that encompasses classical data assimilation and QMDA as particular instances (abelian and non-abelian, respectively). The principal distinguishing aspects of this framework are as follows. 

1. \emph{Dynamical consistency}. We employ a dynamically consistent embedding of abelian data assimilation \La\ into the non-abelian framework \Lq. As in ref.~\cite{Giannakis19b}, observables in $\mathfrak A$ are mapped into multiplication operators in $\mathfrak B$, but here we also employ an embedding $\Gamma : \ssa \to \ssb$ that is compatible with the transfer operators on $\ssa$ and $\ssb$ (see the commutative loops between \La\ and \Lq\ in \cref{figSchematic}). This allows us to study QMDA in relation to the underlying classical theory, and establish the consistency between the two approaches.        

2. \emph{Effect system}. In both the abelian and non-abelian settings, the analysis step given observations in a space $Y$ acquired through an observation map $ h : X \to Y$  is carried out using \emph{quantum effects} (loosely speaking, algebra-valued logical predicates) \cite{AlberEtAl01}. In the abelian case, the effect $F : Y \to \mathfrak A$ is induced by a kernel feature map. In the non-abelian setting, $F$ is promoted to an operator-valued feature map $\mathcal F : Y \to \mathfrak B$; see the column in the schematic of \cref{figSchematic} labeled ``Analysis''. Our use of feature maps enables assimilation of data of arbitrarily large dimension,  overcoming an important limitation of the original QMDA scheme \cite{Giannakis19b} (which becomes prohibitively expensive for high-dimensional observation maps). 

3. \emph{Positivity-preserving discretization}. The discretization procedure leading to the finite-dimensional scheme \Lm\  has the important property that positive elements of $\mathfrak B$ are mapped into positive elements of $\mathfrak B_L$. Moreover, the transfer operator on $\ssb$ is mapped into a completely positive, trace-non-increasing map, so the finite-dimensional data assimilation system is a quantum operation. We call this system ``matrix mechanical''. By virtue of these properties, the matrix mechanical system preserves the sign of sign-definite observables of the original system. Relevant examples include positive physical quantities such as mass and temperature, but also statistical quantities such as probability density which are useful for uncertainty quantification. We emphasize that the approach of \emph{first} embedding classical data assimilation in $\mathfrak A$ to the non-abelian operator setting of $\mathfrak B$ and \emph{then} projecting to the finite-dimensional system on $\mathfrak B_L$ is important in ensuring positivity preservation.

4. \emph{Data-driven formulation}. The matrix mechanical system \Lm\ admits a data-driven approximation in which all operators are represented in a kernel eigenbasis learned from time-ordered training data, without requiring a priori knowledge of the equations of motion. In the limit of large training data, the predictions made by the data-driven data assimilation system converge to those made by the system on $\mathfrak B_L$, which in turn converge to the infinite-dimensional system \Lq\ on $\mathfrak B$  as $L\to\infty$.          

5. \emph{Route to quantum computing}. The matrix mechanical scheme is well-suited for implementation on a quantum computer. Previous work \cite{GiannakisEtAl22} has shown that Koopman operators of measure-preserving ergodic dynamical systems can be approximated using shallow quantum circuits, offering an exponential computational advantage over classical deterministic algorithms for Koopman operator approximation. Our approach provides a novel route to quantum algorithms that sequentially alternate between unitary evolution and projective measurement to perform inference and prediction of classical dynamics with quantum computers. 

\section{Embedding data-assimilation in operator algebras}

Consider a dynamical flow $\Phi^t : X \to X $, $t \in \mathbb R$, on a completely metrizable, separable space $X$ with an ergodic, invariant, Borel probability measure $\mu$. The flow induces Koopman operators $U^t : f \mapsto f \circ \Phi^t$, which are isomorphisms of the $L^p(X,\mu)$ spaces with $ p \in [1,\infty]$. The flow also induces transfer operators $P^t : L^p(X,\mu)^* \to L^p(X,\mu)^*$ on the dual spaces $L^p(X,\mu)^*$ of $L^p(X,\mu)$, given by the adjoint of the Koopman operator, $P^t \gamma = \gamma \circ U^t$. Under the canonical identification of $L^p(X,\mu)^*$, $p \in [1,\infty)$, with finite, complex Borel measures $\gamma$ with densities in $ L^q(X,\mu) $, $\frac{1}{p}+\frac{1}{q}=1$, the transfer operator is identified with the inverse Koopman operator; that is, for $\gamma \in L^p(X,\mu)^*$ with density $\varrho= \frac{d\gamma}{d\mu} \in L^q(X,\mu)$, $\gamma_t := P^t \gamma $ has density $ \varrho_t = \frac{d\gamma_t}{d\mu} \in L^q(X,\mu)$ with $\varrho_t = U^{-t} \varrho$. In what follows, $\lVert f \rVert_{L^p(X,\mu)} = (\int_X \lvert f \rvert^p \, d \mu)^{1/p}$ and $\lVert f \rVert_{L^\infty(X,\mu)} = \lim_{p\to\infty} \lVert f \rVert_{L^p(X,\mu)}$ denote the standard $L^p(X,\mu)$ norms for $p \in [1, \infty)$ and $p=\infty$, respectively. 

Among the $L^p(X,\mu)$ spaces, $H := L^2(X,\mu)$ is a Hilbert space and $\mathfrak A := L^\infty(X,\mu)$ is an abelian von Neumann algebra with respect to function multiplication and complex conjugation. In particular, for any two elements $f,g\in \mathfrak A$ we have 
\begin{equation}
    \label{eqAAlg}
    \lVert f g \rVert_{\mathfrak A} \leq \lVert f \rVert_{\mathfrak A} \lVert g \rVert_{\mathfrak A}, \quad \lVert f^* f \rVert_{\mathfrak A} = \lVert f \rVert^2_{\mathfrak A}, 
\end{equation}
making $\mathfrak A$ a $C^*$-algebra, and in addition $\mathfrak A$ has a predual $\mathfrak A_* := L^1(X,\mu)$ (i.e., a Banach space for which $\mathfrak A$ is the dual), making it a von Neumann algebra. We let $ \langle f, g \rangle = \int_X f^* g \, d\mu$ denote the inner product on $H$. On $H$, the Koopman operator $U^t$ is unitary, $U^{t*} = U^{-t}$. 

\subsection{Embedding observables}

Let $ \mathfrak B := B(H)$ be the space of bounded operators on $H$, equipped with the operator norm, $ \lVert A \rVert_{\mathfrak B} = \sup_{f \in H} \frac{\lVert A f \rVert_H}{\lVert f \rVert_H}$. This space is a non-abelian von Neumann algebra with respect to composition and adjoint of operators. That is, for any two operators $A,B\in \mathfrak B$ we have
\begin{displaymath}
    \lVert A B \rVert_{\mathfrak B} \leq \lVert A \rVert_{\mathfrak B} \lVert B \rVert_{\mathfrak B}, \quad \lVert A^* A \rVert_{\mathfrak B} = \lVert A \rVert^2_{\mathfrak B}, 
\end{displaymath}
which is the non-abelian analog of~\eqref{eqAAlg} making $\mathfrak B$ a $C^*$-algebra. Moreover, $\mathfrak B$ has a predual, $\mathfrak B_* := B_1(H)$, making it a von Neumann algebra. Here, $B_1(H) \subseteq B(H)$ is the space of trace-class operators on $H$, equipped with the trace norm, $\lVert A \rVert_1 := \tr \sqrt{A^*A}$, which can be thought of as a non-abelian analog of $L^1(X,\mu)$. The unitary group of Koopman operators $U^t$ on $H$ induces a unitary group $\mathcal U^t : \mathfrak B \to \mathfrak B$ (i.e., a group of linear maps mapping unitary operators to unitary operators), which acts by conjugation, i.e., 
\begin{equation}
    \label{eqHeisenberg}
    \mathcal U^t A = U^t A U^{t*}.
\end{equation}

The abelian algebra $\mathfrak A$ embeds isometrically into $\mathfrak B$ through the map $\pi : \mathfrak A \to \mathfrak B$, such that $ \pi f $ is the multiplication operator by $f$, $(\pi f) g = f g$. This map is injective, and satisfies $\pi(fg) = (\pi f)(\pi g)$, $\pi(f^*) = (\pi f)^*$ for all $ f,g \in \mathfrak A$. Thus, $\pi$ is a $^*$-representation, preserving the von Neumann algebra structure of $\mathfrak A$. The representation $\pi$ is also compatible with Koopman evolution, in the sense that $\mathcal U^t \circ \pi = \pi \circ U^t$
holds for all $t \in \mathbb R$. Equivalently, we have the commutative diagram
\begin{displaymath}
    \begin{tikzcd}
        \mathfrak A \ar[r,"U^t"] \ar[d,"\pi",swap] & \mathfrak A \ar[d,"\pi"] \\
        \mathfrak B \ar[r,"\mathcal U^t"] & \mathfrak B
    \end{tikzcd},
\end{displaymath}
which shows that $\pi$ provides a dynamically consistent representation of observables of the dynamical system in $\mathfrak A$ as elements of the non-abelian operator algebra $\mathfrak B$. In \cref{figSchematic}, we refer to the level of description involving $\mathfrak B$ as quantum mechanical due to the fact that $C^*$-algebras of operators on Hilbert space are employed in the algebraic formulation of quantum mechanics \cite{Emch09}. In particular, \eqref{eqHeisenberg} is mathematically equivalent to the Heisenberg picture for the unitary evolution of quantum observables (here, under the Koopman operator).     

\subsection{Embedding states} 

A dual construction to the representation $\pi: \mathfrak A \to \mathfrak B$ of observables can be carried out for the spaces of normal states $\ssa$ and $\ssb$. Let $\omega_p$ be the normal state in $\ssa$ induced by the probability density $p \in \mathfrak A_*$. Since $p$ is a positive function with $\lVert p \rVert_{\mathfrak A_*} = 1$, we have that $\sqrt p $ is a real unit vector in $ H $, and thus $ \rho = \langle \sqrt p, \cdot \rangle \sqrt p $ is a rank-1 orthogonal projection. Every such projection is a density operator inducing a normal state $\omega_\rho \in \ssb $ such that 
\begin{equation}
    \label{eqGamma}
    \omega_\rho A = \tr (\rho A) = \langle \sqrt p, A \sqrt p \rangle, \quad \forall A \in \mathfrak B. 
\end{equation}
Such states $\omega_\rho$ induced by unit vectors in $H$ are called \emph{vector states}. In fact, $\omega_\rho$ is a \emph{pure state}, i.e., it is an extremal point of the state space $S(\mathfrak B)$ (which is a convex set). Defining the map $\Gamma: \ssa \to \ssb$ such that $\Gamma(\omega_p) = \omega_\rho$, one can readily verify that $\Gamma$ is compatible with the regular representation $\pi$. That is, for every observable $f \in \mathfrak A$ and probability density $p \in \mathfrak A_*$, we have 
\begin{equation}
    \label{eqPiGamma}
    \omega_p f = \Gamma(\omega_{p})(\pi f). 
\end{equation}

Next, analogously to the transfer operator $P^t : S(\mathfrak A) \to S(\mathfrak A)$ given by the adjoint of the Koopman operator $U^t$ on $\mathfrak A$, we define $\mathcal P^t : S(\mathfrak B) \to S (\mathfrak B)$ as the adjoint of $ \mathcal U^t : \mathfrak B \to \mathfrak B$ from~\eqref{eqHeisenberg}, $ \mathcal P^t \omega = \omega \circ \mathcal U^t$. Note that $\mathcal U^t$ and $\mathcal P^t$ form a dual pair, i.e., $(\mathcal P^t \omega) A = \omega(\mathcal U^t A) $ for every state $\omega \in S(\mathfrak B)$ and element $A \in \mathfrak B$. Moreover, if $\omega_\rho \in \ssb$ is a normal state induced by a density operator $\rho \in \mathfrak B_*$, then $\mathcal P^t \omega_\rho = \omega_{\rho_t}$ where $\rho_t$ is the density operator given by $\rho_t = \mathcal U^{-t} \rho = U^{t*} \rho U^t$. 

In quantum mechanics, the evolution $\rho \mapsto \rho_t$ is known as the Schr\"odinger picture, and it is the dual of the Heisenberg picture from~\eqref{eqHeisenberg}. In the particular case that $\rho = \langle \xi, \cdot \rangle \xi$ is a vector state induced by $\xi \in H$ (which would be called a wavefunction in quantum mechanical language), we have $\rho_t = \langle \xi_t, \cdot \rangle \xi_t$, where $\xi_t = U^{t*} \xi $. Using this fact and~\eqref{eqGamma}, it follows that $ \Gamma $ is compatible with the evolution on $\ssa$ and $\ssb$ under the transfer operator; that is, $\mathcal P^t \circ \Gamma =  \Gamma \circ P^t$. This relation is represented by the commutative diagram 
\begin{equation}
    \label{eqCommutP}
    \begin{tikzcd}
        \ssa \ar[r,"P^t"] \ar[d,"\Gamma",swap] & \ssa \ar[d,"\Gamma"] \\
        \ssb \ar[r,"\mathcal P^t"] & \ssb
    \end{tikzcd},
\end{equation}
which also captures the correspondence between the abelian and quantum mechanical forecast steps in \cref{figSchematic}. 

\subsection{Probabilistic forecasting} 

In both abelian and non-abelian data assimilation, we can describe probabilistic forecasting of observables of the dynamical system using the formalism of \emph{positive operator-valued measures} (POVMs) \cite{DaviesLewis70}. First, we recall that an element $a$ of a $C^*$-algebra $\mathfrak W$ is (i) \emph{self-adjoint} if $a^* = a$; (ii) \emph{positive} (denoted as $a \geq 0 $) if $a = b^* b $  for some $ b \in \mathfrak W $; and (iii) a \emph{projection} if $a^*=a=a^2$. Supposing that $\mathfrak W$ is also a von Neumann algebra, a map $E : \Sigma \to \mathfrak W$ on the $\sigma$-algebra of a measurable space $(\Omega,\Sigma)$ is said to be a POVM if (i) for every set $S \in \Sigma$, $E(S) \geq 0$; (ii) $E(\Omega) = I$, where $I$ is the unit of $\mathfrak W$; and (iii) for every countable collection $S_1,S_2,\ldots$ of disjoint sets in $\Sigma$, $E(\bigcup_i S_i) = \sum_i E(S_i)$, where the sum converges in the weak-$^*$ topology of $\mathfrak W$ (i.e., for every element $\gamma$ of the predual $\mathfrak W_*$, we have $E(\bigcup_i S_i) \gamma = \lim_{n\to\infty} \sum_{i=1}^n E(S_i) \gamma $.) Together, these properties imply that for every $\gamma \in \mathfrak W_*$, the map $\mathbb P_{E,\gamma} : \Sigma \to \mathbb C$ given by 
\begin{equation}
    \label{eqProbMeasure}
    \mathbb P_{E,\gamma}(S) = E(S)\gamma 
\end{equation}
is a complex normalized measure. In particular, if $\gamma$ induces a normal state $ \omega_\gamma \in \ssw $, then $\mathbb P_{E,\gamma}$ is a probability measure on $\Omega$. We say that the POVM $E$ is a \emph{projection-valued measure} (PVM) if $E(S)$ is a projection for every $ S \in \Sigma$. 

In quantum mechanics, a triple $(\Omega,\Sigma,E)$ where $E$ is a POVM is referred to as an \emph{observable}. We alert the reader to the fact that in dynamical systems theory an observable is generally understood as a function $f: X \to V$ on state space $X$ taking values in a vector space, $V$. Thus, in situations where the space of dynamical observables forms an algebra (e.g., the  $\mathfrak A = L^\infty(X,\mu)$ algebra corresponding to $V = \mathbb C$), the term ``observable'' is overloaded and its meaning must be understood from the context. 

Given a POVM $(\Omega,\Sigma,E)$ as above, and a bounded, measurable function $ u : \Omega \to \mathbb C$, we define the integral $\int_\Omega u(\omega)\, dE(\omega) $ as the unique element $a$ of $\mathfrak W$ such that for every $\gamma \in \mathfrak W_*$, $ a \gamma = \int_\Omega u(\omega)\, d\mathbb P_{E,\gamma}(\omega)$.  If $a$ is a self-adjoint element of $\mathfrak W$, i.e., $a^*=a$, the spectral theorem states that there exists a unique PVM $E : \mathcal B(\mathbb R) \to \mathfrak W $ on the Borel $\sigma$-algebra $\mathcal B(\mathbb R)$ of the real line such that $ a = \int_{\mathbb R} \omega\, dE(\omega)$. 

In abelian data assimilation, the self-adjoint elements are the real-valued functions $ f $ in the von Neumann algebra $ \mathfrak A$ (i.e., the real-valued, essentially bounded observables in the dynamical systems sense), and every such $f$ has an associated PVM $E_f : \mathcal B(\mathbb R) \to \mathfrak A$. Explicitly, we have $E_f(S) = \chi_{f^{-1}(S)}$, where $ \chi_{f^{-1}(S)} : X \to \mathbb R $ is the characteristic function of the preimage $f^{-1}(S) \subseteq X$. If, at time $t$, the data assimilation system is in a normal state $\omega_{p_t} \in \ssa$ induced by a probability density $p_t \in \mathfrak A_*$, then the forecast distribution for $f$ at lead time $\tau \geq 0$  is given by $\mathbb P_{f,t,\tau} \equiv \mathbb P_{E_f,p_{t,\tau}}$, where $p_{t,\tau}$ is the probability density associated with the normal state $\omega_{p_{t,\tau}} = P^\tau \omega_{p_t}$. 

The forecast distribution $\mathbb P_{f,t,\tau}$ is equivalent to the distribution obtained via classical probability theory. That is, given an observable $f \in L^\infty(X,\mu)$, the density $p_{t,\tau} \in L^1(X,\mu) $ induces a probability distribution on $\mathbb R$, such that for every set $S \in \mathcal B(\mathbb R)$, $\Prob (S) = \int_{f^{-1}(S)} p_{t,\tau}\, d\mu$ is the probability that $f$ takes values in $S$. It follows by definition of $\mathbb P_{f,t,\tau}$ that $\Prob(S) = \mathbb P_{f,t,\tau}(S)$.

In the non-abelian setting of $\mathfrak B$, the spectral theorem states that for every self-adjoint operator $A \in \mathfrak B$, there exists a unique PVM $E_A : \mathcal B(\mathbb R) \to \mathfrak B$, such that $A = \int_{\mathbb R} y\, dE_A(y)$. If, at time $t$, the non-abelian data assimilation system is in a normal state $ \omega_{\rho_t}$ induced by a density operator $\rho_t \in \mathfrak B_*$, then the forecast distribution for $A$ at lead time $\tau \geq 0$ is given by $\mathbb P_{A,t,\tau} \equiv \mathbb P_{E_A,\rho_{t,\tau}}$, where $\rho_{t,\tau}$ is the density operator associated with $\omega_{\rho_{t,\tau}} = \mathcal P^\tau \omega_{\rho_t}$. This distribution is compatible with the embeddings of states $\Gamma : \ssa \to \ssb$ and observables $\pi: \mathfrak A \to \mathfrak B$ introduced above. That is, for every observable $f \in \mathfrak A$, probability density $p_{t,\tau} \in \ssa$, and Borel set $S \in \mathcal B(\mathbb R)$ we have $\mathbb P_{f,p_{t,\tau}}(S) = \mathbb P_{\pi f,\rho_{t,\tau}}(S)$, where $\Gamma(\omega_{p_{t,\tau}}) = \omega_{\rho_{t,\tau}}$.

\subsection{Representing observations by effects}

For a unital $C^*$-algebra $\mathfrak W$, an \emph{effect} is an element $e \in \mathfrak W$ satisfying $0 \leq e \leq I$. Intuitively, one can think of effects as generalizations of logical truth values, used to model outcomes of measurements or observations \cite{LeiferSpekkens13,JacobsZanasi16}. In Boolean logic, truth values lie in the set $\{ 0, 1 \}$. In fuzzy logic, truth values are real numbers in the interval $[0,1]$. In unital $C^*$-algebras, the analogs of truth values are elements $e$ satisfying $0 \leq e \leq I$ \cite{Gudder07}. We denote the set of effects in a $C^*$-algebra $\mathfrak W$ as $ \mathcal E(\mathfrak W)$. It can be shown that $\mathcal E(\mathfrak W)$ is a convex space, whose extremal points are projections. Given a state $\omega \in S(\mathfrak W)$ and an effect $e \in \mathcal E(\mathfrak W)$, we call the number $ \omega e \in [0, 1 ] $ as the \emph{validity} of $e$. Note that every effect $e \in E(\mathfrak W)$ induces a binary POVM $E : \{ 0, 1 \} \to \mathfrak W$ such that $E(\{ 1 \}) = e $ and $E(\{ 0 \} ) = I - e$.

Suppose now that $\mathfrak W$ is a von Neumann algebra, let $\omega_\rho \in \ssw$ be a normal state induced by an element $\rho \in \mathfrak W_*$, and let $e \in \mathcal E(\mathfrak W)$ be an effect. If the validity $\omega e$ is nonzero, we can define the conditional state $\omega_\rho\rvert_e \in \ssw $ as the normal state induced by $\rho\rvert_e \in \mathfrak W_*$, where
\begin{equation}
    \label{eqQBayes}
    \rho\rvert_e = \frac{\sqrt e p \sqrt e}{ \omega_p e}.
\end{equation}
The map $ \omega_\rho \mapsto \omega_\rho\rvert_e$ generalizes the Bayesian conditioning rule employed in the analysis step of classical data assimilation. 

As a concrete example, let $\mathfrak W = \mathfrak A$, let $S \in \mathcal B(X)$ be a measurable set (i.e., $S$ is an event), and let $\chi_S : X \to \{ 0, 1 \} $ be the characteristic function of $S$. According to Bayes' theorem, if $ p \in \mathfrak A_*$ is a probability density and $\int_S p \, d\mu > 0$, the conditional density of $p$ given $S$ is   
\begin{equation}
    \label{eqBayesAbelian}
    q = \frac{p \chi_S}{ \int_X p \chi_S \, d\mu} = \frac{\sqrt{\chi_S} p \sqrt{\chi_S}}{ \int_X p \chi_S \, d\mu}. 
\end{equation}
Since $\chi_S(x) \in \{ 0, 1 \} $ for every $x \in X$, it follows that $\chi_S$ is an effect in $\mathfrak A$, and since $ \int_X p \chi_S \, d\mu = \omega_p \chi_ S$, the Bayesian formula above is a special case of~\eqref{eqQBayes} with $p\rvert_e = q$. Note that to obtain the second equality in~\eqref{eqBayesAbelian} we made use of the commutativity of function multiplication, which does not hold in a non-abelian algebra. 

An important compatibility result between effects in the abelian algebra $\mathfrak A$ and effects in the non-abelian algebra $\mathfrak B$ is as follows: The regular representation $\pi : \mathfrak A \to \mathfrak B $ maps the effect space $\mathcal E(\mathfrak A)$ into the effect space $\mathcal E(\mathfrak B)$. As a result, and by virtue of \eqref{eqPiGamma}, for every normal state $ \omega_p \in \ssa $ and effect $e \in \mathcal E(\mathfrak A)$, the conditioned state $\omega_p \rvert_e$ satisfies 
\begin{equation}
    \label{eqBayesCompat}
    \Gamma (\omega_p\rvert_e) = (\Gamma \omega_p) \rvert_{\pi e}.
\end{equation}
This means that conditioning by effects in $\mathcal E(\mathfrak A)$ consistently embeds to conditioning by effects in $\mathcal E(\mathfrak B)$.

Next, let $Y$ be a set. In first-order logic, a predicate is a map $F : Y \to \{ 0, 1 \}$ such that $F(y)=1$ means that the proposition $F(y)$ is true, and $F(y)=0$ means that it is false. In fuzzy logic, predicates are generalized to maps $F: Y \to [0,1]$. In quantum logic, predicates are represented by effect-valued maps $ F : Y \to \mathcal E( \mathfrak W)$. Applying~\eqref{eqQBayes} for $e = F(y)$ leads to the update rule $ p \mapsto p\rvert_{F(y)} $, which represents the conditioning of the normal state associated with $p$ by the truth value of the proposition $F(y)$ associated with $ y \in Y$.    

In our algebraic data assimilation framework, we use an effect-valued map to carry out the analysis step given observations of the system in a space $Y$ (see \cref{figSchematic}). Specifically, let $h : X \to Y $ be a measurable observation map, such that $y = h(x)$ corresponds to the assimilated data given that the system is in state $x \in X$. Let $\psi : Y \times Y \to [0,1]$ be a measurable kernel function on $Y$, taking values in the unit interval. Every such kernel induces an effect-valued map $F: Y \to \mathcal E(\mathfrak A)$ given by $F(y) = \psi(y,h(\cdot))$. Possible choices for $\psi$ include bump kernels---in such cases, $F(y)$ can be viewed as a relaxation of a characteristic function $\chi_S$ of a set $S$ containing $h^{-1}(\{y\})$ (see~\eqref{eqKBump}). 

If, immediately prior to an observation at time $t + \Delta t$, the abelian data assimilation system has state $\omega_{p_{t,\Delta t}} \in \ssa$ (recall that $p_{t,\Delta t}$ is the forecast density for lead time $\Delta t$ initialized at time $t$), and $F(y)$ has nonzero validity with respect to  $\omega_{p_{t,\Delta t}}$, our analysis step updates $\omega_{p_{t,\Delta t}}$ to the conditional state $ \omega_{p_{t,\Delta t}} \rvert_{F(y)} \equiv \omega_{p_{t+\Delta t}} $ using~\eqref{eqQBayes}. In the non-abelian setting, we promote $F$ to the operator-valued function $\mathcal F : Y \to \mathcal E(\mathfrak B)$ with $ \mathcal F = \pi \circ F$, and we use again~\eqref{eqQBayes} to update the prior state $\omega_{\rho_{t,t+\Delta t}} \in \ssb$ to $\omega_{\rho_{t,\Delta t}} \rvert_{\mathcal F(y)} \equiv \omega_{\rho_{t+\Delta t}}$; see the Analysis column of the schematic in \cref{figSchematic}. By~\eqref{eqBayesCompat}, the abelian and non-abelian analysis steps are mutually consistent, in the sense that if $\omega_{\rho_{t,t+\Delta t}} = \Gamma(\omega_{p_{t,t+\Delta t}})$, then for every observable $ f \in \mathfrak A$, we have $\omega_{\rho_{t+\Delta t}}(\pi f) = \omega_{p_{t+\Delta t}}f$.

We should note that the effect-based analysis step introduced above can naturally handle data spaces $Y$ of arbitrarily high dimension, overcoming an important limitation of the QMDA framework proposed in ref.~\cite{Giannakis19b}. It is also worthwhile pointing out connections between effect-valued maps and feature maps in reproducing kernel Hilbert space (RKHS) theory \cite{PaulsenRaghupathi16}: If $\psi$ is positive-definite, there is an associated RKHS $\mathcal H$ of complex-valued functions on $X$ with $w(x,x') := \psi( h(x), h(x'))$ as its reproducing kernel. The map $F $ then takes values in the space $ \mathcal E(\mathfrak A) \cap \mathcal H$, and is thus an instance of a feature map. In the non-abelian case, one can think of $\mathcal F$ as an operator-valued feature map. Elsewhere \cite{GiannakisEtAl22}, we have found that operator-valued feature maps are useful for simulating classical dynamical systems using quantum computers.

\subsection{\label{secPositivityPreserving}Positivity-preserving discretization}

The abelian and non-abelian formulations of data assimilation described thus far employ the infinite-dimensional algebras $\mathfrak A$ and $\mathfrak B$, respectively. To arrive at practical computational algorithms, these algebras must be projected to finite dimensions, carrying along the associated dynamical and observation operators to finite-rank operators. We refer to this process as \emph{discretization}. 

To motivate our approach, we recall the definitions of quantum operations and  channels \cite{Holevo01}:  A linear map $T : \mathfrak W_2 \to \mathfrak W_1$ between two von Neumann algebras $\mathfrak W_1$ and $\mathfrak W_2$ is said to be a \emph{quantum operation} if (i) $T$ is \emph{completely positive}, i.e., for every $n \in \mathbb N$ the tensor product map $T \otimes \Id_n : M_n(\mathfrak W_2) \to M_n(\mathfrak W_1)$ is positive, where $M_n(\mathfrak W_1)$ and $M_n(\mathfrak W_2)$ are the von Neumann algebras of $n\times n $ matrices over $\mathfrak W_1$ and $\mathfrak W_2$, respectively; (ii) $T$ is the the adjoint of a map  $T_* : \mathfrak W_{1*} \to \mathfrak W_{2*} $ such that  $\omega_{T_*\rho} \bm 1 \leq 1$ for every normal state $\omega_\rho \in S_*(\mathfrak W_1)$. If, in addition, $\omega_{T_*\rho} \bm 1 =1$, $T$ is said to be a \emph{quantum channel}. 

In quantum theory, operations and channels characterize the transfer of information in open and closed systems, respectively. Here, the requirement of complete positivity of $T : \mathfrak W_1 \to \mathfrak W_2$ (as opposed to mere positivity) ensures that $T$ is extensible to a state-preserving map between any two systems that include $\mathfrak W_1$ and $\mathfrak W_2$ as subsystems. If $\mathfrak W_1$ is abelian, then positivity and complete positivity of $T$ are equivalent notions. If $\mathfrak W_2 = B(H_2)$ for a Hilbert space $H_2$, Stinespring's theorem \cite{Stinespring55} states that $T$ is completely positive if and only if there is a Hilbert space $H_1$, a representation $\varpi : \mathfrak W_1 \to B(H_1) $, and a bounded linear map $V: H_2 \to H_1 $ such that $ T a = V^* \varpi(a) V$.   

It follows from these considerations that the Koopman operator $U^t : \mathfrak A \to \mathfrak A $ is a quantum operation (since $U^t$ is positive, the transfer operator preserves normal states, and $\mathfrak A$ is abelian), and so is $\mathcal U^t : \mathfrak B \to \mathfrak B$ (by Stinespring's theorem). In fact, $U^t$ and $\mathcal U^t$ are both quantum channels. It is therefore natural to require that the discretization procedure leads to a quantum operation in both of the abelian and non-abelian cases. A second key requirement is that the discretization procedure is positivity preserving; that is, positive elements of the infinite-dimensional algebra are mapped into positive-elements of the finite-dimensional algebra associated with the projected system. This requirement is particularly important when modeling physical systems, where failure to preserve signs of sign-definite quantities may result to loss of physical interpretability and lead to numerical instabilities \cite{YuvalOGorman20}. Our third requirement is that the finite-dimensional approximations converge in an appropriate sense to the original system as the dimension increases. One of the main perspectives put forward in this paper is that the construction of discretization schemes meeting these requirements is considerably facilitated by working in the non-abelian setting of $\mathfrak B$ rather than the abelian setting of $\mathfrak A$. 

First, as an illustration of the fact that a ``naive'' projection will fail to meet our requirements, consider the Koopman operator $U^t : H \to H$. Fix an orthonormal basis $ \{ \phi_0, \phi_1, \ldots \} $ of $H$ with $ \phi_j \in \mathfrak A$, and let $ \Pi_L : H \to H$ be the orthogonal projection that maps into the $L$-dimensional subspace $ H_L := \spn \{ \phi_0, \ldots, \phi_{L-1} \}$. A common approach to Koopman and transfer operator approximation \cite{BerryEtAl15,KlusEtAl16} is to orthogonally project elements of $H$ to elements of $H_L$, $ f \mapsto f_L := \Pi_L f$, and similarly approximate $U^t$ by the finite-rank operator $U^{(t)}_L := \Pi_L U^t \Pi_L $. The rank of $U^{(t)}_L$ is at most $L$, and it is represented in the $\{ \phi_l \} $ basis by an $L \times L $ matrix $\bm U $ with elements $ U_{ij} = \langle \phi_i, U^t \phi_j \rangle_H$. Note that we have the inclusions $H_L \subset \mathfrak A \subset H$, and $H_L$ and $\mathfrak A$ are invariant subspaces of $H$ under $U^{(t)}_L$. Moreover, an element $f \in H$ is mapped under $U^{(t)}_L $ to $ g = U^{(t)}_L f \in H_L $ such that $ g = \sum_{i,j=0}^{L-1} \phi_i U_{ij} \hat f_j$, where $ \hat f_j = \langle \phi_j, f \rangle_H$. Letting $ \bm f = ( \hat f_0, \ldots, \hat f_{L-1} )^\top$ and $ \bm g = (\hat g_0, \ldots, \hat g_{L-1} )^\top$ with $ \hat g_j = \langle \phi_j, g \rangle_H$ be the $L$-dimensional column vectors giving the representation of $\Pi_L f$ and $g$ in the $ \{ \phi_j \}$ basis of $H$, respectively, we can express the action of $ U^{(t)}_L$ on $f$ as the matrix--vector product $\bm g = \bm U \bm f$. 

Unfortunately, such methods are not positivity-preserving; that is, if $f$ is a positive function in $\mathfrak A$, $ \Pi_L f $ need not be positive. A classical example is a tophat function on the real line, which develops oscillations to negative values upon Fourier filtering (the Gibbs phenomenon). Even if $f$ is a positive function in the finite-dimensional subspace $H_L$ (so that $\Pi_L f = f$), the function $ g = U^t_L f $ need not be positive. Thus, standard discretization approaches based on orthogonal projections fail to meet the requirements laid out above.       

Next, we turn to positivity-preserving discretizations utilizing the abelian algebra $\mathfrak A$, as opposed to the Hilbert space $H$. Recalling that the projections in $\mathfrak A$ are the multiplication operators by characteristic functions of measurable sets, let $S$ be a measurable subset of $X$, and consider the multiplication operator $M_S : \mathfrak A \to \mathfrak A$ such that $M_S f = \chi_S f $. The map $M_S$ is positive and the projected Koopman operator, $M_S U^t M_S$ is a quantum operation. However, in order for $M_S$ to be a discretization map, we must have that its range is a finite-dimensional algebra. This is equivalent to asking that the restriction of $\mu$ to $S$ is supported on a finite number of atoms, i.e., measurable sets that have no measurable subsets of positive measure. This is a highly restrictive condition that fails to hold for broad classes of dynamical systems (e.g., volume-preserving flows on manifolds), so the abelian algebra $\mathfrak A$ does not provide an appropriate environment to perform discretizations meeting our requirements.

We now come to discretizations based on the operator algebra $\mathfrak B$. Working with $\mathfrak B$ allows us to use both Hilbert space techniques to construct finite-rank operators by orthogonal projection \emph{and} algebraic techniques to ensure that these projections are positivity-preserving. With $H_L$ as above, consider the finite-dimensional von Neumann algebra $\mathfrak B_L := B(H_L)$. This algebra has dimension $L^2$, and is isomorphic to the algebra $\mathbb M_L$ of $L\times L$ complex matrices. In particular, each element $A \in \mathfrak B_L$ is represented by a matrix $\bm A \in \mathbb M_L$ with elements $A_{ij} = \langle \phi_i, A \phi_j \rangle_H$. Correspondingly, we refer to data assimilation based on $ \mathfrak B_L$ as \emph{matrix mechanical}; see \Lm\ in \cref{figSchematic}.  

Next, note that $\mathfrak B_L$ can be canonically identified with the subalgebra of $\mathfrak B$ consisting of all operators $A$ satisfying $\ker A \supseteq H_L$ and $\ran A \subseteq H_L $. As a result, we can view the projection $\bm \Pi_L : \mathfrak B \to \mathfrak B$ with $\bm \Pi_L A = \Pi_L A \Pi_L$ as an operator from $\mathfrak B$ to $\mathfrak B_L$; see \Lq\ and \Lm\ in \cref{figSchematic}. By Stinespring's theorem, $\bm \Pi_L$ is completely positive. As a result, (i) the projection $ A \in \mathfrak B \mapsto \bm \Pi_L A \in \mathfrak B_L $ is positivity-preserving, and thus so is the projected representation $\pi_L : \mathfrak A \to \mathfrak B_L$ with $\pi_L = \bm \Pi_L \circ \pi$; and (ii) the projected Koopman operator $ \mathcal U^{(t)}_L : \mathfrak B_L \to \mathfrak B_L $ with $\mathcal U^{(t)}_L A = U^{(t)}_L A U^{(t)*}_L $ is a quantum operation. Moreover, since $ \{ \phi_l \} $ is an orthonormal basis of $H$, for any $f \in H$ we have $\lim_{L \to \infty} \Pi_L  f = f$. This implies that for every $A \in \mathfrak B$  the operators $ A_L = \bm\Pi_L A \in \mathfrak B_L $ converge strongly to $ A$, i.e., we have $\lim_{L\to\infty} A_L g = A g$, for all $ g \in H$. In particular, $\pi_L f$ with $f \in \mathfrak A$ converges strongly to $\pi f$.
Further details on these approximations can be found in \cref{appOpApprox,appState,appObs,appKoopman,appQuantumChannel}. Note that, in general, $ \pi_L f $ is \emph{not} a multiplication operator. That is, the act of embedding $\mathfrak A$ in the non-abelian algebra $\mathfrak B$ using $\pi: \mathfrak A \to \mathfrak B$ and then projecting to the finite-dimensional subalgebra $\mathfrak B_L$ using $ \bm \Pi_L : \mathfrak B \to \mathfrak B_L$ is not equivalent to projecting $\mathfrak A$ into $H_L$ using $\Pi_L$ and then embedding $H_L$ into $\mathfrak B$ using $\pi$.

Consider now  a normal state $ \omega_p \in \ssa$ induced by a probability density $ p \in \mathfrak A_*$, and let $ \omega_\rho = \Gamma(\omega_p)$ be the associated normal state on $\mathfrak B$ obtained via~\eqref{eqGamma}. For $L$ sufficiently large, $C_L(\rho) := \bm \Pi_L \rho$ is nonzero, and thus $\rho_L = \bm \Pi_L \rho / C_L(\rho) $ is a density operator in $\mathfrak B_L$ inducing a state $\omega_{\rho_L} \in \ssbl$, as well as an extension of that state to $\ssb$ (which we continue to denote by $\omega_{\rho_L}$). In \cref{figSchematic}, we denote the map $\omega_{\rho} \mapsto \omega_{\rho_L}$ as $\bm \Pi'_L$. By construction, the state $\omega_{\rho_L}$ satisfies $\omega_{\rho_L} A = \omega_\rho(\bm \Pi_L A) /C_L(\rho)$ for all $ A \in \mathfrak B$. Setting, in particular, $A = \pi f$ with $ f \in \mathfrak A $, it follows from~\eqref{eqPiGamma} and the strong convergence of $\pi_L f$ to $\pi_f$ that
\begin{equation}
    \lim_{L\to\infty} \omega_{\rho_L} (\pi_L f) = \omega_\rho (\pi f) = \omega_p f;
    \label{eqWeakStarConv}
\end{equation}
see \cref{appState}. It should be kept in mind that, aside from special cases, $\omega_{\rho_L}$ is not the image of a state $\omega_{p_L}$ in $\ssa$ under $\Gamma$ for a probability density $p_L \in L^1(X,\mu)$; that is, in general $\omega_{\rho_L}$ is an intrinsically ``quantum mechanical'' state. Note also that $\omega_{\rho_L}$ in \eqref{eqWeakStarConv} is a vector state (see \eqref{eqGamma}) induced by the unit vector $\xi_L = \Pi_L \sqrt{p} / \lVert \Pi_L \sqrt{p} \rVert_H$, which, as just mentioned, is generally not the square root of a probability density. 

Let now $\mathcal P^{(t)}_L : S(\mathfrak B_L) \to S(\mathfrak B_L) $ with $\mathcal P^{(t)}_L \omega = \omega \circ \mathcal U^{(t)}_L $ be the projected transfer operator on $ S(\mathfrak B_L)$. Unless $H_L$ is a $U^t$-invariant subspace, $\mathcal P^{(t)}_L \circ \bm \Pi'_L $ is not equal to $ \bm \Pi'_L \circ \mathcal P^t $; see the dashed arrow in the third column of the schematic in \cref{figSchematic}. Nevertheless, we have the asymptotic consistency $\lim_{L\to\infty} \left( (\mathcal P^t_L \circ \bm \Pi'_L) \omega_\rho \right) A_L=  (\mathcal P^t \omega_\rho) A$, which holds for all $\omega_\rho \in \ssb$ and $A \in \mathfrak B$; see \cref{appChannelConsistency}. Applying this result for $A = \pi f$ and $ \omega_\rho = \Gamma(\omega_p)$, with $f \in \mathfrak A$ and $\omega_p \in \ssa$, it follows that 
\begin{equation}
    \label{eqForecastConv}
\lim_{L \to \infty} \left( (\mathcal P^{(t)}_L \circ \bm \Pi'_L) \omega_\rho \right) ( \pi_L f)  = (\mathcal P^t \omega_\rho)(\pi f) = (P^t \omega_p) f.   
\end{equation}
Equation~\eqref{eqForecastConv} implies that the matrix mechanical data assimilation scheme consistently recovers the forecast step of data assimilation in the abelian algebra $\mathfrak A$ in the limit of infinite dimension $L$. 

In \cref{appSpectralApprox} we describe how, for any self-adjoint element $A \in \mathfrak B$, the spectral measures of $\bm \Pi_L A$ converge to the spectral measure of $A$ in a suitable sense. Since $\pi f \in \mathfrak B $ is self-adjoint if and only if $f \in \mathfrak A $ is self-adjoint (real-valued), the spectral convergence of $\pi_L f$ to $\pi f$ implies that the forecast distributions $\mathbb P_{\pi_L f, t,\tau}$ associated with $\omega_{\rho_{t,\tau,L}} = (\mathcal P^{(\tau)}_L \circ \bm \Pi'_L) \omega_{\rho_{t,L}} $ consistently recover the forecast distributions $\mathbb P_{\pi f, t,\tau} $ and $\mathbb P_{\pi f, t,\tau} $ from the infinite-dimensional quantum mechanical and abelian systems, respectively. 

With a similar approach (see \cref{appEffect}) one can deduce that the analysis step is also consistently recovered: Defining the effect-valued map $\mathcal F_L : Y \to \mathcal E(\mathfrak B_L)$ with $\mathcal F_L = \bm \Pi_L \circ \mathcal F$, it follows from~\eqref{eqBayesCompat} and~\eqref{eqWeakStarConv} that for every $f \in \mathfrak A$ and $ \omega_p \in \ssa$,    
\begin{equation}
    \label{eqAnalysisConv}
    \lim_{L\to\infty} \omega_{\rho_L}\rvert_{\mathcal F_L(y)} (\pi f) = \omega_{\rho}\rvert_{\mathcal F(y)}(\pi f) = \omega_p\rvert_{F(y)}f,
\end{equation}
where $ \omega_\rho = \Gamma(\omega_p)$ and $\omega_{\rho_L} = \bm \Pi'_L \omega_{\rho_L} $, so the matrix mechanical analysis step is asymptotically consistent with the infinite-dimensional quantum mechanical and abelian analyses.

On the basis of~\eqref{eqForecastConv} and~\eqref{eqAnalysisConv}, we conclude that as the dimension $L$ increases, the matrix mechanical data assimilation system is consistent with the abelian formulation of sequential data assimilation. Moreover, the discretization leading to this system is positivity-preserving, and the projected Koopman operator $\mathcal U^{(t)}_L$ is a quantum operation. Thus, matrix mechanical data assimilation provides a non-abelian, finite-dimensional framework that simultaneously meets all of the requirements listed in the beginning of this subsection.      

\subsection{Data-driven approximation}

The matrix mechanical data assimilation scheme described above admits a consistent data-driven approximation using kernel methods for machine learning \cite{Giannakis19,Giannakis19b,GiannakisEtAl22}. The data-driven scheme employs three, possibly related, types of training data, all acquired along a dynamical trajectory $X_N = \{x_0, x_1, \ldots, x_{N-1} \} \subset X $ with $x_n = \Phi^{n\, \Delta t}(x_0)$, where $\Delta t > 0 $ is a sampling interval: (i) Samples $y_n = h(x_n)$ of the observation map $ h : X \to Y$; (ii) samples $f_n = f(x_n)$ of the forecast observable $f \in \mathfrak A$; (iii) samples $z_n = z(x_n)$ from a map $z : X \to Z $, used as proxies of the dynamical states $x_n$. If the $x_n$ are known, we set $ Z = X $ and $z = \Id$. Otherwise, we set $Z=Y^{2Q+1}$ for a parameter $Q \in \mathbb N$, and define $z$ as the  delay-coordinate map $z(x) = ( h(\Phi^{-Q\, \Delta t}(x)), h(\Phi^{(-Q + 1) \, \Delta t}(x)), \ldots, \Phi^{Q\,\Delta t}(x) )$, giving
\begin{equation}
    \label{eqDelay}
    z_n = (y_{n-Q}, y_{n-Q+1}, \ldots, y_Q). 
\end{equation}
By the theory of delay-coordinate maps \cite{SauerEtAl91}, for sufficiently large number of delays $Q$, $z$ becomes an injective map for typical observation maps $h$ and sampling intervals $\Delta t$.  

The dynamical trajectory $x_n$ has an associated sampling measure 
\begin{displaymath}
    \mu_N := \frac{1}{N}\sum_{n=0}^{N-1} \delta_{x_n} 
\end{displaymath}
and a finite-dimensional Hilbert space  $\hat H_N := L^2(X,\mu_N)$. By ergodicity, as $N$ increases, the measures $\mu_N$ converge to the invariant measure $\mu$ in \mbox{weak-$^*$} sense, so we can interpret $\hat H_N$ as a data-driven analog of the infinite-dimensional Hilbert space $H$ (see \cref{appAssumptions}). Given the training data $z_0,z_1,\ldots, z_{N-1}$, and without requiring explicit knowledge of the underlying states $x_n$, we use kernel integral operators to build an orthonormal basis $\{ \phi_{0,N}, \ldots, \phi_{L-1,N} \}$ of an $L$-dimensional subspace $H_{L,N} \subseteq \hat H_N$ that plays the role of a data-driven counterpart of $H_L$. More specifically, the basis elements $\phi_{l,N}$ are eigenvectors of a kernel integral operator $K_N : \hat H_N \to \hat H_N$ induced by a kernel function $\kappa : Z \times Z \to \mathbb R$. The operator $K_N$ is represented by an $N\times N$ kernel matrix $\bm K_N$ constructed from the training data $z_n$; see \cref{appBasis}. We let $\mathfrak B_{L,N} = B(H_{L,N})$ be $L^2$-dimensional algebra of linear maps on $H_{L,N}$, which, as in the case of $\mathfrak B_L$, is isomorphic to the matrix algebra $\mathbb M_L$.  

Every operator employed in the matrix-mechanical scheme described in \cref{secPositivityPreserving} has a data-driven counterpart, represented as an $L \times L$ matrix with respect to the $\phi_{l,N}$ basis. Specifically, the projected Koopman operator $U^{(t)}_L$ at time $q\, \Delta t$, $ q\in\mathbb Z$, is represented by an operator $U^{(q)}_{L,N} \in \mathfrak B_{L,N}$ induced by the shift map on the trajectory $x_n$ \cite{BerryEtAl15}, with a corresponding quantum operation $\mathcal U^{(t)}_{L,N}: \mathfrak B_{L,N} \to \mathfrak B_{L,N}$. Moreover, the projected multiplication operator $\pi_L f$ is represented by an operator $\pi_{L,N} \hat f_N \in \mathfrak B_{L,N}$, and the effect-valued map $\mathcal F_L$ is represented by a map $\mathcal F_{L,N} : Y \to \mathcal E(\mathfrak B_{L,N})$. Here, $\hat f_N $ denotes the restriction of $f$ on $\hat X_N$. See \cref{appOpApprox,appState,appObs,appKoopman,appQuantumChannel,appChannelConsistency,appEffect} for further details. 

The data-driven scheme is positivity-preserving and constitutes a quantum operation analogously to the matrix mechanical scheme. Moreover, by results on spectral approximation of kernel integral operators \cite{VonLuxburgEtAl08} and ergodicity of the dynamics, the kernel matrices $\bm K_N$ exhibit spectral convergence in the large-data limit, $N\to \infty$, to a kernel integral operator $K : H \to H$ in a suitable sense (see~\cref{thmK}). Correspondingly, all matrix representations of operators, and thus all predictions made by the data-driven scheme converge to the predictions of the matrix mechanical scheme \Lm\ in \cref{figSchematic}. Overall, we obtain a data-driven, positivity-preserving, and asymptotically consistent data assimilation scheme. 

\section{Lorenz 96 multiscale system}

As our first numerical example, we apply QMDA to assimilate and predict the slow variables of the Lorenz~96 (L96) multiscale system \cite{Lorenz96}. This system was introduced by Lorenz in 1996 as a low-order model of atmospheric circulation at a constant-latitude circle. The dynamical degrees of freedom include $K$ slow variables $ \lx_1, \ldots, \lx_K$, representing the zonal (west to east) component of the large-scale atmospheric velocity field at $K$ zonally equispaced locations. Each slow variable $\lx_k$ is coupled to $J$ fast variables $\ly_{1,k}, \ldots, \ly_{J,k}$, representing small-scale processes such as atmospheric convection. The state space of the dynamical system is thus $X = \mathbb R^{J(K+1)}$ with $ x = (\lx_k, \ly_{j,k})_{j,k=1,1}^{J,K} \in X$. 

The governing equations are
\begin{equation}
    \label{eqL96}
    \begin{gathered}
        \begin{aligned}
            \dot\lx_k &= -\lx_{k-1}(\lx_{k-2} - \lx_{k+1}) - \lx_k + \lf + \frac{\lh_{\lx}}{J} \sum_{j=1}^J \ly_{j,k}, \\
            \dot\ly_{j,k} &= \frac{1}{\varepsilon}\left( -\ly_{j+1,k}(\ly_{j+2,k} - y_{j-1,k}) - \ly_{j,k} + \lh_{\ly} \lx_k \right),
        \end{aligned}\\
        \lx_{k+K} = \lx_k, \quad \ly_{j,k+K} = \ly_{j,k}, \quad \ly_{j+J,k} = \ly_{j,k+1},
    \end{gathered}
\end{equation}
where the parameter $\lf$ represents large-scale forcing (e.g., solar heating), $\lh_\lx$ and $\lh_\ly$ control the coupling between the slow and fast variables, and $\varepsilon$ is a parameter that controls the timescale separation between the fast and slow variables. The governing equations for $\lx_k$ feature large-scale forcing, $\lf$, a quadratic nonlinearity, $-\lx_{k-1}(\lx_{k-2} - \lx_{k+1})$, representing advection, a linear damping term, $-\lx_k$, representing surface drag, and a flux term, $\lh_{\lx} \sum_{j=1}^J \ly_{j,k} / J$, representing forcing from the fast variables. The terms in the $\ly_{j,k}$ equations have similar physical interpretations. In general, the dynamics becomes more turbulent/chaotic as $\lf$ increases. 

Here, we focus on the chaotic dynamical regime studied in refs.~\cite{FatkullinVandenEijnden04,BurovEtAl21} with $K = 9$, $J=8$, $\varepsilon = 1/128$, $\lf = 10$, $\lh_\lx = -0.8$, and $\lh_\ly = 1$. We consider that the observation map $h : X \to Y$ projects the state vector $x \in X$ to the slow variables, i.e., $ Y = \mathbb R^K$ and $ h(x) = y := (\lx_1, \ldots, \lx_K)$. Our forecast observable $ f \in \mathfrak A$ is set to the first slow variable, $ f(x) = \lx_1$. In addition, we consider that only slow variables in $Y$, rather than full dynamical dynamical states in $X$, are available to us for training. 

\subsection{Training} We employ a training dataset consisting of $\tilde N=\text{40,000}$ samples $y_0, \ldots, y_{\tilde N-1} \in Y $ and $ f_0, \ldots, f_{\tilde N-1} \in \mathbb R$ with $ y_n = h(x_n)$, $f_n = f(x_n)$, and $ x_n = \Phi^{n\,\Delta t}(x_0)$, taken at a sampling interval $\Delta t = 0.05$. To assess forecast skill, we use $\hat N = \text{7,000}$ samples $\hat y_0, \ldots, \hat y_{\hat N-1}$ with $ \hat y_n = h(\hat x_n)$ and $ \hat x_n = \Phi^{n\,\Delta t}(\hat x_0)$, taken on an independent dynamical trajectory from the training data. See \cref{appL96} for further details on L96 data. 

Using the samples $y_n$, we form delay-embedded data $z_0, z_1, \ldots, z_{N-1} \in Z = \mathbb R^{(2Q+1) K}$ with $N = \tilde N - 2 Q $ by applying \eqref{eqDelay} with the delay parameter $Q=12$. We then use the $z_n$ to build the data-driven basis of $H_{L,N}$ for dimension $L=1000$. Using the basis vectors, we compute $L\times L $ matrix representations of the projected Koopman operators $U^{(\tau_j)}_{L,N}$ for lead times $\tau_j = j \, \Delta t$ with $ j \in \{ 0, 1, \ldots, J_\text{f}-1 \}$, $J_\text{f}=101$ (see \cref{algKoopman}). Moreover, using the $ \phi_{l,N}$ and the training samples $f_n$, we compute the $L \times L$ matrix representation $\bm A_{L,N}$ of the operator $A_{L,N} := \pi_{L,N}f$ associated with the forecast observable. To evaluate forecast distributions for $f$, we compute the PVM $E_{A_{L,N}}$ of $A_{L,N}$, which amounts to solving the eigenvalue problem for $\bm A_{L,N}$ (see \cref{algObservable}). To report forecast probabilities, we evaluate $E_{A_{L,N}}$ on a collection of bins $S_1, \ldots, S_M \subset \mathbb R$ of equal probability mass in the equilibrium distribution of $f$. As our observation kernel $\psi : Y \times Y \to [0, 1]$, we use a variable-bandwidth bump function. The corresponding effect-valued map $\mathcal F_{L,N} : Y \to \mathcal E(\mathfrak B_{L,N})$ is represented by a matrix-valued function in the $\phi_{L,N}$. See \cref{appEffect} and \cref{algEffect} for further details on $\psi$ and $\mathcal F_{L,N}$.      

\subsection{Data assimilation} We perform data assimilation experiments initialized with the pure state $\omega_0 \equiv \omega_{\rho_0} \in S(\mathfrak B_{L,N})$ induced by the density operator $ \rho_0 = \langle \bm 1_X, \cdot \rangle_{\hat H_N} \bm 1_X \in \mathfrak B_{L,N}$. We interpret this state as an uninformative equilibrium state, in the sense that (i) $\omega_0 A_{L,N} = \tr(\rho_0 A_{L,N}) = \bar f_N$, where $\bar f_N = \sum_{n=0}^{N-1} f_n / N$ is the empirical mean of the forecast observable $f$ from the training data; and (ii) $ \omega_{\rho_0}$ is invariant under the action of the transfer operator, i.e., $\mathcal P^{(t)}_{L,N}\omega_0 := \omega_0 \circ \mathcal U^{(t)}_{L,N} = \omega_0$. 

\begin{figure}
    \includegraphics[width=\linewidth]{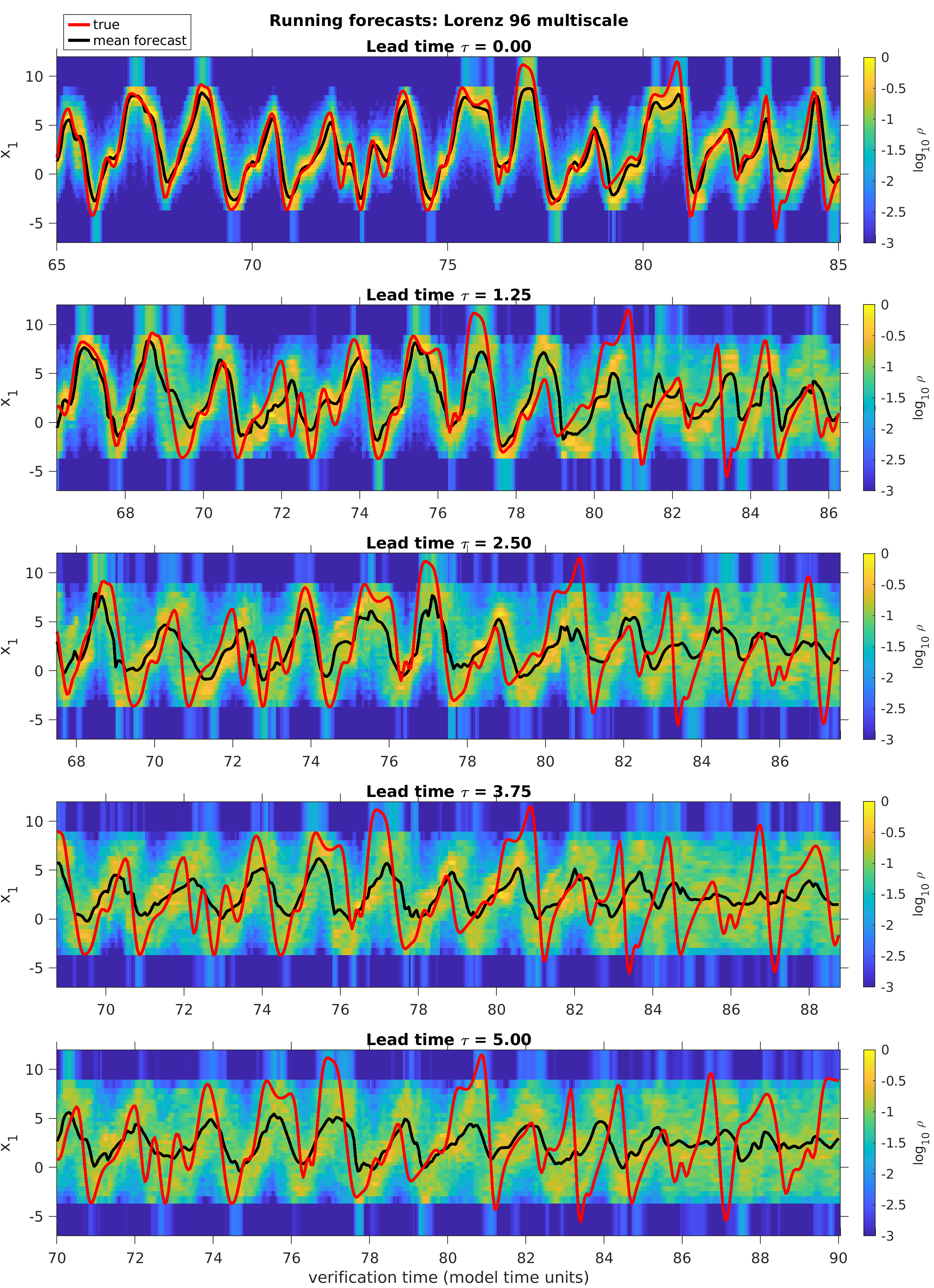}
    \caption{\label{figL96Prob}Running QMDA forecasts of the $\lx_1$ variable of the L96 multiscale system in a chaotic regime. The panels show the true $\lx_1$ evolution (black lines), the logarithm of the discrete forecast probability density $\varrho_{n,j}$ (colors), and the corresponding forecast mean (red lines) as a function of verification time for lead times in the range 0--5 model time units (top to bottom). The assimilated observable is the $K$-dimensional vector of the L96 slow variables, $ h = (x_1, \ldots, x_K)$, of the L96 system. }
\end{figure}
\begin{figure}
    \includegraphics[width=\linewidth]{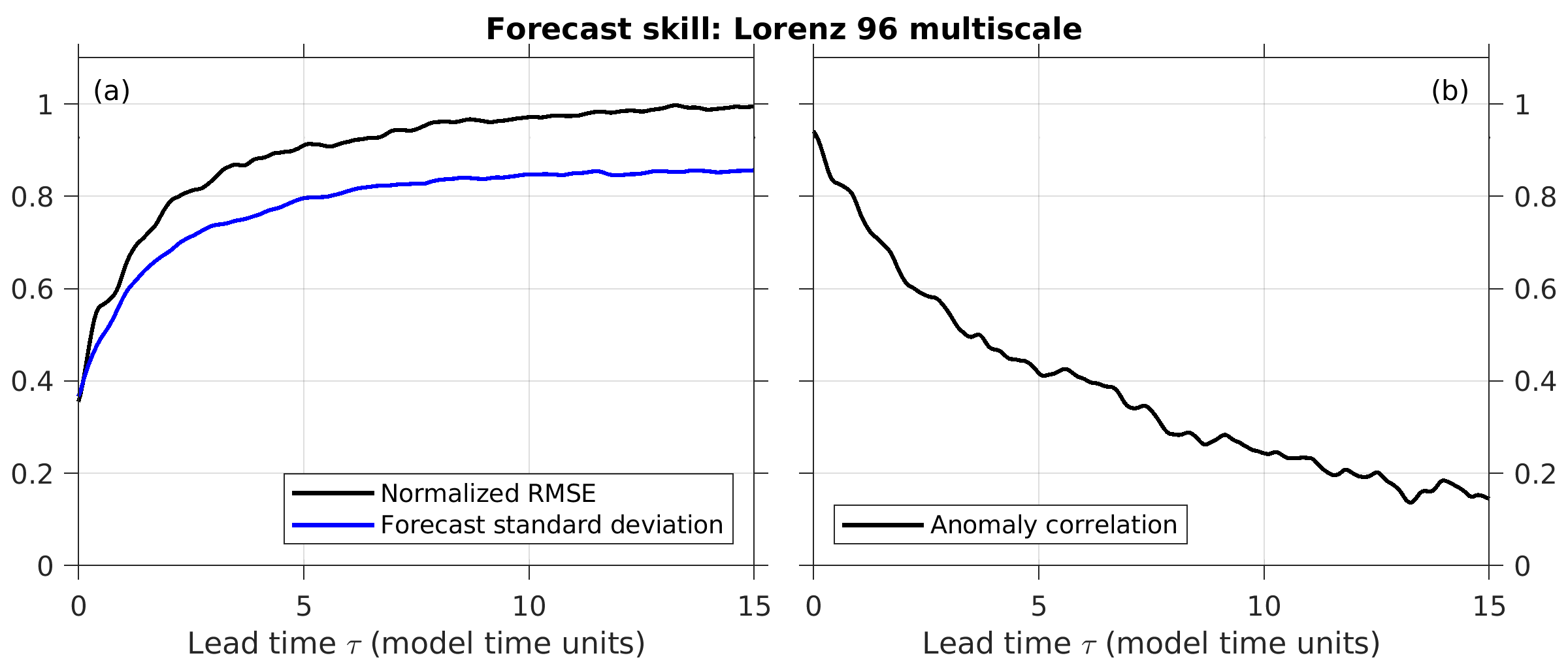}
    \caption{\label{figL96Err}NRMSE (left) and AC score (right) of the L96 forecasts from \cref{figL96Prob}.}
\end{figure}

Starting from $\omega_0$, QMDA produces a sequence of states $\omega_0, \omega_1, \ldots, \omega_{\hat N -1}$ by repeated application of the forecast--analysis steps, as depicted schematically in \cref{figSchematic} and in pseudocode form in \cref{algQMDA}. Specifically, for $ n \in \{ 1, \ldots, \hat N-1 \}$, we compute $\omega_n$ by first using the transfer operator to compute the state $\omega_{n-1,1} := \mathcal P^{(\Delta t)}_{L,N} \omega_{n-1}$ (which is analogous to the prior in classical data assimilation), and then applying the effect map to observation $\hat y_n$ to yield $ \omega_n = \omega_{n-1} \rvert_{{\mathcal F}_{L,N}(\hat y_n)}$ (which is analogous to the classical posterior). For each $n\in \{0, \ldots, \hat N-1 \}$, we also compute forecast states $ \omega_{n,j} = \mathcal P^{(\tau_j)}_{L,N} \omega_n$ and associated forecast distributions $\mathbb P_{n,j}$ for the observable $f$. We evaluate $\mathbb P_{n,j}$ on the bins $S_m$ and normalize the result by the corresponding bin size, $ s_m := \binlength (S_M) $ to produce discrete probability densities $\varrho_{n,j} = (\varrho_{n,j,0}, \ldots, \varrho_{n,j,M-1})$ with $\varrho_{n,j,m} := \mathbb P_{n,n}(S_m) / s_m $. We also compute the forecast mean and standard deviation, $ \bar f_{n,j} = \omega_{\rho_{n,j}} A_{L,N} $ and $\sigma_{n,j} = (\omega_{\rho_{n,j}} (A_{L,N})^2 - \bar f_{n,j}^2)^{1/2} $, respectively. We assess forecast skill through the normalized mean square error (NRMSE) and anomaly correlation (AC) scores, computed for each lead time $\tau_j$ by averaging over the $\hat N$ samples in the verification dataset; see \cref{appMetrics}.       

Figure \ref{figL96Prob} shows the forecast probability densities $\varrho_{n,j}$ (colors), forecast means $\bar f_{n,j}$ (black lines), and true signal $\hat f_{n+j}$ (red lines), plotted as a function of verification time $t_{n+j}$ over intervals spanning 20 time units for representative lead times $\tau_j $ in the range 0 to 5 time units. The corresponding NRMSE and AC scores are displayed in \cref{figL96Err}. Given the turbulent nature of the dynamics, we intuitively expect the forecast densities $\varrho_{n,j}$ to start from being highly concentrated around the true signal for small $\tau_j$, and progressively broaden as $\tau_j$ increases (i.e., going down the panels of \cref{figL96Prob}), indicating that the forecast uncertainty increases. Correspondingly, we expect the mean $\bar f_{n,j}$ to accurately track the true signal for $\tau_j$, and progressively relax towards the equilibrium mean $\int_X f \, d\mu$.

The results in Figs.~\ref{figL96Prob} and~\ref{figL96Err} are broadly consistent with this behavior: The forecast starts at $\tau_j = 0$  from a highly concentrated density around the true signal (note that \cref{figL96Prob} shows logarithms of $\varrho_{n,j}$), which is manifested by low NRSME and large AC values in \cref{figL96Err} of approximately 0.35 and 0.95, respectively. As $\tau_j$ increases, the forecast distribution broadens, and the NRMSE (AC) scores exhibit a near-monotonic increase (decrease). In \cref{figL96Err}(a), the estimated error based on the forecast variance $\sigma_{n,j}$ is seen to track well the NRMSE score, which indicates that the forecast distribution $\varrho_{n,j}$  provides an adequate representation of forecast uncertainty. It should be noted that errors are present even at time $\tau_j = 0$, particularly for periods of time where the true signal takes extreme positive or negative values. Such reconstruction errors are expected for a fully data-driven driven method applied to a system with a high-dimensional attractor. Overall, the skill scores in \cref{figL96Err} are comparable with the results obtained in ref.~\cite{BurovEtAl21} using kernel analog forecasting (KAF). 

\section{El Ni\~no Southern Oscillation}

The El Ni\~no Southern Oscillation (ENSO) \cite{Bjerknes69} is the dominant mode of interannual (3--5 year) variability of the Earth's climate system. Its primary manifestation is an oscillation between positive sea surface temperature (SST) anomalies over the eastern tropical Pacific Ocean, known as El Ni\~no events, and episodes of negative anomalies known as La Ni\~na events \cite{WangEtAl17}. Through atmospheric teleconnections, ENSO drives seasonal weather patterns throughout the globe, affecting the occurrence of extremes such as floods and droughts, among other natural and societal impacts \cite{McPhadenEtAl06}. Here, we demonstrate that QMDA successfully predicts ENSO within a comprehensive climate model by assimilating high-dimensional SST data.

\begin{figure}
    \includegraphics[width=\linewidth]{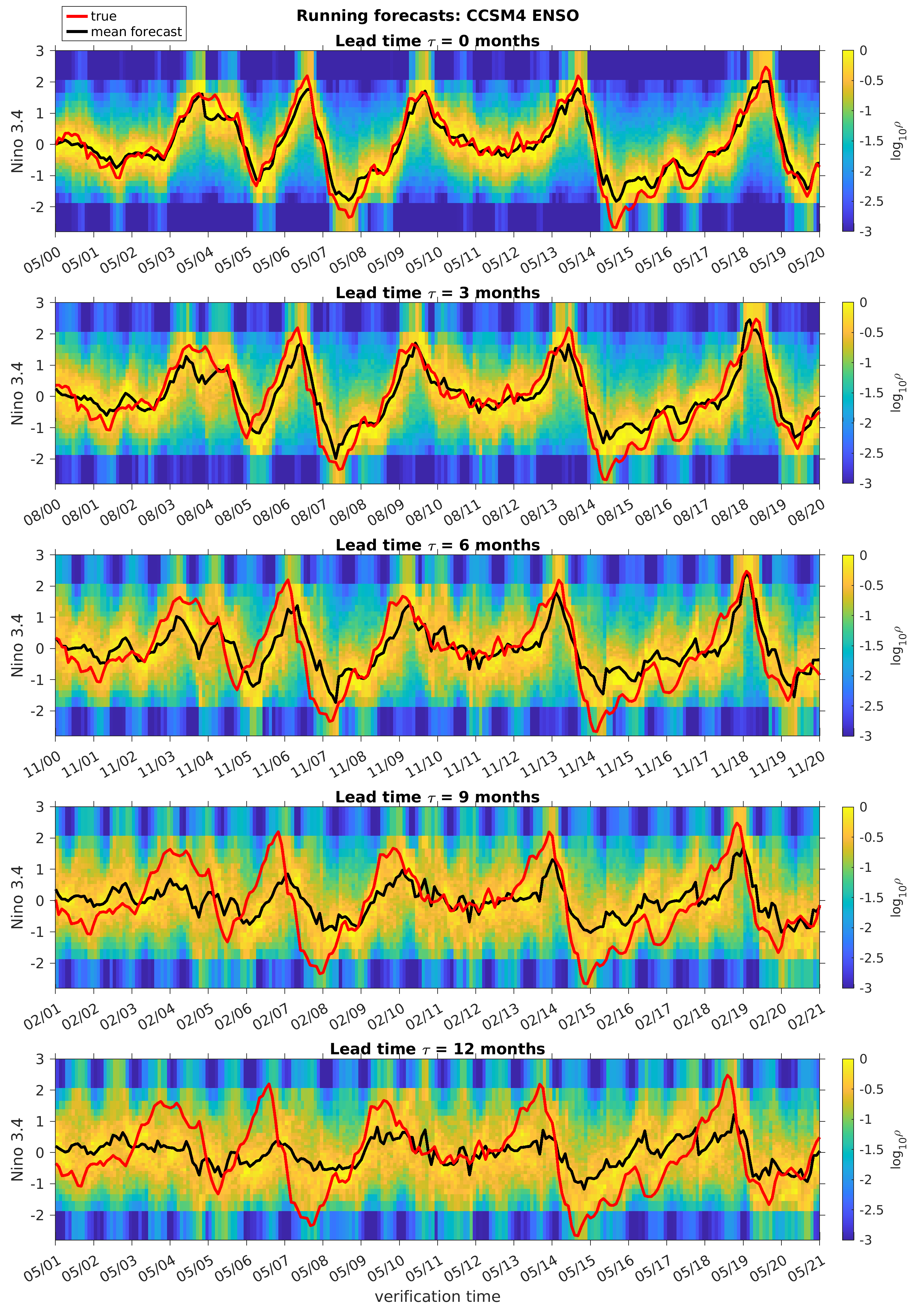}
    \caption{\label{figNinoProb}As in \cref{figL96Prob}, but for forecasts of the Ni\~no~3.4 index in CCSM4. The assimilated observable is the vector of Indo-Pacific SST gridpoint values. The forecast lead times are in 3-month increments in the range 0--12 months. The verification times ($x$-axes) are shown in mm/yy date format relative to an arbitrary year in the verification interval.}
\end{figure}

\begin{figure}
    \includegraphics[width=\linewidth]{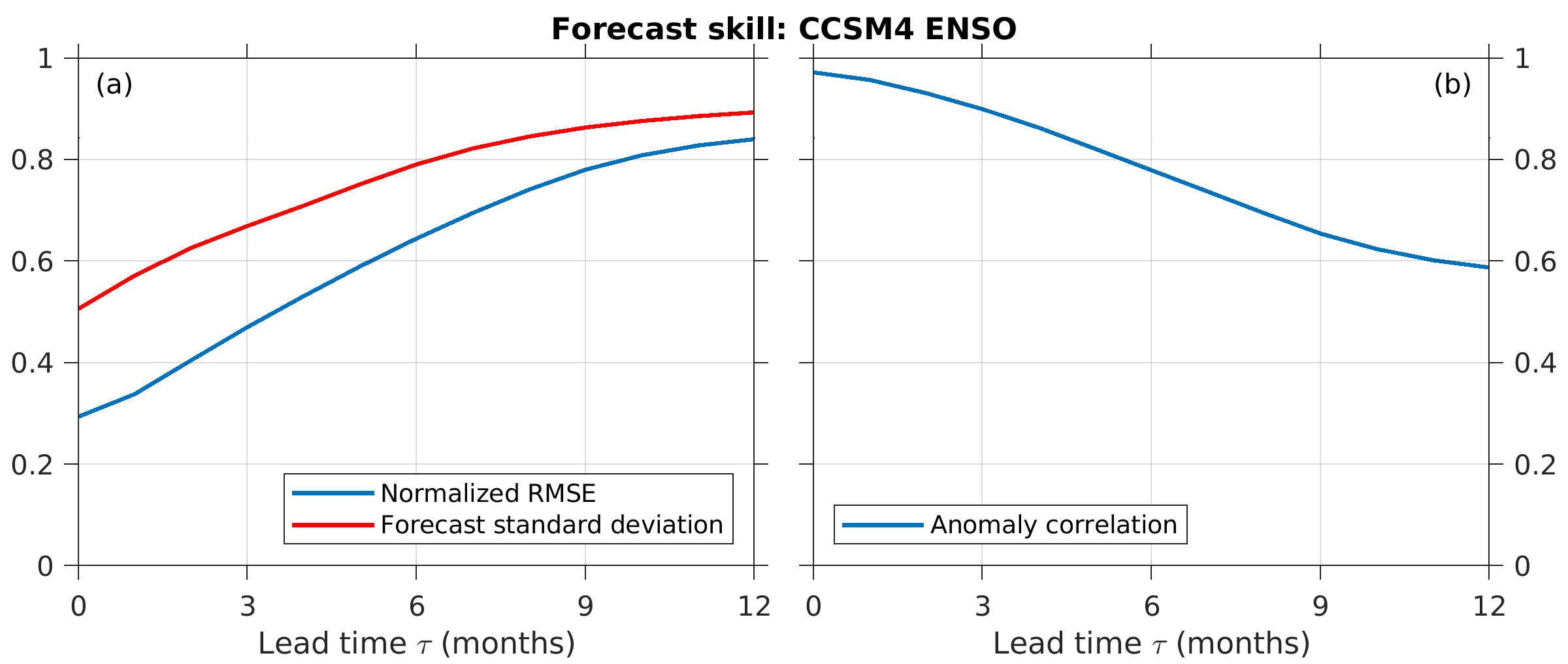}
    \caption{\label{figNinoErr}NRMSE (a) and AC score (b) of the Ni\~no~3.4 forecasts from \cref{figNinoProb}.}
\end{figure}

Our experimental setup follows closely ref.~\cite{WangEtAl20}, who performed data-driven ENSO forecasts using KAF. As training and test data, we use a control integration of the Community Climate System Model Version 4 (CCSM4) \cite{GentEtAl11}, conducted with fixed pre-industrial greenhouse gas forcings. The simulation spans 1300 years, sampled at an interval $\Delta t = 1$ month.  Abstractly, the dynamical state space $X$ consists of all degrees of freedom of CCSM4, which is of order $10^7$ and includes variables such as density, velocity, and temperature for the atmosphere, ocean, and sea ice, sampled on discretization meshes over the globe. Since this simulation has no climate change, there is an implicit invariant measure $\mu$ sampled by the data, and we can formally  define the algebras $\mathfrak A$ and $\mathfrak B$ associated with the invariant measure as described above. 

In our experiments, the observation map $h : X \to Y$ returns monthly-averaged SST fields on an Indo-Pacific domain; that is, we have $Y = \mathbb R^d$ where $d$ is the number of surface ocean gridpoints within the domain. We have $d = \text{44,414}$, so these experiments test the ability of QMDA to assimilate high-dimensional data. However, note that $h$ is a highly non-invertible map since Indo-Pacific SST comprises only a small subset of CCSM4's dynamical degrees of freedom. As our forecast observable $ f \in \mathfrak A$ we choose the Ni\~no~3.4 index---a commonly used index for ENSO monitoring defined as the average SST anomaly over a domain in the tropical Pacific Ocean. Large positive (negative) values of Ni\~no~3.4 represent El Ni\~no (La Ni\~na) conditions, whereas values near zero represent neutral conditions. Additional information on the CCSM4 data is included in \cref{appCCSM4}. 

Following ref.~\cite{WangEtAl20}, we use the SST and Ni\~no~3.4 samples from the first 1,100 years of the simulation as training data, and the corresponding samples for the last 200 years as test data. Thus, with the notation of the previously described L96 experiments, our training data are $ y_n = h(x_n)$ (Indo-Pacific SST) and $f_n = f(x_n)$ (Ni\~no~3.4) for $ n \in \{ 0, \ldots, N-1 \} $ and $N = \text{1,100} \times 12 = 13,200 $, and our test data are $ \hat y_n = h(x_{n+N})$ and $ \hat f_n = f(x_{n+N})$ for $ n \in \{ 0, \ldots, \hat N-1\} $ and $\hat N = 200 \times 12 = 2400 $. Here, $x_n = \Phi^{n\,\Delta t}(x_0) \in X$ is the (unknown) dynamical trajectory of the CCSM4 model underlying our training and test data. Using the SST samples $y_n$, we build the training data $z_n $ using delay-coordinate maps with parameter $Q = 5$; i.e., the data $z_n$ used for building the data-driven basis of $H_{L,N}$ consist of SST ``videos'' that span a total of $2Q+1 = 11 $ months. We use $L=3000$ dimensions. Following computation of the basis vectors $\phi_{l,N}$, the procedure for initializing and running QMDA is identical to our L96 experiments, and we will use the same notation to present results for ENSO. 

Figure~\ref{figNinoProb} shows the forecast probability density ($\varrho_{n,j} $; colors), forecast mean ($\bar f_{n,j}$; black lines), and true signal ($\hat f_{n+j}$; red lines) as a function of verification time $t_{n+j}$ over 20-year portions of the test dataset for lead times $\tau_j$ in the range 0 to 12 months. The corresponding NRMSE and AC scores are displayed in \cref{figNinoErr}. Qualitatively, the forecast density $\varrho_{n,j}$ displays a similar behavior as in the L96 experiments; that is, it is concentrated around the true signal on short lead times ($\tau_j \lesssim 3$ months), and gradually broadens as forecast uncertainty grows with increasing lead time $\tau_j$ due to chaotic climate dynamics. In \cref{figNinoErr}(a), the estimated forecast error based on the forecast variance $\sigma_{n,j}$ agrees reasonably well with the actual NRMSE evolution. Adopting $\text{AC}=0.6$ as a commonly used threshold for ENSO predictability, we see from the AC results in \cref{figNinoErr}(b) that QMDA produces useful forecasts out to $\tau_j \simeq 12$ months. The performance of QMDA in terms of the NRMSE and AC metrics is comparable to that found for KAF in ref.~\cite{WangEtAl20}, but QMDA has the advantage of producing full forecast probability distributions instead of point estimates. Compared to KAF, QMDA also has the advantage of being positivity-preserving. While failure to preserve signs may not be critical for sign-indefinite ENSO indices, there are many climatic variables such as temperature or moisture where sign preservation is particularly important.   

\section{Concluding remarks}

We have developed new theory and methods for sequential data assimilation of partially observed dynamical systems using techniques from operator algebra, quantum information, and ergodic theory. At the core of this framework, called quantum mechanical data assimilation (QMDA), is the non-abelian algebraic structure of spaces of operators. One of the main advantages that this structure provides is that it naturally enables finite-dimensional discretization schemes that preserve the sign of sign-definite observables in ways that are not possible with classical projection-based approaches. 

We build these schemes starting from a generalization of Bayesian data assimilation based on a dynamically consistent embedding into an infinite-dimensional operator algebra, to our knowledge described here for the first time. Under this embedding, forecasting is represented by a quantum operation induced by the Koopman  operator of the dynamical system, and Bayesian analysis is represented by quantum effects. In addition to providing a useful starting point for discretizing data assimilation, this construction draws connections between statistical inference methods for classical dynamical systems with quantum information and quantum probability, which should be of independent interest. 

QMDA leverages properties of operator algebras to project the infinite-dimensional framework into the level of a matrix algebra in a manner that positive operators are represented by positive matrices, and the finite-dimensional system is a quantum operation. QMDA also has a data-driven formulation based on kernel methods for machine learning with consistent asymptotic behavior as the amount of training data increases. We have demonstrated the efficacy of QMDA with forecasting experiments of the slow variables of the Lorenz 96 multiscale system in a chaotic regime and the El Ni\~no Southern Oscillation in a climate model. QMDA was shown to perform well in terms of point forecasts from quantum mechanical expectations, while also providing uncertainty quantification by representing entire forecast distributions via quantum states. 

This work motivates further application and development of algebraic approaches and quantum information to building models and performing inference of complex dynamical systems. In particular, as we enter the quantum computing era, there is a clear need to lay out the methodological and algorithmic foundations for quantum simulation of complex classical systems. Being firmly rooted in quantum information and operator theory, the QMDA framework presented in this paper is a natural candidate for implementation in quantum computers. As noted in the opening section of the paper, efforts to simulate classical dynamical systems on quantum computers are being actively pursued \cite{Joseph20,LiuEtAl21,GiannakisEtAl22}. Porting data assimilation algorithms such as QMDA to a quantum computational environment presents new challenges as the iterative nature of the forecast--analysis cycle will require repeated interaction between the quantum computer and the assimilated classical system. We believe that addressing these challenges is a fruitful area for future research with both theoretical and applied dimensions.        

\appendix

\section{\label{appAssumptions}Assumptions}

We make the following standing assumption on the dynamical system and forecast observable.

\begin{assump}\ 
    \label{assump1}
    \begin{enumerate}[label=(\alph*)]
        \item $\Phi^t :X \to X$, $ t\in \mathbb R$, is a continuous-time, continuous flow, on a compact metrizable space $X$. 
        \item $\mu$ is an invariant, ergodic, Borel probability measure under $\Phi^t$. 
        \item The forecast observable $ f : X \to \mathbb R$ is a real-valued function lying in $\mathfrak A = L^\infty(X,\mu)$.
    \end{enumerate}
\end{assump}

Note that the support of $\mu$ is a closed subset of the compact space $X$, and thus is compact. Moreover, the compactness assumption on $X$ can be replaced by the weaker assumption that $\Phi^t$ has a forward-invariant compact set $X_+$ that contains the support of $\mu$ (which is again necessarily compact). The analysis performed below can be carried over to this setting by replacing the space of continuous functions $C(X)$ (which is a Banach space equipped with the uniform norm when $X$ is compact) with $C(X_+)$.  

For the purpose of data-driven approximation, we additionally require:
\begin{assump}\ 
    \label{assump2}
    \begin{enumerate}[label=(\alph*)]
        \item For the sampling interval $\Delta t > 0 $, the discrete-time system induced by the map $\Phi^{\Delta t} : X \to X$ is ergodic with respect to $\mu$.
        \item The forecast observable $ f : X \to \mathbb R$ is continuous.  
        \item The observation map $h: X \to Y$ is continuous.
    \end{enumerate}
\end{assump}

\cref{assump2}(a) implies that for $\mu$-a.e.\ initial condition $x_0 \in X$, the sampling measures $\mu_N = \sum_{n=0}^{N-1} \delta_{x_n}/N$ with $ x_n = \Phi^{n\,\Delta t}(x_0)$ weak-$^*$ converge to the invariant measure $\mu$; that is, 
\begin{equation}
    \lim_{N\to\infty} \int_X f \, d\mu_N = \lim_{N\to\infty}\frac{1}{N} \sum_{n=0}^{N-1} f(x_n) = \int_X f \, d\mu, \quad \forall f \in C(X).
    \label{eqWeakConv}
\end{equation}
Henceforth, we will assume for convenience that the states $x_0, x_1, \ldots$ are all distinct---aside from the trivial case that the support of $\mu$ is a singleton set consisting of a fixed point, this assumption holds for $\mu$-a.e.\ initial condition $x_0$, and ensures that the Hilbert space $ \hat H_N = L^2(X,\mu_N)$ has dimension $N$. 

In what follows, 
\begin{displaymath}
    \langle f, g \rangle = \int_X f^* g \, d\mu, \quad \langle f, g \rangle_N = \int_X f^* g \, d\mu_N = \frac{1}{N} \sum_{n=0}^{N-1}f^*(x_n)g(x_n) 
\end{displaymath}
will denote the inner products of $H$ and $\hat H_N$, respectively. The Hilbert space $\hat H_N$ is isomorphic to $\mathbb C^N$ equipped with the normalized dot product $ \bm f \cdot \bm g \equiv \bm f^\dag \bm g /N $, where $\bm f^\dag$ is the Hermitian conjugate (complex conjugate transpose) of the column vector $\bm f \in \mathbb C^N$. Under this isomorphism, two elements $f, g \in \hat H_N$ are represented by column vectors $\bm f = (f(x_0), \ldots, f(x_{N-1}))^\top$ and $\bm g = (g(x_0), \ldots, g(x_{N-1}))^\top$, and we have $\langle f, g \rangle_N = \bm f \cdot \bm g$.

\section{\label{appMethods}Finite-dimensional approximation}

This section provides an overview and pseudocode listings of the data-driven approximation techniques underpinning QMDA. We begin with \cref{algQMDA}, which gives a high-level description of the QMDA pipeline employed in the L96 multiscale and CCSM4 experiments presented in the main text. The algorithm is divided up into two parts: 
\begin{enumerate}
    \item A training phase, which uses the training data $y_0, \ldots, y_{N-1} \in Y $ and $f_0, \ldots, f_{N-1} \in \mathbb R$ for the observation map $h$ and forecast observable $f$, respectively, to build an orthonormal basis $\{ \phi_{0,N}, \ldots, \phi_{L-1,N} \}$ of the Hilbert space $H_{L,N}$. The basis is used to approximate the Koopman operator $U^t$ of the dynamical system, the multiplication operator $\pi f $ representing the forecast observable, and the effect-valued map $\mathcal F$ employed in the analysis step. 
    \item A prediction phase, which iteratively executes the sequential forecast--analysis steps of QMDA given a test dataset of observations $\hat y_0, \ldots, \hat y_{\hat N-1} \in Y$. The state of the data assimilation system at time $t_n$ is a vector state of the operator algebra $\mathfrak B_{L,N}$, induced by a unit vector $\xi_n \in H_{L,N}$. This vector is represented in the $\{\phi_{l,N} \}$ basis of $H_{L,N}$ by a column vector $\bm \xi_n \in \mathbb C^L$.   
\end{enumerate}
\cref{algQMDA} depends on a number of lower-level procedures, which we describe in the following subsections.

\begin{algorithm}
    \caption{\label{algQMDA}QMDA pipeline employed in the L96 and CCSM4 experiments described in the main text.}
    \small
    \Inputs
    \begin{enumerate}
        \item Delay embedding parameter $Q\in \mathbb N$.
        \item Hilbert space dimension $L \in \mathbb N$.
        \item Number of spectral bins $M \in \mathbb N$.
        \item Kernel neighborhood parameter $k_\text{nn}$ in $\mathbb N$.
        \item Bandwidth exponent parameter $a>0$ and range parameters $J_1, J_2 \in \mathbb N$.  
        \item Number of forecast timesteps $J_\text{f} \in \mathbb N$.
        \item Training data from observation map, $y_{-Q}, \ldots, y_{N-1+Q} \in Y \equiv \mathbb R^d$ with $y_n = h(x_n)$.
        \item Training data from forecast observable, $ f_0, \ldots, f_{N-1+Q} \in \mathbb R $ with $ f_n = f(x_n)$. 
        \item Observed data $ \hat y_0, \ldots, \hat y_{\hat N-1} \in Y$ in test period.
    \end{enumerate}
    \Require All training data are induced by the same sequence of (unknown) time-ordered states $x_{-Q}, \ldots, x_{N-1+Q} \in X $ with $ x_n = \Phi^{n\,\Delta t}(x_0)$, taken at a fixed sampling interval $\Delta t > 0$.

    \Outputs
    \begin{enumerate}
        \item Mean forecast $\{ \bar f_{nj} \}$ for $ n \in \{ 0, \ldots, \hat N-1 \}$ and $ j \in \{ 0, \ldots, J_\text{f} \}$, where $\bar f_{nj}$ has initialization time $t_n = n \, \Delta t$ in the test period and lead time $\tau_j = j \,\Delta t$. 
        \item Forecast uncertainty $\{ \sigma_{nj} \}$ for $ n \in \{ 0, \ldots, \hat N-1 \}$ and $ j \in \{ 0, \ldots, J_\text{f} \}$, where $\sigma_{nj}$ has initialization time $ t_n$ and lead time $\tau_j$. 
        \item Spectral bins (intervals) $ S_0, \ldots, S_{M-1} \subseteq \mathbb R $.
        \item Forecast probability vectors $\{ \bm p_{nj} \}$ for $ n \in \{ 0, \ldots, \hat N-1 \}$ and $j \in \{ 0, \ldots, J_\text{f} \}$, where $\bm p_{nj} = (p_{0nj}, \ldots, p_{M-1,nj} )$ is the probability that, for initialization time $t_n$ and lead time $\tau_j$, the forecast observable $f$ lies in $S_m$. 
    \end{enumerate}

    \Training
    \begin{enumerate}
        \item Apply~\eqref{eqDelay} to the training data $y_n$ to build the delay-embedded dataset $z_0, \ldots, z_{N-1} \in Z \equiv \mathbb R^{(2Q+1)d}$. 
        \item Set $d_Z$ to the Euclidean distance on $Z$. Execute \cref{algBandwidth} with inputs $\{ z_n\}_{n=0}^{N-1}$, $d_Z$, and $k_\text{nn}$ to obtain a kernel bandwidth function $b_Z : Z \to \mathbb R_+$.   
        \item Execute \cref{algBasis} with inputs $\{z_n\}_{n=0}^{N-1}$, $d_Z$, and $k_\text{nn}$ to obtain basis vectors $\bm \phi_0, \ldots, \bm \phi_{L-1} \in \mathbb R^N$ for $H_{L,N}$.
        \item Execute \cref{algSpecBins} with inputs $ \{ f_n\}_{n=0}^{N-1}$ to obtain spectral bins $S_0, \ldots, S_{M-1} \subset \mathbb R$.
        \item Execute \cref{algObservable} with inputs $\{f_n\}_{n=0}^{N-1}$, $\{\bm \phi_l\}_{l=0}^{L-1}$, and $\{S_m \}_{m=0}^{M-1}$ to obtain the projected multiplication operator $\bm A \in \mathbb M_L$ representing $f$ and spectral projectors $\bm E_0, \ldots, \bm E_{M-1} \in \mathbb M_L$.
        \item For each $j \in \{1, \ldots, J_\text{f}\}$, execute \cref{algKoopman} with inputs $j$ and $\{\bm \phi_l\}_{l=0}^{L-1}$ to obtain Koopman operator matrices $\bm U^{(1)}, \ldots, \bm U^{(J_\text{f})} \in \mathbb M_L $.
        \item Set $d_Y$ to the Euclidean distance on $Y$. Execute \cref{algBandwidth} with inputs $\{ y_n\}_{n=0}^{N-1}$, $d_Y$, and $k_\text{nn}$ to obtain a kernel bandwidth function $b_Y : Y \to \mathbb R_+$.   
        \item Define the scaled distance function $\tilde d_Y : Y \times Y \to \mathbb R_+ $ with $\tilde d_Y(y,y') = d_Y(y,y')/ \sqrt{ b_Y(y) b_Y(y') } $. Execute \cref{algBandwidth} with inputs $\{y_n\}_{n=0}^{N-1}$, $a$, $J_1, J_2$, $\tilde d_Y$, and $\eta_\text{bump}$ (where $\eta_\text{bump}$ is the bump function from \eqref{eqEtaBump}) to obtain an optimal bandwidth parameter $\epsilon_*$. 
        \item Define the kernel function $\psi : Y \times Y \to [0, 1] $ with $\psi(y,y') = \eta_\text{bump}( \tilde d_Y(y,y') / \epsilon_* )$. Execute \cref{algEffect} with inputs $\psi$, $\{y_n\}_{n=0}^{N-1}$ and $\{\bm \phi_l\}_{l=0}^{L-1}$ to obtain the effect-valued feature map $\bm F : Y \to \mathbb M_L$.
    \end{enumerate}
    \Prediction
    \begin{enumerate}
        \item Set the initial state vector $\bm \xi_0 = (1, 0, \ldots, 0)^\top \in \mathbb C^L $.
        \item For each $n \in \{ 1, \ldots, \hat N-1\}$ execute the forecast--analysis cycle in \cref{algQMDAStepPure} with inputs $J_\text{f}$, $J_\text{o} = 1$, $\{ \bm U^{(j)} \}_{j=1}^{J_\text{f}}$, $\bm A$, $\{ \bm E_m \}_{m=0}^{M-1}$, $\bm F$, $\bm \xi_{n-1}$, and $\hat y_n$. 

            \medskip
            \Return
            \begin{itemize}
                \item The mean forecasts $\bar f_{n-1,0}, \ldots, \bar f_{n-1,J_\text{f}}$.
                \item The forecast uncertainties $\sigma_{n-1,0}, \ldots, \sigma_{n-1,J_\text{f}}$.
                \item The forecast probability vectors $\bm p_{n-1,0}, \ldots, \bm p_{n-1,J_\text{f}} $.
                \item The posterior state vector $\bm \xi_n$ given the observation $y_n$.
            \end{itemize}
    \end{enumerate}
\end{algorithm}

\subsection{\label{appBasis}Kernel eigenfunctions}

Following refs.~\cite{BerryEtAl15,GiannakisEtAl15,Giannakis19,Giannakis19b,DasEtAl21}, we use eigenvectors of kernel integral operators to construct both the $L$-dimensional Hilbert spaces $H_L$ and their data-driven counterparts $H_{L,N}$. We make the following assumptions on the kernels used to define these operators.

\begin{assump}\ 
    \label{assumpK}
    \begin{enumerate}[label=(\alph*)]
        \item $k : X \times X \to \mathbb R$ is a continuous, symmetric kernel.
        \item $k_0, k_1, \ldots : X \times X \to \mathbb R$ is a family of continuous, symmetric kernels such that, as $N \to \infty $, $k_N$ converges uniformly to $k$.   
    \end{enumerate}
\end{assump}

As we describe below, the kernels $k_N$ are typically data-dependent kernels obtained by normalization of a fixed (data-independent) kernel on $X$. 

Defining $ K : H \to H $ as the integral operator
\begin{displaymath}
    K f= \int_X k(\cdot,x) f(x) \, d\mu(x),
\end{displaymath}
we have that $K$ is a real, self-adjoint, Hilbert-Schmidt operator, and thus there exists a real, orthonormal basis $ \{ \phi_0, \phi_1, \ldots \}$ of $H = L^2(X,\mu)$ consisting of eigenvectors of $K$, 
\begin{equation}
    \label{eqKEig}
    K \phi_l = \lambda_l \phi_l, \quad \lambda_l \in \mathbb R,
\end{equation}
where the eigenvalues $\lambda_0, \lambda_1, \ldots$ are ordered in order of decreasing absolute value and converge to 0 as $l \to \infty$. In the data-driven setting, we replace $H$ by the $N$-dimensional Hilbert space $\hat H_N$, and define the integral operator $K_N : \hat H_N \to \hat H_N$ as
\begin{displaymath}
    K_N f = \int_X k_N(\cdot, x) f(x) \, d\mu_N(x) = \frac{1}{N} \sum_{n=0}^{N-1} k_N(\cdot,x_n) f(x_n).
\end{displaymath}
The operator $K_N$ has an associated real, orthonormal eigenbasis $ \{\phi_{0,N}, \ldots, \phi_{N-1,N} \}$ of $\hat H_N$, where
\begin{equation}
    \label{eqKNEig}
    K_N \phi_{l,N} = \lambda_{l,N}\phi_{l,N}, \quad \lambda_{l,N} \in \mathbb R,
\end{equation}
and the eigenvalues $\lambda_{l,N}$ are ordered again in order of decreasing modulus. 

An important property of the eigenvectors $\phi_l$ and $\phi_{l,N}$ corresponding to nonzero eigenvalues is that they have continuous representatives. Specifically, assuming that $\lambda_l$ and $\lambda_{l,N}$ are nonzero, we define $ \varphi_l, \varphi_{l,N} \in C(X)$ such that 
\begin{displaymath}
    \varphi_l(x) = \frac{1}{\lambda_l} \int_X k(x,x') \phi_l(x') \, d\mu(x'), \quad \varphi_{l,N}(x) = \frac{1}{\lambda_{l,N}} \int_X k_N(x,x') \phi_{l,N}(x') \, d\mu_N(x').
\end{displaymath}
It then follows from \eqref{eqKEig} and \eqref{eqKNEig}, respectively, that $\varphi_l = \phi_l $ $\mu$-a.e.\ and $\varphi_{l,N} = \phi_{l,N}$ $\mu_N$-a.e. Note that the latter relation simply means that $\varphi_{l,N}(x_n) = \phi_{l,N}(x_n)$ for every $n \in \{ 0, \ldots, N-1 \}$.

The following theorem summarizes the spectral convergence of the operators $K_N$ to $K$ and the convergence of the associated basis functions. The results are based on techniques developed in ref.~\cite{VonLuxburgEtAl08}. Additional details and proofs for the setting of ergodic dynamical systems and data-dependent kernels employed in this work can be found, e.g., in refs.~\cite{DasEtAl21,Giannakis21a}.  

\begin{theorem}
    \label{thmK}
    Under \cref{assump1}, \cref{assump2}, and \cref{assumpK}, the following hold as $N\to\infty$ for $\mu$-a.e.\ initial condition $x_0 \in X$. 
    \begin{enumerate}[label=(\alph*)]
        \item For each nonzero eigenvalue $\lambda_l$ of $K$, the sequence of eigenvalues $\lambda_{l,N}$ of $K_N$ converges to $\lambda_l$, including multiplicities.
        \item If $\phi_l \in H$ is an eigenvector of $K$ corresponding to $\lambda_l$ with continuous representative $\varphi_l \in C(X)$, there exists a sequence of eigenvectors $\phi_{l,N}$ of $K_N$ corresponding to eigenvalue $\lambda_{l,N} $, whose continuous representatives $\varphi_{l,N} \in C(X)$ converge uniformly to $\varphi_l$.
    \end{enumerate}
\end{theorem}

In numerical applications, we use the $\hat H_N \simeq \mathbb C^N$ isomorphism to represent the eigenvectors  $\phi_{l,N}$ by $N$-dimensional column vectors $\bm \phi_{l,N} \in \mathbb C^N$ (which are real since the $\phi_{l,N}$ are real) with $\bm \phi_{l,N} = (\phi_{l,N}(x_0), \ldots, \phi_{l,N}(x_{N-1}))^\top$. The vectors $\bm \phi_{l,N}$ are solutions of the eigenvalue problem
\begin{displaymath}
    \bm K_N \bm \phi_{l,N} = \lambda_{l,N} \bm \phi_{l,N} 
\end{displaymath}
for the $N\times N$ kernel matrix $\bm K_N = [K_{ij,N}]_{i,j=0}^{N-1} $ with $ K_{ij,N} =  k_N(x_i,x_j) / N$. We impose the orthonormality condition $ \bm \phi_{l,N} \cdot \bm \phi_{m,N} = \delta_{l,m}$, which is equivalent to $\langle \phi_{l,N}, \phi_{m,N} \rangle_N = \delta_{lm} $ on $\hat H_N$.

Henceforth, we will assume that for a given choice of basis vectors $\phi_l$ of $H$ and so long as $\lambda_l$ is nonzero, the data-driven basis vectors $\phi_{l,N}$ of $\hat H_N$ are chosen such that they converge to $\phi_l$ as per \cref{thmK}. This assumption leads to no loss of generality since every real, orthonormal eigenbasis $\phi_{l,N}$ can be orthogonally rotated to a basis that converges to $\phi_l$ without affecting the results of the computations presented below.

\subsection{Choice of kernel}

Since our training data $z_n$ are in the space $Z$, we employ kernels which are pullbacks of kernels on that space; specifically, we set $k(x,x') = \kappa(z(x),z(x'))$ and $k_N(x,x')=\kappa_N(z(x),z(x'))$, where $\kappa : Z \times Z \to \mathbb R$ and $\kappa_N : Z \times Z \to \mathbb R $ are continuous, symmetric kernels. With this approach, all kernel computations can be executed using the data $z_n \in Z$ without knowledge of the underlying dynamical states $x_n \in X$.  

Following ref.~\cite{DasEtAl21}, we construct the kernels $\kappa_N$ by applying the bistochastic normalization procedure introduced in ref.~\cite{CoifmanHirn13} to the family of variable-bandwidth diffusion kernels developed in ref.~\cite{BerryHarlim16}. 
Using the training data $z_n$, we construct a variable-bandwidth radial basis function kernel $\tilde\kappa_N : Z \times Z \to \mathbb R$ of the form
\begin{equation}
    \label{eqKVB}
    \tilde\kappa_N(z,z') = \eta_\text{gauss}\left( \frac{d(z,z')}{\epsilon_N \sqrt{b_N(z) b_N(z')}} \right),
\end{equation}
where $\eta_\text{gauss} : \mathbb R \to \mathbb R$ is the Gaussian shape function, $\eta_\text{gauss}(u) = e^{-u^2}$, $ d : Z \to Z \to \mathbb R_+$ is a distance function (which we nominally set to the Euclidean when $Z = \mathbb R^d$), $\epsilon_N > 0 $ is a bandwidth parameter, and $ b_N : Z \to \mathbb R_+$ is a (data-dependent) bandwidth function. The construction of the bandwidth function, which resembles a kernel density estimation procedure, is summarized in \cref{algBandwidth}. The bandwidth parameter $\epsilon_N$ is tuned automatically via \cref{algTuning}, which is described in \cref{appTuning} below. Further details on \cref{algBandwidth,algTuning} can be found in refs.~\cite{BerryEtAl15,BerryHarlim16}.

\begin{algorithm}
    \caption{\label{algBandwidth}Kernel bandwidth function.}
    \Inputs
    \begin{enumerate}
        \item Dataset $x_0,x_1,\ldots, x_{N-1} \in \mathcal X$; $\mathcal X$ is an arbitrary set.
        \item Distance function $d: \mathcal X \times \mathcal X \to \mathbb R_+$.
        \item Neighborhood parameter $k_\text{nn} \in \mathbb N$.
        \item Bandwidth exponent parameter $a>0$ and range parameters $J_1, J_2 \in \mathbb N$.  
    \end{enumerate}

    \Outputs
    \begin{enumerate}
        \item Bandwidth function $b: \mathcal X \to \mathbb R_+$.
    \end{enumerate}
    \Steps
    \begin{enumerate}
        \item Construct the function $ r : \mathcal X \to \mathbb R_+$, where $r^2(x) = \sum_{j=1}^{k_\text{nn}} d^2(x, I(x,j) ) / k_\text{nn} $ and  $I(x,j) \in \{ 0, \ldots, N - 1\}$ is the index of the $j$-th nearest neighbor of $x$ in the $\{ x_i \}_{i=0}^{N-1}$ dataset with respect to the distance $d$. 
        \item Construct the distance-like function $\tilde d : \mathcal X \times \mathcal X \to \mathbb R$ with $\tilde d(x,x') = d(x,x') / \sqrt{r(x) r(x')}$.
        \item Execute \cref{algTuning} with inputs $\{ x_n \}_{n=0}^{N-1}$, $\tilde d$, $\eta_\text{gauss}$, $k_\text{nn}$, $a$, $J_1$, and $J_2$ to obtain an optimal bandwidth $\epsilon_*$ and dimension estimate $m_*$.
        \item Construct the kernel $\tilde k : \mathcal X \times \mathcal X \to \mathbb R_+$, where $\tilde k(x,x') = \eta_\text{gauss}(\tilde d(x,x')/\epsilon_*)$.
        \item \Return The function $b : \mathcal X \to \mathbb R_+$ such that $b(x) = \sum_{j=0}^{N-1} \tilde k(x,x_j) / (N(\pi \epsilon_* r^2(x))^{m_*/2})$.
    \end{enumerate}
\end{algorithm}

Using the kernel $\tilde\kappa_N$, we perform the sequence of normalization steps described in ref.~\cite{CoifmanHirn13} to obtain a symmetric, positive, positive-definite kernel $\kappa_N$ which is Markovian with respect to the pushforward $ \nu_N := z_*(\mu_N)$ of the sampling measure on $Z$,
\begin{displaymath}
    \int_Z \kappa_N(z,z') \, d\nu_N(z') = \frac{1}{N}\sum_{n=0}^{N-1} \kappa_N(z,z_n) = 1, \quad \forall z \in Z.
\end{displaymath}
\cref{algBasis} describes the computation of the eigenvectors $\bm\phi_{l,N}$ associated with this kernel. We note that due to the particular form of the normalization leading to $\kappa_N$, the eigenvectors $\bm \phi_{l,N} $ can be computed without explicit formation of the $N\times N$ kernel matrix $\bm K_N$. Instead, we compute the $\bm \phi_{l,N}$ through the singular value decomposition (SVD) of an  $N\times N$ kernel matrix $\tbm K_N = [ \tilde \kappa_N(z_i,z_j) ]_{i,j=0}^{N-1} $ associated with a non-symmetric kernel function $\tilde \kappa_N : Z \times Z \to \mathbb R$ that factorizes $\bm K_N$ as $\bm K_N = N \tbm K_N \tbm K_N^\top $. The steps leading to $\tilde \kappa_N$ are listed in \cref{algBistoch}. See Appendix~B in ref.~\cite{DasEtAl21} for further details. 

\begin{algorithm}
    \caption{\label{algBasis}Orthonormal basis vectors of $\hat H_N$ from variable-bandwidth bistochastic kernel. We suppress $N$ subscripts from our notation for $\bm \phi_{l,N}$.}
    \Inputs 
    \begin{enumerate}
        \item Dataset $z_0,z_1,\ldots, z_{N-1} \in Z$.
        \item Distance function $d: Z \times Z \to \mathbb R_+$.
        \item Neighborhood parameter $k_\text{nn} \in \mathbb N$.
        \item Bandwidth exponent parameter $a>0$ and range parameters $J_1, J_2 \in \mathbb N$.  
        \item Number of basis vectors $L \leq N$.
    \end{enumerate}

    \Outputs
    \begin{enumerate}
        \item Column vectors $\bm \phi_0, \ldots, \bm \phi_{L-1} \in \mathbb R^N$.
    \end{enumerate}

    \Steps
    \begin{enumerate}
        \item Execute \cref{algBandwidth} with inputs $\{z_n\}_{n=0}^{N-1}$, $d$, $a$, $J_1$, and $J_2$ to obtain a bandwidth function $r : Z \to \mathbb R_+$.
        \item Construct the distance-like function $\tilde d : Z \times Z \to \mathbb R_+$ with $\tilde d(z,z') = d(z,z') / \sqrt{r(z) r(z')}$.
        \item Execute \cref{algTuning} with inputs $\{ z_n \}_{n=0}^{N-1}$, $\tilde d$, $\eta_\text{gauss}$, $a$, $J_1$, and $J_2$  to obtain an optimal bandwidth $\epsilon_*$ and dimension estimate $m_*$.
        \item Construct the kernel $k : Z \times Z \to \mathbb R_+$ with $ k(z,z') = \eta_\text{gauss}(\tilde d(z,z')/\epsilon_*)$.
        \item Execute \cref{algBistoch} with inputs $\{ z_n \}_{n=0}^{N-1}$ and $k$ to obtain a non-symmetric kernel function $\hat k : Z \times Z \to \mathbb R_+$.
        \item Form the $N\times N$ kernel matrix $\hbm K = [\hat K_{ij}]_{i,j=0}^{N-1}$ with $\hat K_{ij} =\hat k(z_i, z_j)$. 
        \item \Return The leading $L$ left singular vectors $\bm\phi_0, \ldots, \bm \phi_{L-1}$ of $\hbm K$, arranged in order of decreasing corresponding singular value, and normalized such that $\lVert \bm \phi_l \rVert_2 = \sqrt{N}$.
    \end{enumerate}
\end{algorithm}

\begin{algorithm}
    \caption{\label{algBistoch}Factorization of bistochastic kernel function.}
    \Inputs
    \begin{enumerate}
        \item Dataset $x_0,x_1,\ldots, x_{N-1} \in \mathcal X$; $\mathcal X$ is an arbitrary set.
        \item Kernel function $k: \mathcal X \times \mathcal X \to \mathbb R_+$.
    \end{enumerate}

    \Outputs
    \begin{enumerate}
        \item Non-symmetric kernel function $\hat k : \mathcal X \times \mathcal X \to \mathbb R$. 
    \end{enumerate}

    \Steps
    \begin{enumerate}
        \item Construct the degree function $ d : \mathcal X \to \mathbb R_+ $, where $d(x) = \sum_{j=0}^{N-1} k(x,x_j)$. 
        \item Construct the function $ q : \mathcal X \to \mathbb R_+ $, where $q(x) = \sum_{j=0}^{N-1} k(x, x_j) / d(x_j)$. 
        \item \Return The kernel function $\hat k : \mathcal X \times \mathcal X \to \mathbb R_+ $, where $ \hat k(x,x') = k(x, x') / ( d(x) q^{1/2}(x'))$.
    \end{enumerate}
\end{algorithm}

In addition to the bistochastic kernel from \cref{algBasis}, QMDA can be implemented with a variety of kernels, including non-symmetric kernels satisfying a detailed-balance condition (e.g., the family of normalized kernels from the diffusion maps algorithm \cite{CoifmanLafon06}). Two basic guidelines on kernel choice are that the data-dependent kernels $k_N$ are regular-enough such that the integral operators $K_N$ converge spectrally to $K$ (in the sense of \cref{thmK}), and the limit kernel $k$ is ``rich-enough'' such that all eigenvalues $\lambda_l$ are strictly positive (i.e., $k$ is an $L^2(\mu)$ integrally strictly-positive kernel \cite{SriperumbudurEtAl11}). In that case, as $N$ and $l$ increase, the eigenvectors $\phi_{l,N}$ provide a consistent approximation of an orthonormal basis for the entire Hilbert space $H$. The bistochastic kernels $k_N$ from \cref{algBasis} have this property if the map $z: X \to Z$ is injective. The latter, holds for sufficiently large delay parameter $Q$ from~\eqref{eqDelay} under appropriate assumptions on delay-coordinate maps \cite{SauerEtAl91}. 

For certain classes of kernels $k$ constructed from shape functions with rapid decay (e.g., the Gaussian shape function $\eta_\text{gauss}$), the asymptotic behavior of the eigenfunctions in the limit of vanishing bandwidth parameter $\epsilon_N$ may be studied using the theory of heat kernels  \cite{GrigoryanEtAl14}. Under appropriate conditions (e.g., the support of the invariant measure $\mu$ is a differentiable manifold or a metric measure space), the eigenfunctions are extremizers of a Dirichlet energy induced by the kernel, which defines a notion of regularity of functions akin to a Sobolev norm. In such cases, for any given $L \in \mathbb N$, the set of orthonormal vectors $ \{ \phi_0, \ldots, \phi_{L-1} \} $ (which we will use in \cref{appOpApprox} to define the subspaces $H_L \subset H$ used in QMDA) is optimal in the sense of having maximal regularity with respect to the kernel-induced Dirichlet energy.

\subsection{\label{appTuning}Bandwidth tuning} 

\cref{algTuning} is a tuning procedure for bandwidth-dependent kernels $k: \mathcal X \times \mathcal X \to \mathbb R_+ $ of the form $k(x,x') = \kappa( d(x, x' ) / \epsilon )$, where $\mathcal X$ is an arbitrary set, $ d : \mathcal X  \times \mathcal X \to \mathbb R_+$ is a distance-like function, $ \kappa : \mathbb R_+ \to \mathbb R_+$ a positive kernel shape function, and $\epsilon > 0 $ a kernel bandwidth parameter. The tuning approach in \cref{algTuning} was proposed in ref.~\cite{CoifmanEtAl08} using scaling arguments for heat-like kernels on manifolds, and was also used in refs.~\cite{BerryEtAl15,BerryHarlim16,Giannakis19b}. It takes as input a dataset in $\mathcal X$ and a logarithmic grid of candidate bandwidth values $\epsilon_j$, and returns an ``optimal`` bandwidth $\epsilon_*$ from this candidate set that maximizes a kernel-induced dimension function $ m(\epsilon_j)$ for the dataset. If $k$ is a heat-like kernel on a Riemannian manifold, $m(\epsilon_*)$ is an estimator of the manifold's dimension, but $m(\epsilon_*)$ also provides a notion of dimension for non-smooth sets.   

\begin{algorithm}
    \caption{\label{algTuning}Tuning of bandwidth-dependent kernels.}
    \Inputs
    \begin{enumerate}
        \item Dataset $x_0, x_1, \ldots, x_{N-1} \in \mathcal X$; $\mathcal X$ is an arbitrary set.
        \item Bandwidth exponent parameter $a>0$ and range parameters $J_1, J_2 \in \mathbb N$.  
        \item Distance-like function $d: \mathcal X \times \mathcal X \to \mathbb R_+$. 
        \item Kernel shape function $ \eta : \mathbb R_+ \to \mathbb R_+ $.
    \end{enumerate}

    \Outputs
    \begin{enumerate}
        \item Optimal bandwidth $\epsilon_* > 0$.
        \item Estimated dataset dimension $m_* \geq 0$. 
    \end{enumerate}
    \Steps
    \begin{enumerate}
        \item Compute the $N\times N$ pairwise distance matrix $\bm D = [ D_{ij} ]_{i,j=0}^{N-1}$ with $D_{ij} = d(x_i,x_j)$. 
        \item Generate logarithmic grid $\{ \epsilon_j \}_{j=J_1}^{J_2}$ with $\epsilon_j = 2^{a j}$.
        \item For each $j \in \{ J_1, \ldots, J_2 \}$, compute the kernel sum $S(\epsilon_j) = \sum_{i,l=0}^{N-1} K_{il}/N^2$, where $K_{il} = \eta(D_{il} / \epsilon_j)$.
        \item For each $ j \in \{ J_1 + 1, \ldots, J_2 - 1\}$, compute the logarithmic derivative 
            \begin{displaymath}
                m(\epsilon_j) = \frac{\log S_{j+1} - \log S_{j-1} }{\log \epsilon_{j+1} - \log \epsilon_{j-1}} = \frac{\log(S_{j+1}/S_{j-1})}{2a}. 
            \end{displaymath}
        \item \Return $\epsilon_* = \argmax_{\epsilon_j \in \{ \epsilon_{J_1}, \ldots, \epsilon_{J_2} \}} m(\epsilon_j)$ and $m_* = m(\epsilon_*)$.
    \end{enumerate}
\end{algorithm}

\subsection{\label{appOpApprox}Finite-dimensional Hilbert spaces and operator approximation}

Given the basis vectors $\phi_l$ and $\phi_{l,N}$ from \cref{appBasis}, we define the $L$-dimensional Hilbert spaces
\begin{displaymath}
    H_L = \spn\{\phi_0, \ldots, \phi_{L-1} \} \subset H, \quad H_{L,N} = \spn\{ \phi_{0,N}, \ldots, \phi_{L-1,N} \} \subset \hat H_N,
\end{displaymath}
where in the case of $H_{L,N}$ $L$ is at most $N$. As in the main text, we let $\Pi_L : H \to H$ and $\Pi_{L,N}: \hat H_N \to \hat H_N $ be the orthogonal projections on $H$ and $\hat H_N$, respectively, with $\ran \Pi_L = H_L$ and $\ran \Pi_{L,N} = H_{L,N}$. We also let $\bm \Pi_L : \mathfrak B \to \mathfrak B$ and $ \bm \Pi_{L,N} : \hat{\mathfrak B}_N \to \hat{\mathfrak B}_N$ be the induced projections on the operator algebras $\mathfrak B = B(H)$ and $\mathfrak B_N = B(\hat H_N)$, defined as $ \bm \Pi_L A = \Pi_L A \Pi_L$ and $ \bm \Pi_{L,N} A = \Pi_{L,N} \hat A \Pi_{L,N}$, respectively. Defining $\mathfrak B_L = \ran \bm \Pi_L$, we can canonically identify $\mathfrak B_L$ with the subalgebra of $\mathfrak B$ consisting of all operators $A$ satisfying $\ker A \supseteq H_L$ and $\ran A \subseteq H_L $. The space $\mathfrak B_{L,N} := \ran \bm \Pi_{L,N}$ can be canonically identified with a subalgebra of $\hat{\mathfrak B}_N$ in a similar manner. We will be making these identifications whenever convenient. 

Within this setting, we are interested in the following two types of operator approximation, respectively described in \cref{appOpCompression,appOpDataApprox}. 
\begin{enumerate}
    \item Approximation of an operator $A \in \mathfrak B$ by a finite-rank operator $A_L \in \mathfrak B_L$.
    \item Approximation of $A_L \in \mathfrak B_L$ by an operator $A_{L,N} \in \mathfrak B_{L,N}$.
\end{enumerate}
Intuitively, we think of an approximation of the first type listed above as a ``compression'' of an operator $A \in \mathfrak B$ of possibly infinite rank to an operator $A_L \in \mathfrak B_L \subset \mathfrak B$ of at most rank $L$. Approximations of the second type are of a fundamentally different nature since there are no inclusion relationships between $\mathfrak B_L$ and $\mathfrak B_{L,N}$. One can think instead of such approximations as data-driven approximations of the \emph{representation} of an operator in a basis.  

\subsubsection{\label{appOpCompression}Operator compression} 

Given $A \in \mathfrak B$, we define $A_L \in \mathfrak B_L$ as 
\begin{equation}
    \label{eqAL}
    A_L := \bm \Pi_L A = \Pi_L A \Pi_L, 
\end{equation}
Since $ \{ \phi_0, \phi_1, \ldots \}$ is an orthonormal basis of $H$, the projections $\Pi_L$ converge strongly to the identity; that is, for every $f \in H$, we have $\lim_{L\to\infty}(\Pi_L - I ) f = 0$, where the limit is taken in the norm of $H$. As a result, the operators $\breve A_L := A \Pi_L$ converge strongly to $A$, $\lim_{L\to\infty}(\breve A_L - A ) f = 0$ for all $ f \in H$.
It then follows from standard results in functional analysis that $A_L = \Pi_L \breve A_L $ converges strongly to $A$, i.e, 
\begin{equation}
    \lim_{L\to\infty} A_L f= A f, \quad \forall f \in H.
    \label{eqALConv}
\end{equation}
As we will see below, this type of strong operator convergence is sufficient to deduce convergence of the matrix mechanical formulation of data assimilation based on $\mathfrak B_L$ to the infinite-dimensional quantum mechanical level based on $\mathfrak B$ (see the rows labeled \Lm\ and \Lq\ in the schematic of \cref{figSchematic}). 

\subsubsection{\label{appOpDataApprox}Data-driven operator approximation} 

In order to facilitate approximation of operators in $\mathfrak B_L$ by operators in $\mathfrak B_{L,N}$, we use operators acting on the Banach space of continuous functions $C(X)$ as intermediate approximations. In what follows, we will let $\iota : C(X) \to H$ and $\iota_N : C(X) \to \hat H_N$ be the canonical linear maps that map continuous functions to their $L^2$ equivalence classes in $H$ and $\hat H_N$, respectively. In addition, we let $\mathfrak C = B(C(X))$ be the unital Banach algebra of bounded linear operators on $C(X)$. We assume $L\in \mathbb N$ is chosen such that the eigenvalues $\lambda_{L-1}$ and $\lambda_{L-1,N}$ of $K$ and $K_N$ from~\eqref{eqKEig} and~\eqref{eqKNEig}, respectively, are nonzero. This means that all elements of $H_L$ and $H_{L,N}$ have continuous representatives. 

With these definitions and assumptions, we restrict attention to approximation of operators $A_L \in \mathfrak B_{L,N}$ which are obtained by applying \eqref{eqALConv} to operators $A \in \mathfrak B$ that satisfy
\begin{equation}
    \label{eqIota}
    A \circ \iota = \iota \circ \tilde A,
\end{equation}
for some $\tilde A \in \mathfrak C$. In addition, we assume that there is a uniformly bounded family of operators $\hat A_N \in \mathfrak B_N$ that satisfy an approximate version of~\eqref{eqIota} in the following sense: For every $ f \in C(X)$, the norm of the residual $(\hat A_N \circ \iota_N - \iota_N \circ \tilde A) f$ converges to 0. That is, we require
\begin{equation}
    \label{eqIotaN}
    \lim_{N\to\infty} \lVert R_N f \rVert_{\hat H_N} = 0, \quad R_N = \hat A_N \circ \iota_N - \iota_N \circ \tilde A, \quad \forall f \in C(X),
\end{equation}
where the operators $ \hat A_N$ satisfy the uniform norm bound
\begin{equation}
    \label{eqIotaN2}
    \lVert \hat A_N \rVert_{\hat{\mathfrak B}_N} \leq a,
\end{equation}
for a constant $a$. As we will see in the ensuing subsections, under \cref{assump2}, all operators employed in QMDA satisfy~\eqref{eqIota}, \eqref{eqIotaN}, and~\eqref{eqIotaN2}. 

We have the following approximation lemma for the matrix elements of $A$ in terms of the matrix elements of $\hat A_N$. 

\begin{lemma} 
    \label{lemApprox}
Suppose that $A \in \mathfrak B$, $\hat A_N \in \hat{\mathfrak B}_N$, and $\tilde A \in \mathfrak C$ satisfy~\eqref{eqIota}, \eqref{eqIotaN}, and~\eqref{eqIotaN2}. Then, under \cref{assump1} and \cref{assump2}, and with the notation and assumptions of \cref{appBasis}, the matrix elements of $\hat A_N$ in the $\{\phi_{l,N} \}$ bases of $\hat H_N$ converge almost surely to the matrix elements of $A$ in the $\{\phi_l \}$ basis of $H$. That is, for $\mu$-a.e.\ initial condition $x_0 \in X$, and every $i,j \in \mathbb N$ such that $\lambda_i, \lambda_j \neq 0$,
    \begin{displaymath}
        \lim_{N\to\infty} \langle \phi_{i,N}, \hat A_N \phi_{j,N} \rangle_N = \langle \phi_i, A \phi_j \rangle.
    \end{displaymath}
\end{lemma}

\begin{proof}
    See \cref{appProofApproxLemma}.   
\end{proof}

Let $A_{ij} = \langle \phi_i, A \phi_j \rangle$ and $A_{ij,N} = \langle \phi_{i,N}, \hat A_N \phi_{j,N} \rangle_N $. The convergence of $ A_{ij,N}$ to $A_{ij}$ from \cref{lemApprox} is not uniform with respect to $i,j \in \mathbb N$. However, restricting $i$ and $j$ to the finite index set $\{0, \ldots, L-1\}$ associated with the basis vectors of the finite-dimensional spaces $H_L$ and $H_{L,N}$ makes the convergence of $\hat A_{ij,N}$ to $ \hat A_{ij}$ uniform, and we can conclude that the matrix representations of the projected operators $ A_{L,N} = \bm \Pi_{L,N} \hat A_N$ converge to the matrix representation of $ A_L = \bm \Pi_L A$. 

\begin{corollary}
    \label{corApprox}
    With notation as above, let $\bm A_L = [ A_{ij} ]_{i,j=0}^{L-1} $ and $\bm A_{L,N} = [A_{ij,N}]_{i,j=0}^{L-1}$ be the $L\times L$ matrix representations of $A_L$ and $A_{L,N}$ in the $\{\phi_l \}$ and $\{\phi_{l,N}\}$ bases of $H_L$ and $H_{L,N}$, respectively. Then, for $\mu$-a.e.\ initial condition $x_0$, we have $\lim_{N\to\infty} \bm A_{L,N} = \bm A_L$ in any matrix norm.   
\end{corollary}

\subsubsection{\label{appProofApproxLemma}Proof of \cref{lemApprox}}

Recall from \cref{appBasis} that the $\phi_{i,N}$ have continuous representatives $\varphi_{i,N}$ which converge $\mu$-a.s.\ to the continuous representatives $\varphi_i$ of $\phi_i$ in the uniform norm of $C(X)$. Note also that for every $N \in \mathbb N$, $\iota_N : C(X)\to \hat H_N$ has unit operator norm. Using these facts, we get
\begin{align}
    \nonumber\lvert \langle \phi_{i,N}, \hat A_N \phi_{j,N} \rangle_N - \langle \phi_i, A \phi_j \rangle \rvert &= \lvert \langle \phi_{i,N}, \hat A_N \iota_N \varphi_{j,N} \rangle_N - \langle \phi_i, A \phi_j \rangle \rvert  \\
    \nonumber& \leq \lvert \langle \phi_{i,N}, \hat A_N \iota_N (\varphi_{j,N} - \varphi_j) \rangle_N \rvert \\
    \nonumber & \quad + \lvert \langle \phi_{i,N}, \hat A_N \iota_N \varphi_j \rangle_N - \langle \phi_i, A \phi_j \rangle \rvert  \\
    \nonumber &\leq a \lVert \varphi_{j,N} - \varphi_j \rVert_{C(X)} \\
    \label{eqProof1} & \quad + \lvert \langle \phi_{i,N}, \hat A_N \iota_N \varphi_j \rangle_N - \langle \phi_i, A \phi_j \rangle \rvert.
\end{align}
Moreover, we have 
\begin{align}
    \nonumber\lvert \langle \phi_{i,N}, \hat A_N \iota_N \varphi_j \rangle_N - \langle \phi_i, A \phi_j \rangle \rvert &= \lvert \langle \phi_{i,N}, \iota_N \tilde A \varphi_j \rangle_N + \langle \phi_{i,N} R_N \varphi_j \rangle_N - \langle \phi_i, A \phi_j \rangle \rvert \\
    \nonumber& \leq \lvert \langle \phi_{i,N}, \iota_N \tilde A \varphi_j \rangle_N - \langle \phi_i, A \phi_j \rangle \rvert + \lVert R_N \varphi_j \rVert_{\hat H_N} \\
    \nonumber& =\lvert \langle\iota_N \varphi_{i,N}, \iota_N\tilde A \varphi_j \rangle_N - \langle \phi_i, A \phi_j \rangle \rvert + \lVert R_N \varphi_j \rVert_{\hat H_N}\\
    \nonumber& =\lvert \langle \iota_N (\varphi_{i,N} - \varphi_i ), \iota_N \tilde A \varphi_j \rangle_N + \langle \iota_N\varphi_i , \iota_N\tilde A \varphi_j \rangle_N \\
    \nonumber & \quad - \langle \phi_i, A \phi_j \rangle \rvert + \lVert R_N \varphi_j \rVert_{\hat H_N} \\
    \nonumber & \leq \lVert \varphi_{i,N} - \varphi_i \rVert_{C(X)} \lVert \tilde A\rVert_{\mathfrak C} \lVert \varphi_j \rVert_{C(X)} \\
    \label{eqProof2} & \quad + \lvert \langle \iota_N\varphi_i , \iota_N\tilde A \varphi_j \rangle_N- \langle \phi_i, A \phi_j \rangle \rvert + \lVert R_N \varphi_j \rVert_{\hat H_N}. 
\end{align}

Now, by~\eqref{eqIota} we have
\begin{align*}
    \lvert \langle \iota_N\varphi_i , \iota_N\tilde A \varphi_j \rangle_N- \langle \phi_i, A \phi_j \rangle \rvert &= \lvert \langle \iota_N\varphi_i , \iota_N\tilde A \varphi_j \rangle_N- \langle \iota\varphi_i, A \iota \varphi_j \rangle \rvert \\
    &= \lvert \langle \iota_N\varphi_i , \iota_N\tilde A \varphi_j \rangle_N- \langle \iota\varphi_i, \iota \tilde A \varphi_j \rangle \rvert \\
    &= \left\lvert \int_X \varphi_i \tilde A \varphi_j d\mu_N- \int_X \varphi_i \tilde A \varphi_j \,d\mu \right\rvert, 
\end{align*}
so by the weak-$^*$ convergence of $\mu_N$ to $\mu$ (see \eqref{eqWeakConv}) it follows that for $\mu$-a.e.\ initial condition $x_0$, 
\begin{displaymath}
    \lim_{N\to\infty}\lvert \langle \iota_N\varphi_i , \iota_N\tilde A \varphi_j \rangle_N- \langle \phi_i, A \phi_j \rangle \rvert = 0.
\end{displaymath}
Using this result, the uniform convergence of $\varphi_{i,N}$ to $\varphi_i$, and~\eqref{eqIotaN} in~\eqref{eqProof2}, we obtain
\begin{displaymath}
    \lim_{N\to\infty}\lvert \langle \phi_{i,N}, \hat A_N \iota_N \varphi_j \rangle_N - \langle \phi_i, A \phi_j \rangle \rvert = 0.
\end{displaymath}
Finally, using the above and the uniform convergence of $\varphi_{i,N}$ to $\varphi_i$ in~\eqref{eqProof1}, we arrive at 
\begin{displaymath}
    \lim_{N\to\infty}\lvert \langle \phi_{i,N}, \hat A_N \phi_{j,N} \rangle_N - \langle \phi_i, A \phi_j \rangle \rvert = 0,
\end{displaymath}
which holds again for $\mu$-a.e.\ initial condition $x_0$. This completes the proof of the lemma.

\subsection{\label{appState}Approximation of states}

Let $\omega_\rho \in \ssb$ be a normal state of $\mathfrak B$ induced by a density operator $\rho \in \mathfrak B_*$. We recall that the predual $\mathfrak B_* $ of $ \mathfrak B$ is the space of trace-class operators on $H$ (denoted as $B_1(H)$ in the main text), equipped with the trace norm, $\lVert A \rVert_{\mathfrak B_*} = \tr\sqrt{A^* A}$. In the case of the finite-dimensional algebras $\mathfrak B_L$ and $\mathfrak B_{L,N}$, the preduals $\mathfrak B_{L*}$ and $\mathfrak B_{L,N*}$, respectively, can be identified with the algebras themselves, but we will continue to distinguish them using $_*$ subscripts since they are equipped with a different norm (the trace norm) from the operator norm of the algebras.    

As in \cref{appOpApprox}, we are interested in two types of state approximation, which can be thought of as state compression and data-driven approximation, respectively:
\begin{enumerate}
    \item Approximation of $\rho$ by a finite-rank density operator $\rho_L \in \mathfrak B_L$; see \cref{appStateCompression}.
    \item Approximation of $\rho_L$ by a data-driven density operator $\rho_{L,N} \in \mathfrak B_{L,N*}$; see \cref{appStateDataApprox}.
\end{enumerate}

\subsubsection{\label{appStateCompression}State compression}

Similarly to \cref{appOpCompression}, for a given density operator $\rho \in \mathfrak B_* $ we define the projected operators $\sigma_L = \bm \Pi_L \rho$. Letting $C_L = \tr \sigma_L$, we have $C_L \leq \tr \rho = 1 $, so in general the $\sigma_L$ are not density operators. Nevertheless, the $\sigma_L$ are positive, finite-rank (and thus trace class)  operators that converge to $\rho$ in trace norm (as opposed to merely strongly; cf.\ \eqref{eqALConv}). Indeed, we have $ \rho - \sigma_L = (I - \Pi_L) \rho (I - \Pi_L) $, so $\rho - \sigma_L$ is positive, and
\begin{displaymath}
    \Vert \rho - \sigma_L \rVert_{\mathfrak B_*} = \tr( \rho-\sigma_L) = \sum_{l=L}^\infty \langle \phi_l, \rho \phi_l \rangle,
\end{displaymath}
where the sum in the right-hand side of the last equality is a positive, decreasing function of $L$, converging to 0 as $L \to \infty$. We also have $C_L = \sum_{l=0}^{L-1} \langle \phi_l, \rho \phi_l \rangle$, which implies that $\lim_{L\to\infty} C_L= 1$, and thus that there exists $L_* \in \mathbb N$ such that $C_L > 0 $ for all $L > L_*$. For any such $L$, $\rho_L := \sigma_L / C_L$ is a density operator, and the sequence $\rho_L$ converges to $\rho$ in trace norm, 
\begin{equation}
    \lim_{L\to\infty} \lVert \rho_L - \rho \rVert_{\mathfrak B_*} =0.
    \label{eqRhoLConv}
\end{equation}
In the main text, we denote the map that sends the normal state $\omega_\rho \in \ssb$ to $\omega_{\rho_L} \in S_*(\mathfrak B_L)$ as $\bm \Pi'_L(\rho) = \rho_L$.

Let now $A$ be an element of $ \mathfrak B$ with corresponding projected elements $A_L = \bm \Pi_L A \in \mathfrak B_L$ from \eqref{eqAL}. By the cyclic property of the trace, we have $\tr (\rho_L A_L) = \tr (\rho_L A)$, and the trace-norm convergence in \eqref{eqRhoLConv} implies $\lim_{L\to\infty} \tr(\rho_L A_L) = \tr(\rho A)$. Equivalently, letting $\omega_\rho \in \ssb$ and $\omega_{\rho_L}$ be the states of $\mathfrak B$ and $\mathfrak B_L$ induced by $\rho$ and $\rho_L$, respectively, we have
\begin{equation}
    \label{eqRhoALConv}
    \lim_{L\to\infty} \omega_{\rho_L} A_L = \omega_\rho A.
\end{equation}
We conclude that evaluation of the projected observables $A_L$ on the projected states $\omega_{\rho_L}$ asymptotically recovers the evaluation of $A$ on $\rho$.

\subsubsection{\label{appStateDataApprox}Data-driven state approximation}

Proceeding analogously to \cref{appOpDataApprox}, we seek data-driven approximations of projected density operators $\rho_L \in \mathfrak B_{L*}$ by density operators $\rho_{L,N} \in \mathfrak B_{L,N*}$ for a subset of density operators $\rho \in \mathfrak B_*$ that behave compatibly with bounded operators on continuous functions. 

First, we recall that every density operator $\rho \in \mathfrak B_*$ admits a decomposition (diagonalization) of the form 
\begin{equation}
    \label{eqRhoDecomp}
    \rho = \sum_{j=0}^\infty r_j \langle \xi_j, \cdot \rangle \xi_j, 
\end{equation}
where $\{ \xi_0, \xi_1, \ldots \} $ is an orthonormal basis of $H$, $(r_0, r_1,\ldots)$ is an $\ell^1$ sequence of real numbers in the interval $[0,1]$, and the sum over $j$ converges in the trace norm of $\mathfrak B_*$. In what follows, we shall restrict attention to a subset $\scb \subset \ssb $, consisting of all normal states  $\omega_\rho$ of $\mathfrak B$ whose corresponding density operators $\rho \in \mathfrak B_*$ are decomposable as in \eqref{eqRhoDecomp} with the following additional requirement: The orthonormal basis vectors $\xi_j$ have uniformly bounded continuous representatives; that is, we have 
\begin{displaymath}
    \xi_j = \iota \tilde \xi_j, \quad \tilde \xi_j \in C(X), \quad \lVert \xi_j \rVert_{C(X)} \leq b, 
\end{displaymath}
for a constant $b$. Given such an $\omega_\rho \in \scb $, for each $N\in \mathbb N$ we define the positive operator  $\hat \sigma_N : \hat H_N \to \hat H_N$, where
\begin{displaymath}
    \hat \sigma_N  = \sum_{j=0}^\infty r_j \langle \hat \xi_{j,N}, \cdot \rangle_N \hat \xi_{j,N}, \quad \hat \xi_{j,N} = \iota_N \tilde \xi_j.
\end{displaymath}
Note that the well-definition of $\hat \sigma_N$ follows from the uniform boundedness of the $\tilde \xi_j$ and the fact that $(r_0, r_1, \ldots)$ is an $\ell^1$ sequence. It should also be kept in mind that, in general, the $\sigma_N$ are not normalized as density operators. We then have:

\begin{lemma}
    \label{lemRho}\  
    \begin{enumerate}[label=(\alph*)]
        \item $\tilde \rho : f \mapsto g$ with $g(x) = \sum_{j=0}^\infty r_j \langle \xi_j, \iota f \rangle \tilde \xi_j(x) $ is well-defined as a linear map from $C(X)$ to itself, and it satisfies $\iota \circ \tilde \rho = \rho \circ \iota$.
        \item For $\mu$-a.e.\ initial condition $x_0 \in X$, the residual $R_N f= (\iota_N \circ \tilde \rho) f - (\hat \sigma_N \circ \iota_N) f$ satisfies 
            \begin{displaymath}
                \lim_{N\to\infty} \lVert R_N f \rVert_{\hat H_N} = 0, \quad \forall f \in C(X).
            \end{displaymath}
    \end{enumerate}
\end{lemma}

\begin{proof}
    See \cref{appProofRhoLemma}.
\end{proof}

It follows from \cref{lemRho} that \eqref{eqIota} and~\eqref{eqIotaN} hold with $A = \rho$, $\tilde A = \tilde \rho$, and $\hat A_N = \hat \sigma_N$. Thus, \cref{lemApprox} and \cref{corApprox} apply, and for each $L\in \mathbb N$ such that $\lambda_{L-1} > 0$, the matrix representations $\bm \sigma_{L,N} = [\langle \phi_{i,N}, \hat \sigma_N \phi_{j,N} \rangle_N]_{i,j=0}^{L-1}$ of $\sigma_{L,N} = \Pi_{L,N}\hat\sigma_N$ converge to the matrix representation $\bm \sigma_L = [\langle \phi_i, \rho \phi_j \rangle ]_{i,j=0}^{L-1}$ of $\sigma_L = \bm \Pi_L \rho$. If, in addition, $L$ is sufficiently large such that $C_L > 0 $, then the density operators $\rho_{L,N} \in \mathfrak B_{L,N*}$ defined as $\rho_{L,N} = \sigma_{L,N} / C_{L,N} $ with $C_{L,N} = \tr \sigma_{L,N}$ converge, as $N\to\infty$, in the sense of convergence of the corresponding matrix representations $\bm \rho_{L,N} = \bm \sigma_{L,N} / C_{L,N}$, to the density operator $\rho_L = \sigma_L / C_L $ with matrix representation $\bm \rho_L = \bm \sigma_L / C_L$. As we saw in \cref{appStateCompression}, the latter converges to $\rho$ as $L\to\infty$ in the trace norm. 

Combining the results of this section with those of \cref{appOpApprox}, we conclude that QMDA consistently approximates the action of normal states $\omega_\rho \in \scb$ on elements $A \in \mathfrak B$ satisfying \eqref{eqIota} and~\eqref{eqIotaN} by the action of the data-driven states $\omega_{\rho_{L,N}} \in S(\mathfrak B_{L,N})$ on the data-driven elements $A_{L,N} \in \mathfrak B_{L,N}$  in the sense of the iterated limit
\begin{equation}
    \label{eqRhoLALConv}
    \lim_{L\to\infty} \lim_{N\to\infty} \omega_{\rho_{L,N}} A_{L,N} = \lim_{L\to\infty} \omega_{\rho_L} A_L = \omega_\rho A, 
\end{equation}
where the first equality holds for $\mu$-a.e.\ initial condition $x_0 \in X$.

\subsubsection{\label{appProofRhoLemma}Proof of \cref{lemRho}}

\paragraph*{(a)} Fix $x \in X$ and $\epsilon > 0 $. For any $x' \in X$ and $J \in \mathbb N$, we have 
\begin{displaymath}
    \left\lvert \sum_{j=J}^\infty r_j \langle \xi_j, f \rangle \left(\tilde \xi_j(x) - \tilde \xi_j(x')\right) \right\rvert \leq \lVert f \rVert_{C(X)}\sum_{j=J}^\infty r_j \left\lvert \tilde\xi_j(x) - \tilde \xi_j(x') \right\rvert \leq 2 b \lVert f \rVert_{C(X)} \sum_{j=J}^\infty r_j.     
\end{displaymath}
Thus, since $(r_0, r_1, \ldots ) \in \ell^1$, there exists $J$ such that $\left\lvert \sum_{j=J}^\infty r_j \langle \xi_j, f \rangle (\tilde \xi_j(x) - \tilde \xi_j(x')) \right\rvert < \epsilon$, for all $ x' \in X$. We therefore have 
\begin{displaymath}
    \lvert g(x) - g(x') \rvert = \left\lvert \sum_{j=0}^\infty r_j \langle \xi_j, f \rangle \left(\tilde \xi_j(x) - \tilde \xi_j(x')\right) \right\rvert \leq \left\lvert \sum_{j=0}^{J-1} r_j \langle \xi_j, f \rangle \left(\tilde \xi_j(x) - \tilde \xi_j(x')\right) \right\rvert + \epsilon,
\end{displaymath}
and the continuity of $g$ follows from the fact that the first term in the right-hand side of the last inequality is a finite sum of continuous functions. The boundedness of $g$ can be shown similarly. The relation $\iota \circ \tilde \rho = \rho \circ \iota$ follows directly from the definitions of $\rho$ and $\tilde\rho$. 

\paragraph*{(b)} Given $f \in C(X)$, we have
\begin{align*}
    \lVert R_N f \rVert_{\hat H_N} & = \left \lVert (\iota_N \circ \tilde \rho) f - (\hat \sigma_N \circ \iota_N) f \right \rVert_{\hat H_N} = \left \lVert  \sum_{j=0}^\infty r_j \left( \langle \xi_j, \iota f \rangle - \langle \hat \xi_{j,N}, \iota_N f \rangle_N \right)\hat \xi_{j,N} \right \rVert_{\hat H_N} \\
    &\leq \sum_{j=0}^\infty r_j \left \lvert  \langle \xi_j, \iota f \rangle - \langle \hat \xi_{j,N}, \iota_N f \rangle_N \right\rvert b.
\end{align*}
Since $\left \lvert  \langle \xi_j, \iota f \rangle - \langle \hat \xi_{j,N}, \iota_N f \rangle_N \right\rvert \leq 2 b \lVert f \rVert_{C(X)}$ and $(r_0, r_1, \ldots) \in \ell^1$, for every $\epsilon >0 $ there exists $J \in \mathbb N$ such that 
\begin{displaymath}
    \sum_{j=J}^\infty r_j \left \lvert  \langle \xi_j, \iota f \rangle - \langle \hat \xi_{j,N}, \iota_N f \rangle_N \right\rvert b < \epsilon/2. 
\end{displaymath}
Moreover, by \eqref{eqWeakConv}, for $\mu$-a.e.\ $x_0 \in X$ and every $j \in \mathbb N$ we have 
\begin{displaymath}
    \lim_{N\to\infty}\left \lvert  \langle \xi_j, \iota f \rangle - \langle \hat \xi_{j,N}, \iota_N f \rangle_N \right\rvert = 0, 
\end{displaymath}
so there exists $N_* \in \mathbb N$ such that
\begin{displaymath}
    \sum_{j=0}^{J-1} r_j \left \lvert  \langle \xi_j, \iota f \rangle - \langle \hat \xi_{j,N}, \iota_N f \rangle_N \right\rvert < \epsilon /2, \quad \forall N > N_*. 
\end{displaymath}
We thus have 
\begin{displaymath}
    \lVert R_N f \rVert_{\hat H_N} \leq \sum_{j=0}^{J-1} r_j \left \lvert  \langle \xi_j, \iota f \rangle - \langle \hat \xi_{j,N}, \iota_N f \rangle_N \right\rvert + \sum_{j=J}^{\infty} r_j \left \lvert  \langle \xi_j, \iota f \rangle - \langle \hat \xi_{j,N}, \iota_N f \rangle_N \right\rvert <  \epsilon  
\end{displaymath}
for all $ N > N_*$, and since $\epsilon$ was arbitrary, we conclude that $\lim_{N\to\infty} \lVert R_N f \rVert_{\hat H_N} = 0$. This completes the proof of \cref{lemRho}.

\subsection{\label{appObs}Approximation of the forecast observable and its spectral measure}

In this subsection, we examine the QMDA representation of the forecast observable $f \in \mathfrak A$ by projected multiplication operators in $\mathfrak B_{L,N}$ which we denote, as in the main text, by $\pi_{L,N} f$. We are interested in two types of asymptotic consistency of our representations, respectively described in \cref{appPointwiseApprox} and \cref{appSpectralApprox}:
\begin{enumerate}
    \item Pointwise consistency, meaning that evaluation of $\pi_{L,N} f$ on the states $\omega_{\rho_{L,N}}$ from \cref{appState} should converge to evaluation of the multiplication operator $\pi f \in \mathfrak B$ on the state $\omega_{\rho}$ approximated by $\omega_{\rho_{L,N}}$.
    \item Spectral consistency, meaning that the spectral measures of $\pi_{L,N}f$ should converge to the spectral measure of $\pi f$ in a suitable sense. 
\end{enumerate}
In QMDA applications, pointwise consistency is required for consistency of the forecast mean and variance with the theoretical forecast mean and variance, respectively, from the infinite-dimensional data assimilation system based on the algebra $\mathfrak B$ (i.e., the quantum mechanical level \Lq\ in \cref{figSchematic}). Meanwhile, spectral consistency is required for consistency of the corresponding forecast probabilities (denoted as $\mathbb P_{f,t,\tau}$ in the main text). 

\subsubsection{\label{appPointwiseApprox}Pointwise approximation and its consistency}

For a given trajectory $X_N := \{x_0, \ldots, x_{N-1} \} \subset X$, let $\hat{\mathfrak A}_N$ denote the finite-dimensional, abelian von-Neumann algebra of complex-valued functions on $X_N$ with respect to pointwise function multiplication and complex conjugation, equipped with the maximum norm, $\lVert u \rVert_{\hat{\mathfrak A}_N} = \max_{x_n \in X_N} \lvert u(x_n) \rvert$. As a vector space, $\hat{\mathfrak A}_N$ is isomorphic to the Hilbert space $\hat H_N$, but the two spaces have different norms. Every function $ f:  X \to \mathbb C$ induces an element $\hat f_N \in \hat{\mathfrak A}_N$ by restriction to $X_N$, $\hat f_N(x_n) = f(x_n)$ for all $n\in \{0, \ldots,N-1 \}$. Reusing notation, we will denote the linear map that maps $f$ to $\hat f_N$ by $\iota_N$. Analogously to $\pi : \mathfrak A \to \mathfrak B$, the algebra $\hat{\mathfrak A}_N$ has a regular representation $\hat \pi_N : \hat{\mathfrak A}_N \to \hat{\mathfrak B}_N$ such that, given $u\in \hat{\mathfrak A}$, $\hat \pi_N u$ is the multiplication operator by $u$, i.e., $(\hat \pi_N u) v  = uv $ for all $v \in \hat{\mathfrak A}_N$. Moreover, similarly to $\pi_L : \mathfrak A \to \mathfrak B_L$, for each $L \in \{ 1, \ldots, N\}$ we define the linear map $\pi_{L,N} : \hat{\mathfrak A}_N \to \mathfrak B_{L,N}$, $\pi_{L,N} = \bm \Pi_{L,N} \circ \hat \pi_N$, which maps elements of $\hat{\mathfrak A}_N$ to projected multiplication operators in $\mathfrak B_{L,N}$. Note that, in general, neither $\pi_L$ nor $\pi_{L,N}$ are algebra homomorphisms.     

Consider now the $C^*$-algebra of continuous functions on $X$, $C(X)$, and its regular representation $\tilde \pi: C(X) \to \mathfrak C$, where $\tilde \pi \tilde f$ is the multiplication operator by $\tilde f$, i.e., $(\tilde \pi \tilde f) \tilde g = \tilde f \tilde g$ for all $ \tilde g \in C(X)$. One readily verifies that for every $ \tilde f \in C(X)$, 
\begin{equation}
    \label{eqPiF}
    \iota \circ (\tilde \pi \tilde f) = (\pi f) \circ \iota, \quad \iota_N \circ (\tilde \pi f) = (\hat \pi_N \hat f_N) \circ \iota_N,  
\end{equation}
where $f = \iota \tilde f \in \mathfrak A$ and $\hat f_N = \iota_N \tilde f \in \hat{\mathfrak A}_N$. As a result, \eqref{eqIota} and~\eqref{eqIotaN} hold for $ A = \pi f$, $\tilde A = \tilde \pi \tilde f$, and $\hat A_N = \hat \pi_N \hat f_N$, and by \cref{lemApprox} and \cref{corApprox} we can consistently approximate $\pi f$ by the projected multiplication operators $\pi_L f$ and $\pi_{L,N} \hat f_N$. 

In the main text, $\pi_L f$ and $\pi_{L,N}\hat f_N$ were used to represent the forecast observable $f$ in the matrix mechanical and data-driven formulations of QMDA, respectively. Under \cref{assump2}, \eqref{eqRhoLALConv} and~\eqref{eqPiF} together lead to the following consistency result for these representations,
\begin{displaymath}
    \lim_{L\to\infty}\lim_{N\to\infty} \omega_{\rho_{L,N}} (\pi_{L,N} \hat f_N) = \lim_{L\to\infty} \omega_{\rho_L} (\pi_L f) = \omega_\rho (\pi f),
\end{displaymath}
which holds for $\mu$-a.e.\ initial condition $x_0 \in X$.

As with other linear maps employed in QMDA, in numerical applications we employ the $L \times L$ matrix representation of $\pi_{L,N} \hat f_N$, given by $\bm A_{L,N} = [\langle \phi_{i,N}, (\pi_{L,N} \hat f_N ) \phi_{j,N} \rangle_N ]_{i,j=0}^{L-1}$. \cref{algObservable} describes the computation of this matrix (as well as the spectral measure of $\pi_{L,N} \hat f_N$, which we discuss in \cref{appSpectralApprox} below). By \cref{corApprox}, for $\mu$-a.e.\ initial condition $x_0$, $\bm A_{L,N}$ converges as $N\to\infty$ to the matrix representation $\bm A_L = [\langle \phi_i, \pi_L f \phi_j\rangle]_{i,j=0}^{L-1}$ of $\pi_L f$. 

\begin{algorithm}
    \caption{\label{algObservable}Projected multiplication operator representing the forecast observable $f$ and evaluation of the associated spectral measure. We suppress $L$ and $N$ subscripts from our notation of $\bm A_{L,N}$ and $\bm E_{L,N,}(S_m)$.}
    \Inputs
    \begin{enumerate}
        \item Training observable values $f_0, \ldots, f_{N-1} \in \mathbb R$.
        \item Basis vectors $\bm \phi_0, \ldots, \bm \phi_{L-1}$ from \cref{algBasis}.

        \item Intervals (spectral bins) $ S_0, \ldots, S_{M-1} \subseteq \mathbb R$.
    \end{enumerate}
    \Require The training data $z_n$ used in the computation of $\bm \phi_l$ are induced by the same dynamical states $x_n \in X$ underlying $f_n$, i.e., $z_n = z(x_n)$ and $f_n = f(x_n)$.  
    \medskip

    \Outputs
    \begin{enumerate}
        \item $L\times L$ matrix $\bm A$ representing the projected multiplication operator $\pi_{L,N} f $ in the $\phi_{l,N}$ basis of $H_{L,N}$. 
        \item $L\times L$ projection matrices $\bm E_0, \ldots, \bm E_{M-1}$, where $\bm E_m$ is the matrix representation of the spectral projector $E_{\pi_{L,N}\hat f_N}(S_m)$ in the $\{\phi_{l,N}\}$ basis of $H_{L,N}$.   
    \end{enumerate}

    \Steps
    \begin{enumerate}
        \item \Return $ \bm A = [A_{ij}]_{i,j=0}^{L-1} $, where $A_{ij} = \bm \phi_i^\top (\bm f \odot \bm \phi_j) /N  $, $ \bm f = ( f_0, \ldots, f_{N-1})^\top$, and $\odot$ denotes elementwise multiplication of column vectors. 
        \item Compute the eigendecomposition $\bm A \bm u_j = a_j \bm u_j$, where $a_0, \ldots, a_L \in \mathbb R$ and the eigenvectors $\bm u_0, \ldots, \bm u_{L-1} \in \mathbb R^L$ satisfy $ \bm u_i^\top \bm u_j = \delta_{ij} $. 
        \item \Return The projection matrices $\bm E_0, \ldots, \bm E_{M-1}$, where $\bm E_m = \sum_{a_j \in S_m} \bm u_j \bm u_j^\top $.
    \end{enumerate}
\end{algorithm}

\subsubsection{\label{appSpectralApprox}Spectral approximation}

We are interested in approximating the spectral measure of the multiplication operator $\pi f \in \mathfrak B$ associated with the forecast observable $f$ by the spectral measures of the finite-rank operators $ \pi_L f \in \mathfrak B_L$ and $ \pi_{L,N} \hat f_N \in \mathfrak B_{L,N} $. 

First, we recall that the spectrum of an element $a$ of a unital $C^*$-algebra is the set of complex numbers $z$ such that $a-z$ does not have an inverse. We denote this set as $\sigma(a)$. In the case of a finite-dimensional operator algebra such as $\mathfrak B_L$, and $\mathfrak B_{L,N}$, the spectrum of any element $a$ is a finite set consisting of the eigenvalues of $a$. In the case of the infinite-dimensional operator algebra $\mathfrak B$, the spectrum of a multiplication operator $\pi f$ by an element $f \in \mathfrak A$ is equal to the essential range of $f$, i.e., the support of the pushforward measure $f_*(\mu) : \mathcal B(\mathbb C) \to \mathbb C $ on the Borel $\sigma$-algebra on $\mathcal B(\mathbb C)$ on $\mathbb C$. Note that $\sigma(\pi f)$  coincides with the spectrum of $f$ as an element of the abelian algebra $\mathfrak A$, defined as the set of complex numbers for which $f - z$ is non-invertible.  

The point spectrum of $\pi f$, i.e., the set of its eigenvalues, consists of all elements $z\in \sigma(\pi f)$ such that the preimage $ S= f^{-1}(\{z\}) \subseteq X$ has positive measure, $\mu(S)>0$. We denote the point spectrum of $\pi f$ by $\sigma_p(\pi f)$. Points in the complement of $\sigma_p(\pi f)$ in $\sigma(\pi f)$ lie in the continuous spectrum of $\pi f$, and have no associated eigenspaces. A challenge with spectral approximation of multiplication operators on infinite-dimensional operator algebras is that generically they have a non-empty continuous spectrum, whereas the continuous spectrum of any finite-rank approximation of these operators is necessarily empty.     

Suppose now that $f \in \mathfrak A$ is real-valued as per \cref{assump1}(c), i.e., $f$ is a self-adjoint element of the abelian algebra $\mathfrak A$. Then, $\pi f \in \mathfrak B$ is a self-adjoint operator, and the spectrum $\sigma(\pi f) $ is a subset of the real line. By the spectral theorem for self-adjoint operators, there is a unique projection-valued measure (PVM) $E : \mathcal B(\mathbb R) \to \mathfrak B$, giving $\pi f$ through the spectral integral $\pi f = \int_{\mathbb R} u \, dE(u)$. By construction, the operators $\pi_L f \in \mathfrak B_L$ and $\pi_{L,N} \hat f_N$ are also self-adjoint whenever $f$ is real-valued, and thus have associated PVMs $E_L: \mathcal B(\mathbb R) \to \mathfrak B_L$ and $E_{L,N}: \mathcal B(\mathbb R) \to \mathfrak B_{L,N}$, respectively, such that $ \pi_L f = \int_{\mathbb R} u \, dE_L(u)$ and $\pi_{L,N} \hat f_N = \int_{\mathbb R} u \, dE_{L,N}(u)$. Since the spectra $\sigma(\pi_L f)$ and $\sigma(\pi_{L,N} f)$ are finite sets, these spectral measures are discrete. Explicitly, for a given Borel set $S \in \mathcal B(\mathbb R)$, we have 
\begin{displaymath}
    E_L(S) = \sum_{j: a_{L,j} \in S} E_{L,j}, \quad E_{L,N}(S) = \sum_{j: a_{L,N,j} \in S} E_{L,N,j}, 
\end{displaymath}
where $\{ a_{L,j} \} = \sigma(\pi_L f)$ and $\{ a_{L,N,j} \} = \sigma(\pi_{L,N} f)$ are the sets of eigenvalues of $\pi_L f$ and $\pi_{L,N} f$ (without multiplicities), and $E_{L,j} \in \mathfrak B_L$ and $E_{L,N,j} \in \mathfrak B_{L,N}$ are the orthogonal projections onto the corresponding eigenspaces, respectively.

In what follows, we characterize the convergence of $E_{L,N}$ to $E_L$ as $N\to\infty$ (large data limit) and $E_L$ to $E$ as $L\to\infty$ (infinite-dimension limit).  

\paragraph*{Spectral convergence in the large-data limit.} Let $\bm A_L = [\langle \phi_i, (\pi_L f ) \phi_j \rangle ]_{i,j=0}^{L-1}$ and $\bm A_{L,N} = [\langle \phi_{i,N}, (\pi_{L,N} \hat f_N ) \phi_{j,N} \rangle_N ]_{i,j=0}^{L-1}$ be the matrix representations of $\pi_L f$ and $\pi_{L,N}f$ in the $\{ \phi_l\}$ and $\{\phi_{l,N} \}$ bases of $H_L$ and $H_{L,N}$, respectively. In the same bases, the spectral measures $E_L$ and $E_{L,N}$ are represented by matrix-valued measures $\bm E_L : \mathcal B(\mathbb R) \to \mathbb M_L$ and  $\bm E_{L,N} : \mathcal B(\mathbb R) \to \mathbb M_L$ such that $\bm E_L(S) = [\langle \phi_i, E_L(S) \phi_j \rangle ]_{i,j=0}^{L-1}$ and $\bm E_{L,N}(S) = [\langle \phi_{i,N}, E_{L,N}(S) \phi_{j,N} \rangle_N]_{i,j=0}^{L-1} $ are the matrix representations of the projections $E_L(S)$ and $E_{L,N}(S)$ respectively. Since $\bm A_{L,N}$ converges to $\bm A_L$, it follows from spectral approximation results for finite-rank operators \cite{Chatelin11} that if the boundary of $S$ does not contain any eigenvalues of $\pi_L f$, then for $\mu$-a.e.\ initial condition $x_0$ and Borel set $S \in \mathcal B(\mathbb R)$, $ \bm E_{L,N}(S)$ converges as $N\to\infty$ to $\bm E_L(S)$. 

\cref{algObservable} describes the computation of the spectral projectors $\bm E_{L,N}(S_m)$ on a set of pairwise-disjoint intervals (``spectral bins'') $S_0, \ldots, S_{M-1} \subseteq \mathbb R$ partitioning the range of $\hat f_N$ in the training data. We choose the intervals $S_m$ such that they carry equal probability mass under the distribution of $\hat f_N$ with respect to the sampling measure $\mu_N$. With this choice, the boundaries of the $S_m$ can be computed from the values of the quantile function of $\hat f_N$ on a uniform partition of $[0,1]$; see \cref{algSpecBins} for further details.    

\begin{algorithm}
    \caption{\label{algSpecBins}Spectral bins for the forecast observable $f$ from the empirical quantile function.}
    \Inputs
    \begin{enumerate}
        \item Training observable values $f_0, \ldots, f_{N-1} \in \mathbb R$.
        \item Number of spectral bins $M \in \mathbb N$.
    \end{enumerate}
    \Outputs
    \begin{enumerate}
        \item Intervals (spectral bins) $S_0, \ldots, S_{M-1} \subseteq \mathbb R$.
    \end{enumerate}
    \Steps
    \begin{enumerate}
        \item Compute the empirical quantile function of $f$, $Q_f : (0, 1) \to (0, \infty) $, associated with the samples $f_n$.  
        \item Define $b_1, \ldots, b_{M-1} \in \mathbb R $ with $b_m = m / M$. 
        \item \Return The intervals $S_0, \ldots, S_{M-1}$, where 
            \begin{displaymath}
                S_m = 
                \begin{cases}
                    (-\infty, Q_f(b_1)], & m = 0, \\
                    (Q_f(b_m), Q_f(b_{m+1})], & 1 < m < M -1, \\
                    (Q_f(b_{M-1}), \infty), & m = M - 1.
                \end{cases}
            \end{displaymath}
    \end{enumerate}
\end{algorithm}

\paragraph*{Spectral convergence in the infinite-dimension limit} We employ the following results on spectral approximation of self-adjoint operators. 

\begin{theorem}
    \label{thmSpec}
    Let $A_L : H \to H$ be a sequence of finite-rank, self-adjoint operators on a Hilbert space $H$ converging strongly as $L\to\infty$ to a self-adjoint operator $A: H \to H $. Let $E: \mathcal B(\mathbb R) \to B(H)$ and $E_L : \mathcal B(\mathbb R) \to B(H)$ be the spectral measures of $A$ and $A_L$, respectively. Then, the following hold. 
    \begin{enumerate}[label=(\alph*)]
        \item For every element $a$ of the spectrum of $A$, there exists a sequence $a_L$ of eigenvalues of $A_L$ such that $\lim_{L\to\infty} a_L = a $. 
        \item For every Borel set $S \in \mathcal B(\mathbb R)$ such that $E(\partial S) = 0$ (i.e., the boundary of $S$ does not contain any eigenvalues of $A$), the spectral projections $E_L(S)$ converge to $E(S)$ in the strong operator topology of $B(H)$, i.e., 
            \begin{displaymath}
                \lim_{L\to\infty} E_L(S) f = E(S) f, \quad \forall f \in H.
            \end{displaymath}
    \end{enumerate}
\end{theorem}

\begin{proof}
    Strong convergence of bounded self-adjoint operators implies convergence in the strong resolvent sense; see ref.~\cite{Oliveira09}, Proposition~10.1.13(a). Strong resolvent convergence of operators implies in turn spectral convergence as stated in Part~(a); see ref.~\cite{Oliveira09}, Corollary 10.2.2. For Part~(b), see  ref.~\cite{DasEtAl21}, Proposition~13(iii), which states the analogous result under strong resolvent convergence of skew-adjoint operators. 
\end{proof}

Since the projected multiplication operators $\pi_L f$ converge to $ \pi f$ strongly and they are self-adjoint whenever $f$ is real-valued, it follows from \cref{thmSpec} with $A = \pi f$ and $A_L = \pi_L f $ that the spectra and spectral measures of $\pi_L f$ converge to those of $\pi f$ in the sense stated in the theorem. 

\subsection{\label{appKoopman}Koopman operator approximation}

Let $U \equiv U^{\Delta t}$ be the unitary Koopman operator on $H$ associated with the the temporal sampling interval $\Delta t$ of the data. Following refs.~\cite{BerryEtAl15,GiannakisEtAl15,Giannakis19,Giannakis19b,DasEtAl21}, we approximate $U$ by a shift operator $\hat U_N : \hat H_N \to \hat H_N$. Here, we define $\hat U_N$ as  
\begin{equation}
    \label{eqShiftOp}
    \hat U_N f(x_n) = 
    \begin{cases}
        f(x_{n+1}), & 0 \leq n \leq N- 2, \\
        f(x_0), & n = N-1.
    \end{cases}
\end{equation}
With this definition, $\hat U_N$ is a unitary operator acting as a left circular shift on the sequence of values $f(x_0), \ldots, f(x_{N-1})$.

Next, let $\tilde U : C(X) \to C(X)$ be the time-$\Delta t$ Koopman operator on continuous (as opposed to $L^2$) functions, defined as usual by composition with the time-$\Delta t $ flow, $\tilde U = f \circ \Phi^{\Delta t}$. Since $\Phi^{\Delta t} $ is $\mu$-preserving, we have
\begin{displaymath}
    U \circ \iota = \iota \circ \tilde U, 
\end{displaymath}
so \eqref{eqIota} is satisfied for $A = U$ and $\tilde A = \tilde U$. Moreover, one verifies that~\eqref{eqIotaN} is satisfied for $\hat A_N = \hat U_N$ and $\tilde A = \tilde U$ by observing that for every $f \in C(X)$ the vectors $ g = \hat U_N \circ \iota_N f$ and $g' = \iota_N \circ \tilde U f$ differ only in the $(N-1)$-th component; i.e., $g(x_n) = g'(x_n)$ for all $n \in \{ 0, \ldots, N-2\}$. This implies that the residual $R_N f = g - g'$ has norm 
\begin{displaymath}
    \lVert R_N f \rVert_{\hat H_N} \leq \lVert f \rVert_{C(X)} / \sqrt N, 
\end{displaymath}
which verifies~\eqref{eqIotaN}. It therefore follows that \cref{lemApprox} and \cref{corApprox} apply for $U$, $\hat U_N$, and $\tilde U$. Thus, for any $\mu$-a.e.\ initial condition $x_0$ and each $L \in \mathbb N$ such that $\lambda_{L-1} > 0$, the data-driven shift operator matrices, $\bm U_{L,N} = [ \langle \phi_{i,N}, \hat U_N \phi_{j,N} \rangle_N]_{i,j=0}^{L-1} $, converge as $N\to\infty$ to the projected time-$\Delta t $ Koopman operator matrices, $\bm U_L = [\langle \phi_i, U \phi_j \rangle]_{i,j=0}^{L-1}$. The projected Koopman operators $U_L$ converge in turn as $L\to \infty$ to $U$ in the strong operator topology of $\mathfrak B$, as described in \cref{appOpCompression}.  

To approximate the Koopman operator $U^t$ at time $ t = t_q := q\,\Delta t$ with $q \in \mathbb Z$, we repeat the construction described above using the $q$-th power of the shift operator, $\hat U^q_N $, as the approximating operator on $\hat H_N$, which is equivalent to a circular shift by $q$ steps. This leads to projected operators $U^{(q)}_{L,N} = \bm \Pi_{L,N} \hat U^q_N$ whose matrix representations $\bm U^{(q)}_{L,N}$ converge as $N\to \infty$ to the matrix representation of $U^{(t_q)}_L = \bm \Pi_L U^{t_q} $. As $L\to \infty$, $U^{(t_q)}_L$ converges strongly to $U^{t_q}$ (see \eqref{eqALConv}). Thus, we obtain an asymptotically consistent approximation of the dynamical operators employed in QMDA for any given (finite) time horizon $t_q = q\,\Delta t$. 

The construction of the $\bm U^{(q)}_{L,N}$ matrices is described in \cref{algKoopman}. It is important to note that acting with $\hat U_N^q$ on elements of $\hat H_N$ does not require explicit knowledge of the states $x_n$. It should also be kept in mind that, unless $H_{L,N}$ is a $\hat U_N$-invariant subspace of $\hat H_N$, $U^{(q)}_{L,N}$ is not a unitary operator, and it is not equal to the $q$-th power of $U_{L,N}$. Nevertheless, by unitarity of $\hat U^q_N$ we have $U^{(q)*}_{L,N} = U^{(-q)}_{L,N}$. Similarly, $U^{(t)}_L$ is in general not equal to $(U_L)^{t/\Delta t}$, it is not unitary, but it satisfies $U^{(t)*}_L = U^{(-t)}_L$.

\begin{algorithm}
    \caption{\label{algKoopman}Koopman operator approximation. We suppress $L$ and $N$ indices from our notation of $\bm U^{(q)}_{L,N}$.}
    \Inputs
    \begin{enumerate}
        \item Basis vectors $\bm \phi_0, \ldots, \bm \phi_{L-1}$ from \cref{algBasis}.

        \item Time-shift parameter $q \in \mathbb Z$.
    \end{enumerate}
    \Require The underlying training data $z_0, \ldots, z_{N-1}$ are time-ordered and are taken with a uniform sampling interval $\Delta t > 0$.
    \medskip

    \Outputs
    \begin{enumerate}
        \item $L\times L $ matrix $\bm U^{(q)}$ representing the projected Koopman operator on $U^{q\,\Delta t}_{L,N}$ on $H_{L,N}$.
    \end{enumerate}

    \Steps
    \begin{enumerate}
        \item For each $ l \in \{ 0, \ldots, L-1\}$, compute the time-shifted vectors $\bm \phi_i^{(q)} = (\phi_{0l}^{(q)}, \ldots, \phi_{N-1,l}^{(q)})^\top \in \mathbb R^N$ with
            \begin{displaymath}
                \phi_{nl} = \phi_{n'l}, \quad n' = n + l \mod N.  
            \end{displaymath}
        \item \Return The $L\times L$ matrix $\bm U^{(q)} = [ U^{(q)}_{ij} ]_{i,j=0}^{L-1}$ with $U_{ij} =\bm \phi_i^\top \bm \phi_j^{(q)}/N $.
    \end{enumerate}
\end{algorithm}
\begin{remark*}
    The shift operator in~\eqref{eqShiftOp} differs somewhat from the operators used in refs.~\cite{BerryEtAl15,GiannakisEtAl15,Giannakis19,Giannakis19b,DasEtAl21}, which employ the non-unitary definition $\check U_N : \hat H_N \to \hat H_N$ with 
    \begin{displaymath}
        \check U_N f(x_n) = 
        \begin{cases}
            f(x_{n+1}), & 0 \leq n \leq N- 2, \\
            0, & n = N-1.
        \end{cases}
    \end{displaymath}
    The difference between the two approaches is inconsequential in the large-data limit, i.e., both $\hat U_{N,L}$ and $\check U_{N,L} := \bm \Pi_L \check U_N$ exhibit the $N\to\infty$ convergence in \cref{lemApprox} and \cref{corApprox}. Here, we have opted to work with $\hat U_N$ from \eqref{eqShiftOp}, for, as we will see in \cref{appQuantumChannel} below, the unitarity of this operator ensures that the induced Koopman operator on $\hat{\mathfrak B}_N$ is a quantum channel.
\end{remark*}

\subsection{\label{appQuantumChannel}Finite-dimensional quantum operations}

In the main text, we introduced the projected Koopman operator $\mathcal U^{(t)}_L : \mathfrak B_L \to \mathfrak B_L$, defined as 
\begin{displaymath}
    \mathcal U^{(t)}_L A = U^{(t)}_L A U^{(t)*}_L = U^{(t)}_L A U^{(-t)}_L;
\end{displaymath}
equivalently, $ \mathcal U^{(t)}_L = \bm \Pi_L \circ \mathcal U^t $. Letting $\varpi : \mathfrak B \to \mathfrak B$ be the trivial representation, $ \varpi(A) = A $, we have $\mathcal U^{(t)}_L = \Pi_L \varpi(U^t) \Pi_L $, so by Stinespring's theorem (see \cref{secPositivityPreserving}) $\mathcal U^{(t)}_L$ is completely positive. We also have that $\mathcal U^{(t)}_L = (\mathcal P^{(t)}_{L*})^*$, where $\mathcal P^{(t)}_{L*} : \mathfrak B_{L*} \to \mathfrak B_{L*}$ is the transfer operator on the predual of $\mathfrak B_L$, defined as $\mathcal P^{(t)}_{L*} A = U^{(-t)}_L A U^{(t)}_L $. 

Next, let $ \rho \in \mathfrak B_{L*}$ be a density operator. This operator extends to a density operator $\varrho \in \mathfrak B_*$ given by $ \varrho = \Pi_L \rho \Pi_L$. One can then verify that the transfer operator $\mathcal P^{(t)}_L : \mathfrak B_{L*} \to \mathfrak B_{L*}$ satisfies 
\begin{equation}
    \label{eqTraceNoninc}
    (\mathcal P^{(t)}_L \omega_\rho) \equiv \omega_{\mathcal P^{(t)}_{L*}\rho}  \bm 1_L = \tr((\mathcal U^{-t} \varrho) \Pi_L) \leq \lVert \Pi_L \rVert_{\mathfrak B} = 1.
\end{equation}
By virtue of this fact and the complete positivity of $\mathcal U^{(t)}_L$ it follows that $\mathcal U^{(t)}_L$ is a quantum operation. 

\begin{remark*}
    If $H_L$ happens to be a Koopman-invariant subspace of $H$, i.e., $U^t H_L = H_L$, then  $\tr((\mathcal U^{-t} \varrho) \Pi_L) = \tr(\mathcal U^{-t} \varrho) = 1 $ and $\mathcal U^{(t)}_L$ is a quantum channel. This property holds if and only if $H_L$ is an orthogonal direct sum of Koopman eigenfunctions. The existence of such distinguished subspaces of $H$ cannot be assumed for general measure-preserving dynamical systems. For instance, it is a standard result from ergodic theory that if the dynamical flow $\Phi^t$ is measure-theoretically mixing, then the Koopman operator $U^t : H \to H $ has only constant eigenfunctions \cite{Walters81}. 
\end{remark*}

In the data-driven setting of $\mathfrak B_{L,N}$, we employ an analogous construction based on the unitary shift operator $\hat U_N^q$. The shift operator induces a unitary $\hat{\mathcal U}^q_N : \hat{\mathfrak B}_N \to \hat{\mathfrak B}_N$ that acts by conjugation by $\hat U^q_N$, i.e., $\hat{\mathcal U}^q_N A = \hat U^q_N A \hat U^{q*}_N$. This operator is a quantum channel analogously to the Koopman operator $\mathcal U^t : \mathfrak B \to \mathfrak B$. For $L \leq N-1$ we define the projected shift operator  $\mathcal U^{(q)}_{L,N} : \mathfrak B_{L,N} \to \mathfrak B_{L,N}$ such that $\mathcal U^{(q)}_{L,N} A = U^{(q)}_{L,N} A U^{(q)*}$. This operator has entirely analogous properties to the projected Koopman operator $\mathcal U^{(t)}_L$; that is, $\mathcal U^{(q)}_{L,N}$ is a completely positive map whose associated transfer operator $\mathcal P^{(q)}_{L,N} : \mathfrak B^*_{L,N} \to \mathfrak B^*_{L,N}$ with $\mathcal P^{(q)}_{L,N} = \mathcal U^{(q)*}_{L,N}$ is trace-non-increasing (i.e., satisfies an analog of \eqref{eqTraceNoninc}).  Thus, $\mathcal U^{(q)}_{L,N}$ is a quantum quantum operation.

\begin{algorithm}
    \caption{\label{algQMDAStep}QMDA forecast--analysis step.}

    \Inputs
    \begin{enumerate}
        \item Forecast timesteps $J_\text{f} \in \mathbb N$; observation timesteps $J_\text{o} \in \mathbb N$.
        \item Koopman matrices $\bm U^{(1)}, \ldots, \bm U^{(J)} \in \mathbb M_L$ from \cref{algKoopman} with $ J = \max\{ J_\text{f}, J_\text{o} \}$. 
        \item Spectral bins $S_0, \ldots, S_{M-1} \subseteq R $ from \cref{algSpecBins}.
        \item Forecast observable $\bm A \in \mathbb M_L $ and spectral projectors $\bm E_0, \ldots, \bm E_{M-1} \in \mathbb M_L $ from \cref{algObservable}. 
        \item Matrix-valued effect $ \bm F : Y \to \mathbb M_L$ from \cref{algEffect}. 
        \item Initial density matrix $\bm \rho \in \mathbb M_L$.   
        \item Observation $ y \in Y $ at time $ J_\text{o}\, \Delta t$. 
    \end{enumerate}
    \Require All $L \times L $ matrices are representations of operators in $\mathfrak B_{L,N}$ in the same data-driven basis $\{\phi_{l,N}\}_{l=0}^{L-1}$ from \cref{algBasis}. All training data are generated by the same sequence of (unknown) time-ordered states $x_0, \ldots, x_{N-1} \in X $ with $ x_n = \Phi^{n\,\Delta t}(x_0)$, taken at a fixed sampling interval $\Delta t > 0$.
    \medskip

    \Outputs
    \begin{enumerate}
        \item Mean forecast $\bar f_0, \ldots, \bar f_{J_\text{f}} \in \mathbb R$ at lead time $ \tau_0, \ldots, \tau_{J_\text{f}}$ with $ \tau_j = j \, \Delta t$. 
        \item Forecast uncertainty $\sigma_0, \ldots, \sigma_{J_\text{f}} \in \mathbb R$ at lead time $\tau_0, \ldots, \tau_{J_\text{f}}$.
        \item Forecast probability vectors $ \bm p_0, \ldots, \bm p_{J_f} \in \mathbb R^M$ with $ \bm p_j = (p_{0j}, \ldots, p_{Mj})$. $p_{mj}$ is the probability that, at lead time $\tau_j$, the forecast observable $f$ lies in spectral bin $S_m$.    
        \item Posterior density matrix $\bm\rho^{(+)} \in \mathbb M_L$ at time $J_\text{o} \, \Delta t$. 

    \end{enumerate}
    \Steps
    \begin{enumerate}
        \item Set $\bm U^{(0)} = \Id$. 
        \item For each $ j \in \{ 0, \ldots, J\}$ compute the time-evolved density matrix $\bm \rho_j = \bm \sigma_j / C_j $ with $\bm \sigma_j = (\bm U^{(j)})^\top \bm \rho \bm U^{(j)}$ and $C_j = \tr \bm \sigma_j$. 
        \item \Return The mean forecasts $ \bar f_j = \tr( \bm \rho_j \bm A)$ for $ j \in \{ 0, \ldots, J_\text{f} \} $. 
        \item \Return The forecast uncertainties $ \sigma_j =( \tr( \bm \rho_j \bm A^2) - \bar f_j^2 )^{1/2}$ for $ j \in \{ 0, \ldots, J_\text{f} \} $. 
        \item \Return The probability vectors $\bm p_j$ with $ p_{mj} = \tr(\bm \rho_j E_m)$ for $ m \in \{0, \ldots, M-1 \} $ and $ j \in \{ 0, \ldots, J_\text{f} \}$.
        \item Compute the effect matrix $\bm E_y = \bm F'(y)$. 
        \item \Return The posterior density matrix 
            \begin{displaymath}
                \bm \rho^{(+)} = \frac{\bm E_y \bm \rho_{J_\text{o}} \bm E_y}{\tr(\bm E_y \bm \rho_{J_\text{o}} \bm E_y)}.
            \end{displaymath}
    \end{enumerate}
\end{algorithm}

\cref{algQMDAStep} describes the QMDA forecast of $f:X \to \mathbb R$ via the quantum operation $\mathcal U^{(q)}_{N,L}$ applied to the corresponding projected multiplication operators $\pi_{L,N} \hat f_N$. The computational complexity of this algorithm with respect to the dimension $L$ is dominated by matrix--matrix multiplication of $L\times L$ matrices, and is thus $O(L^3)$. \cref{algQMDAStepPure} specializes the forecasting procedure to pure (vector) states, which allow representation of density matrices $\bm \rho \in \mathbb M_L $ by their corresponding state vectors, $\bm \rho = \bm\xi  \bm \xi^\dag$ where $\bm \xi$ is a unit vector in $\mathbb C^L$. This leads to a reduction in complexity to $O(L^2)$. More generally, the computational cost of forecasting with rank-$k$ density matrices is $O(kL^2)$. \cref{algQMDAStep,algQMDAStepPure} also include the analysis step based on effect-valued maps, which we describe in \cref{appEffect}.   

\begin{algorithm}
    \caption{\label{algQMDAStepPure}QMDA forecast--analysis step, specialized to pure states.}

    \Inputs
    \begin{enumerate}
        \item Forecast timesteps $J_\text{f} \in \mathbb N$; observation timesteps $J_\text{o} \in \mathbb N$.
        \item Koopman matrices $\bm U^{(1)}, \ldots, \bm U^{(J)} \in \mathbb M_L$ from \cref{algKoopman} with $ J = \max\{ J_\text{f}, J_\text{o} \}$. 
        \item Spectral bins $S_0, \ldots, S_{M-1} \subseteq R $ from \cref{algSpecBins}.
        \item Forecast observable $\bm A \in \mathbb M_L $ and spectral projectors $\bm E_0, \ldots, \bm E_{M-1} \in \mathbb M_L $ from \cref{algObservable}. 
        \item Matrix-valued effect $ \bm F : Y \to \mathbb M_L$ from \cref{algEffect}. 
        \item Initial state vector $\bm \xi \in \mathbb C^L$.   
        \item Observation $ y \in Y $ at time $ J_\text{o}\, \Delta t$. 
    \end{enumerate}
    \Require All $L \times L $ matrices are representations of operators in $\mathfrak B_{L,N}$ in the same data-driven basis $ \{ \phi_{l,N} \}$ from \cref{algBasis}. All training data are induced by the same sequence of (unknown) time-ordered states $x_0, \ldots, x_{N-1} \in X $ with $ x_n = \Phi^{n\,\Delta t}(x_0)$, taken at a fixed sampling interval $\Delta t > 0$.
    \medskip

    \Outputs
    \begin{enumerate}
        \item Mean forecast $\bar f_0, \ldots, \bar f_{J_\text{f}} \in \mathbb R$ at lead time $ \tau_0, \ldots, \tau_{J_\text{f}}$ with $ \tau_j = j \, \Delta t$. 
        \item Forecast uncertainty $\sigma_0, \ldots, \sigma_{J_\text{f}} \in \mathbb R$ at lead time $\tau_0, \ldots, \tau_{J_\text{f}}$.
        \item Forecast probability vectors $ \bm p_0, \ldots, \bm p_{J_f} \in \mathbb R^M$ with $ \bm p_j = (p_{0j}, \ldots, p_{Mj})$. $p_{mj}$ is the probability that, at lead time $\tau_j$, the forecast observable $f$ lies in spectral bin $S_m$.    
        \item Posterior state vector $\bm\xi^{(+)} \in \mathbb C^L $ at time $J_\text{o} \, \Delta t$. 
    \end{enumerate}
    \Steps
    \begin{enumerate}
        \item Set $\bm U^{(0)} = \Id$. 
        \item For each $ j \in \{ 0, \ldots, J\}$ compute the time-evolved state vector $\bm \xi_j = \bm u_j / \lVert \bm u_j \rVert_2$ with $ \bm u_j =  \bm (U^{(j)})^\top \bm \xi$. 
        \item \Return The mean forecasts $ \bar f_j = \bm\xi_j^\dag \bm A \bm \xi_j$ for $ j \in \{ 0, \ldots, J_\text{f} \} $. 
        \item \Return The forecast uncertainties $ \sigma_j =( \bm \xi_j^\dag \bm A^2 \bm \xi_j - \bar f_j^2 )^{1/2}$ for $ j \in \{ 0, \ldots, J_\text{f} \} $. 
        \item \Return The probability vectors $\bm p_j$ with $ p_{mj} = \bm \xi_j^\dag \bm E_m \bm \xi_j$ for $ m \in \{0, \ldots, M-1 \} $ and $ j \in \{ 0, \ldots, J_\text{f} \}$.
        \item Compute the effect matrix $\bm E_y = \bm F(y)$. 
        \item \Return The posterior state vector 
            \begin{displaymath}
            \bm \xi^{(+)} = \frac{\bm E_y \bm \xi_{J_\text{o}}}{\lVert \bm E_y \bm \xi_{J_\text{o}}\rVert_2}.
            \end{displaymath}
    \end{enumerate}
\end{algorithm}

\subsection{\label{appChannelConsistency}Channel consistency}

The asymptotic consistency of the quantum operations $\mathcal U^{(q)}_{L,N}$ readily follows from the results established in the previous subsections. 

First, given $A \in \mathfrak B$ and $A_L = \bm \Pi_L A$, it is a direct consequence of the strong convergence of $U^{(t)}_L$ to $U^t$ and of $A_L$ to $A$, together with the uniform boundedness of these operators, that $\mathcal U^{(t)}_L A_L $ converges to $\mathcal U^t A $ in the strong topology of $\mathfrak B$, i.e., 
\begin{equation}
    \label{eqUConv}
    \lim_{L\to\infty} (\mathcal U^{(t)}_L A_L) f = (\mathcal U^t A) f, \quad \forall f \in H. 
\end{equation}
As a result, for every normal state $\omega_\rho \in \ssb $ induced by a density operator $\rho \in \mathfrak B_*$, \eqref{eqRhoALConv} yields
\begin{equation}
    \label{eqPRhoALConv}
    \lim_{L\to\infty} (\mathcal P^{(t)}_L \omega_{\rho_L}) A_L = \lim_{L\to\infty} \omega_{\rho_L} (\mathcal U^{(t)}_L A_L) = \omega_\rho A,
\end{equation}
where $\rho_L = \bm \Pi_L \rho  / \tr(\bm \Pi_L \rho) \in \mathfrak B_L $ are the projected density operators induced by $\rho$. Equation~\eqref{eqPRhoALConv} establishes the consistency of the quantum operations $\mathcal U^{(t)}_L$ with the channel $\mathcal U^{(t)}$ in the infinite-dimension limit. 

Next, turning to the consistency of $\mathcal U^{(q)}_{L,N}$ in the large-data ($N\to\infty$) limit, we restrict attention to states in the subset $\scb \subset \ssb$ from \cref{appStateDataApprox} and elements $A \in \mathfrak B$ satisfying~\eqref{eqIota} and~\eqref{eqIotaN}. We also consider evolution times $t_q = q\,\Delta t$ with $q \in \mathbb N$. Under these assumptions, it follows from~\eqref{eqRhoLALConv} that for $\mu$-a.e.\ initial state $x_0$,
\begin{equation}
    \label{eqPRhoALNConv}
    \lim_{L\to\infty} \lim_{N\to\infty} (\mathcal P^{(q)}_{L,N} \omega_{\rho_{L,N}}) A_{L,N} = \lim_{L\to\infty} (\mathcal P^{(t_q)}_L \omega_{\rho_L}) A_L =  (\mathcal P^t\omega_\rho) A.
\end{equation}
Equation~\eqref{eqPRhoALNConv} holds, in particular, for data-driven, projected multiplication operators $A_{L,N} = \pi_{L,N} \hat f_N$ associated with continuous functions $f : X \to \mathbb R$ and pure states $\hat \rho_{N,L}$ induced by continuous state vectors, as in the QMDA experiments described in the main text.

\subsection{\label{appEffect}Effect-valued feature map}

In this subsection, we describe the construction and properties of the effect-valued feature map $\mathcal F : Y \to \mathcal E(\mathfrak B) $ and its finite-rank counterparts, $\mathcal F_L : Y \to \mathcal E(\mathfrak B_L)$ and $\mathcal F_{L,N} : Y \to \mathcal E(\mathfrak B_{L,N}) $, used in the analysis step of QMDA. As stated in the main text, we build these maps using a continuous, symmetric kernel on observations space $\psi : Y \times Y \to [0,1]$ and a family of data-dependent symmetric, continuous kernels $\psi_N : Y \times Y \to [0, 1]$ with $N \in \mathbb N$, such that, as $N\to\infty$, the pullback kernels $w_N : X \times X \to [0, 1 ]$ with $w_N(x,x') = \psi_N(h(x),h(x')) $ converge uniformly to $w: X \times X \to [0,1]$ with $ w(x,x') = \psi(h(x),h(x'))$. 

As a concrete example, in the L96 and CCSM4 experiments described in the main text we employ variable-bandwidth, bump kernels on $Y = \mathbb R^d$, 
\begin{equation}
    \label{eqKBump}
    \psi_N(y,y') = \eta_\text{bump}\left( \frac{d(y,y')}{\epsilon_* \sqrt{b_N(y)b_N(y')} } \right).
\end{equation}
In \eqref{eqKBump}, $\eta_\text{bump} : \mathbb R \to \mathbb R $ is the bump function
\begin{equation}      
    \label{eqEtaBump}
    \eta_\text{bump}(u) = 
    \begin{cases}
        e^{-1/(1-u^2)}, & u \in (-1, 1),\\
        0, & \text{otherwise},
    \end{cases}
\end{equation}
$d: Y \times Y \to \mathbb R_+$ is the Euclidean distance, $b_N : Y \to \mathbb R_+$ is a bandwidth function obtained by applying \cref{algBandwidth} to the training data $y_0, \ldots, y_{N-1} \in Y$, and $\epsilon_* $ is a positive bandwidth parameter determined via \cref{algTuning}. Note that $b_N$ and $\epsilon_*$ in \eqref{eqKBump} are different from the bandwidth function and bandwidth parameter used in \eqref{eqKVB} (the former are based on data $y_n \in Y$ whereas the latter are based on data $z_n \in Z$). 

The choice of kernel in \eqref{eqKBump} is motivated by the fact that the classical Bayesian analysis step of data assimilation can be modeled as an effect-valued map $\check F : Y \to \mathcal E(\mathfrak A) $ for the abelian algebra $ \mathfrak A = L^\infty(X,\mu)$, where $ \check F(y) = \chi_{Y^{-1}(y)} $ and $\chi_S : X \to \{ 0, 1 \}$ denotes the characteristic function of a set $S \subseteq X$. The effect-valued map $F_N : Y \to \mathcal E(\mathfrak A)$ induced by the kernel in \eqref{eqKBump}, $F_N(y) = \psi_N(y, h(\cdot))$, can be thought of as a smoothed version of $\check F $. In L96 and CCSM4 experiments not reported here, we found that using a fixed-bandwidth kernel in the analysis step (i.e., setting $b_N(y) = b_N(y') = 1$ in~\eqref{eqKBump} led to a noticeable reduction of forecast skill, particularly in the higher-dimensional CCSM4 case. We also ran experiments using the Gaussian shape function $\eta_\text{gauss}$ (as opposed to $\eta_\text{bump}$) in the definition of the kernel $\psi_N$, using either of the variable- or fixed-bandwidth versions. The results were generally comparable to those reported in the main text, though we found that the bump kernel did provide a modest amount of skill improvement over the Gaussian kernel.       

Let $F:Y \to \mathcal E(\mathfrak A)$ be the effect-valued map associated with the kernel $\psi$, 
\begin{displaymath}
    \psi(y,y') = \eta_\text{bump}\left( \frac{d(y,y')}{\epsilon \sqrt{b(y)b(y')} } \right),
\end{displaymath}
where the bandwidth function $b : Y \to \mathbb R_+ $ is the uniform limit of $b_N$ (note that we will not need this map in actual numerical applications). As described in the main text, in the setting of the non-abelian algebra $\mathfrak B$, we promote $F$ to an operator-valued map $\mathcal F : Y \to \mathcal E(\mathfrak B) $, where $\mathcal F(y) = \pi(F(y)) $. Moreover, we introduce projected and data-driven versions of these maps, given by $ \mathcal F_L : Y \to \mathcal E(\mathfrak B_L)$ and $\mathcal F_{L,N} : Y \to \mathcal E(\mathfrak B_{L,N})$, respectively, where $\mathcal F_L = \pi_L \circ \mathcal F $, $\mathcal F_{L,N} = \pi_{L,N} \circ \hat{\mathcal F}_N$, and $\hat F_N : Y \to \mathcal E(\hat{\mathfrak A}_N)$ is given by restriction of $F_N$ on the set of training states $X_N$, $\hat F_N(y) =\iota_N(F_N(y))$. State conditioning (analysis) based on these maps has the following consistency properties.

\begin{proposition}
    \label{propAnalysis}
    Let $\omega_\rho $ be a state of $\mathfrak B$ in $\scb$ and $y$ an observation in $Y$. Let $\bm \rho_{L,N}^{(+)}$ and $\bm \rho_L^{(+)}$ be the matrix representations of the conditional states $\omega_{\rho_{L,N}} \rvert_{\mathcal F_{L,N}(y)} \in S(\mathfrak B_{L,N} )$ and $ \omega_{\rho_L} \rvert_{\mathcal F_L(y)}$ obtained via~\eqref{eqQBayes} in the $\{ \phi_{l,N} \}$ and $ \{ \phi_l \}$ bases of $H_{L,N}$ and $H_L$, respectively. Then the following hold.
    \begin{enumerate}[label=(\alph*)]
        \item For $\mu$-a.e.\ initial state $x_0 \in X$, $\lim_{N\to\infty} \bm \rho_{L,N}^{(+)} = \bm \rho_L^{(+)}$.  
        \item As $L\to\infty$, $\omega_{\rho_L} \rvert_{\mathcal F_L(y)}$ converges to $\omega_\rho\rvert_{\mathcal F(y)}$ in the trace norm topology of $\ssb$.
    \end{enumerate}
\end{proposition}

\begin{proof}
    We use an auxiliary map $\tilde{\mathcal F} : Y \to \mathfrak C$, defined as $\tilde{\mathcal F}(y) = \tilde \pi(F^{1/2}(y))$. For each $y \in Y$, we have  
\begin{displaymath}
    \iota \circ (\tilde{\mathcal F}(y)) = (\mathcal F(y)) \circ \iota
\end{displaymath}
by construction. Moreover, by our assumed uniform convergence of $w_N$ to $w$, we have that for every $ f \in C(X)$, the residual 
\begin{displaymath}
    R_N f  = \left( \iota_N \circ (\tilde{\mathcal F}(y)) \right) f - \left (\hat{\mathcal F}_N(y) \circ \iota_N \right) f
\end{displaymath}
has vanishing $\hat H_N$ norm as $N\to\infty$, for $\mu$-a.e.\ initial condition $x_0$. Thus, \eqref{eqIota} and~\eqref{eqIotaN} hold for $A = \mathcal F(y)$, $\tilde A = \tilde{\mathcal F}(y)$, and $A_N = \hat{\mathcal F}_N(y)$, and correspondingly \eqref{eqRhoLConv}, \cref{lemApprox}, and \cref{corApprox} also hold. The claims of the proposition follow.
\end{proof}

Computationally, a drawback of using $\mathcal F_{L,N}$ for state conditioning is that evaluation of \eqref{eqQBayes} requires the square root  $\sqrt{\mathcal F_{L,N}(y)}$. Specifically, we have $\omega_{\rho_{L,N}}\rvert_{\mathcal F_{L,N}(y)} = \omega_{\rho_{L,N}^{(+)}}$, where
\begin{equation}
    \label{eqDensityUpdate}
     \rho_{L,N}^{(+)} = \frac{\sqrt{\mathcal F_{L,N}(y)} \rho_{L,N} \sqrt{\mathcal F_{L,N}(y)}}{\tr(\sqrt{\mathcal F_{L,N}(y)} \rho_{L,N} \sqrt{\mathcal F_{L,N}(y)})},
\end{equation}
and computing $\sqrt{\mathcal F_{L,N}(y)}$ requires computing the square root of the matrix $\bm F_{L,N}(y) = [\langle \phi_{i,N}, \mathcal F_{L,N}(y) \phi_{j,N} \rangle_N]_{i,j=0}^{L-1}$ representing $\mathcal F_{L,N}(y)$ in the $\{\phi_{l,N}\}$ basis of $H_{L,N}$. To avoid having to perform this expensive operation at every observational update, in applications we replace $\sqrt{\mathcal F_{L,N}(y)}$ by $\mathcal F'_{L,N}(y)$, where $\mathcal F'_{L,N} : Y \to \mathcal E(\mathfrak B_{L,N})$ is the effect-valued map defined as
\begin{displaymath}
    \mathcal F'_{L,N}(y) = \pi_{L,N}(\hat F^{1/2}_N(y)).
\end{displaymath}
The construction of this map is described in \cref{algEffect}. Using $\mathcal F'_{L,N}$, the update of a density operator $\rho_{L,N} \in \mathfrak B_{L,N*}$ given an observation $y\in Y$ becomes (cf.~\eqref{eqDensityUpdate})
\begin{equation}
    \label{eqDensityUpdate2}
    \rho_{L,N}^{(+)} = \frac{\mathcal F'_{L,N}(y) \bm \rho_{L,N}\mathcal F'_{L,N}(y)}{\tr(\mathcal F'_{L,N}(y) \bm \rho_{L,N}\mathcal F'_{L,N}(y))}.
\end{equation}
In the limit of $L\to\infty$ after $N\to\infty$, updating via \eqref{eqDensityUpdate2} consistently recovers $\omega_\rho\rvert_{\mathcal F(y)}$ analogously to Proposition~\ref{propAnalysis}. See step~7 of \cref{algQMDAStep} for the matrix function that implements \eqref{eqDensityUpdate2} in the $ \{ \phi_{l,N} \}$ basis of $H_{L,N}$. Step~7 of \cref{algQMDAStepPure} specializes the state update procedure to pure states. 

\begin{algorithm}
    \caption{\label{algEffect}Effect-valued feature map.}
    \Inputs
    \begin{enumerate}
        \item Kernel function $\kappa : Y \times Y \to [0, 1]$.
        \item Training data $y_0, \ldots, y_{N-1} \in Y$.
        \item Basis vectors $\bm \phi_0, \ldots, \bm \phi_{L-1}$ from \cref{algBasis}.
    \end{enumerate}
    \Require The training data $z_n$ used in the computation of $\bm \phi_l$ are generated by the same dynamical states $x_n \in X$ underlying $f_n$, i.e., $z_n = z(x_n)$ and $f_n = f(x_n)$.  
    \medskip

    \Outputs
    \begin{enumerate}
        \item Matrix-valued map $\bm F'_{L,N} : Y \to \mathbb M_L $ representing the effect-valued function $\mathcal F'_{L,N}: Y \to \mathcal E(\mathfrak B_{L,N})$.
    \end{enumerate}
    \Steps
    \begin{enumerate}
        \item Construct the feature map $ \bm f'_N : Y \to \mathbb R^N$ where $\bm f'_N(y) = ( \psi_N^{1/2}(y,y_0), \ldots, \psi_N^{1/2}(y,y_{N-1}))^\top$.  
        \item \Return The function $\bm F'_{L,N} : Y \to \mathbb M_L$, where $ \bm F_{L,N}(y) = \bm E = [E_{ij}]_{i,j=0}^{L-1} $ and $E_{ij} = \bm \phi_i^\top ( \bm f'_N(y) \odot \bm \phi_j) / N $. 
    \end{enumerate}
\end{algorithm}

\section{\label{appMetrics}Forecast skill metrics}

We assess the skill of the QMDA forecasts in the main text using normalized root mean square error (NRMSE) and anomaly correlation (AC) scores. Using throughout the notation of the main text and \cref{algQMDA}, we perform forecasts of $f$ at lead times $\tau_j = j \, \Delta t$ with $ j \in \{ 0, \ldots, J_\text{f} -1 \}$ given initial data $\hat y_0, \ldots, \hat y_{\hat N-1} \in Y$. We let $\bar f_{n,j} \in \mathbb R$ be the mean forecast at lead time $\tau_j$ initialized with data $\hat y_n$. Given the values $\hat f_0, \ldots, \hat f_{\hat N + J_\text{f}-1} \in \mathbb R$ of $f$ in the verification interval, the error of the forecast mean $\bar f_{n,j}$ relative to the true value of $f$ is $\varepsilon_{n,j} = \bar f_{n,j} - \hat f_{n+j}$. 

Let $\mean_N f $ and $\var_N f$ be the empirical mean and variance of $f$ computed from the training data, 
\begin{displaymath}
    \mean_N f = \int_X f \, d\mu_N = \frac{1}{N} \sum_{n=0}^{N-1} f_n, \quad \var_N f = \int_X ( f - \mean_N f)^2 \, d\mu_N = \frac{1}{N} \sum_{n=0}^{N-1} ( f_n - \mean_N f )^2.
\end{displaymath}
We define the NRMSE and AC scores for lead time $\tau_j$ as 
\begin{align*}
    \NRMSE(\tau_j) &= \sqrt{\frac{1}{\hat N \var_N f}\sum_{n=0}^{\hat N-1} \varepsilon_{n,j}^2 }, \\
    \AC(\tau_j) &= \frac{1}{\hat N \var_N f} \sum_{n=0}^{\hat N-1} (\bar f_{n,j} - \mean_N f)(\hat f_{n+j} - \mean_N f),
\end{align*}
respectively. 

NRMSE values close to 0 and AC values close to 1 indicate high forecast skill. NRMSE values approaching 1 indicate loss of skill as the expected forecast error is comparable to the standard deviation of the forecast observable. In climate dynamics applications, such as ENSO forecasting, $\AC = 0.6$ or $\AC = 0.5$ are commonly used thresholds indicating loss of skill. 

\section{\label{appDatasets}Dataset description}

In this section, we describe the properties of the L96 multiscale and CCSM4 datasets used in the experiments presented in the main text. 

\subsection{\label{appL96}L96 multiscale} 

We integrate the L96 multiscale system in \eqref{eqL96} in MATLAB using the built-in stiff solver |ode15s|, sampling the numerical trajectory every $\Delta t = 0.05 $ model time units. We note that the use of a stiff solver is important for numerical accuracy due to the timescale separation between the $\lx_k$ and $\ly_{j,k}$ variables occurring at small $\varepsilon$. As noted in the main text, our training and test data are sampled on independent dynamical trajectories. The initial conditions for the trajectory $x_0, \ldots, x_{N-1} \in \mathbb R^{J(K+1)}$ underlying the training data are $ (\lx_1, \ldots, \lx_K) = (1, 0, \ldots, 0 ) \in \mathbb R^K $ and $(\ly_{1,k}, \ldots, \ly_{J,k}) = (1, 0, \ldots, 0 ) \in \mathbb R^J $ for each $k \in \{ 1, \ldots, K \}$. Similarly, we initialize the test trajectory $\hat x_0, \ldots, \hat x_{\hat N-1} \in \mathbb R^{J(K+1)} $ at  $ (\lx_1, \ldots, \lx_K) = (1.2, 0, \ldots, 0 ) \in \mathbb R^K $ and $(\ly_{1,k}, \ldots, \ly_{J,k}) = (1.2, 0, \ldots, 0 ) \in \mathbb R^J $ for each $k \in \{ 1, \ldots, K \}$. Starting from these initial conditions, we let the two trajectories equilibrate to the attractor over a time interval of $10^4 \, \Delta t = 500 $ model time units before collecting the first samples, $x_0$ and $\hat x_0$.  

\subsection{\label{appCCSM4}CCSM4}

We sample data every $\Delta t = 1$ month, spanning a 1,300-year period from a control integration of CCSM4 forced with fixed pre-industrial concentrations of greenhouse gases \cite{CCSM410}. Following ref.~\cite{WangEtAl20}, the observation map $h : X \to Y \equiv \mathbb R^d$ returns the monthly averaged sea surface temperature (SST) field on the model’s native ocean grid (of approximately 1\degr nominal resolution) over the Indo-Pacific longitude-latitude box 28\degr E--70\degr W, 30\degr S--20\degr N. The number of gridpoints within this domain (which corresponds to the observation space dimension) is $d= \text{44,414}$. As the forecast observable $f: X\to \mathbb R$, we use the model's Ni\~no~3.4 index---this is defined as the area-averaged SST anomaly over the domain 170\degr W--120\degr W, 5\degr S--5\degr N relative to a monthly climatology computed over the training period. Specifically, let $\tilde h : X \to \mathbb R^{\tilde d}$ denote the observable representing the SST field over the Ni\~no~3.4 region. For each $ m \in \{ 1, \ldots, 12\}$, define the monthly climatology $ \bar h^{(m)} \in \mathbb R^{\tilde d}$ as 
\begin{displaymath}
    \bar h^{(m)} = \frac{1}{N_\text{y}} \sum_{\substack{0 \leq n \leq N-1\\ \mnth(n) = m}} \tilde h(x_n),
\end{displaymath}
where $\mnth(n) := (n \mod 12) + 1$ is the calendar month associated with the $n$-th sample in the training data.

\section*{Acknowledgments}
We thank Travis Russell for helpful discussions and feedback on this work. 

\bibliographystyle{siamplain}

\end{document}